\documentclass[preprint,3p,10pt]{elsarticle}

\biboptions{numbers,sort&compress}

\usepackage{amsmath,amssymb,amsthm,bm}
\usepackage{breqn}

\usepackage{pdflscape}
\usepackage{graphicx}
\usepackage{subfigure}
\usepackage{float}
\usepackage{multirow}
\usepackage{makecell}
\usepackage{booktabs}
\usepackage{tikz}
\usepackage{ytableau}
\usepackage{xcolor}
\usepackage{enumerate}
\usepackage[hidelinks]{hyperref}
\usepackage[capitalize,nameinlink]{cleveref}

\newcommand{\Z}{\mathbb Z}
\newcommand{\R}{\mathbb R}

\renewcommand{\Re}{\operatorname{Re}}
\renewcommand{\Im}{\operatorname{Im}}
\renewcommand{\i}{\mathrm{i}}

\DeclareMathOperator{\Wr}{Wr}

\newcounter{subeq}

\theoremstyle{definition}

\newtheorem{definition}{Definition}[section]
\newtheorem{remark}{Remark}[section]

\newtheorem{lemma}{Lemma}[section]
\newtheorem{proposition}[definition]{Proposition}
\newtheorem{theorem}{Theorem}[section]
\newtheorem{conjecture}{Conjecture}
\newtheorem{example}{Example}[section]
\newtheorem{corollary}[definition]{Corollary}

\journal{xxx}

\begin{document}
\small
\begin{frontmatter}

\title{Zeros of the generalized Wronskian--Hermite polynomials}

\author[szu_ias,szu_sms]{Chengfa Wu\corref{cor1}}
\ead{cfwu@szu.edu.cn}
\author[szu_ias]{Guangxiong Zhang}

\cortext[cor1]{Corresponding author.}
\affiliation[szu_ias]{organization={Institute for Advanced Study, Shenzhen University},
            city={Shenzhen},
            postcode={518060},
            country={People's Republic of China}}
\affiliation[szu_sms]{organization={School of Mathematical Sciences, Shenzhen University},
            city={Shenzhen},
            postcode={518060},
            country={People's Republic of China}}

\begin{abstract}
In this paper, we study generalized Wronskian--Hermite (WH) polynomials associated with arithmetic-progression index sets.  We verify the conjecture that all nonzero roots are simple for three subclasses of these polynomials, which, in a natural sense, cover more than half of the relevant parameter range. We also provide an interpretation of the root multiplicity at \(z=0\) in terms of Young diagrams. In addition, we show that certain members of generalized WH polynomials provide representations of rational solutions of
the Noumi--Yamada systems, a family of higher-order Painlev\'e equations.
Finally, we apply our results to the large-parameter
asymptotic analysis of rogue wave patterns for the multi-component Hirota
equation.
\end{abstract}


\end{frontmatter}

\setcounter{tocdepth}{2}
\tableofcontents

\section{Introduction}

For an integer \(m\ge2\), define the polynomials \( p_n^{[m]}(z) \) by the generating function
\begin{equation}\label{eq:p-generating-intro} \sum_{n=0}^{\infty}p_n^{[m]}(z)\epsilon^n = \exp(z\epsilon+\epsilon^m).
\end{equation} 
We adopt the convention that \(p_n^{[m]}(z)\equiv0\) for \(n<0\). These polynomials
are normalized versions of the Gould--Hopper polynomials \cite{GouldHopper1962}, which generalize the classical Hermite polynomials and have been extensively studied \cite{feigin2021quasi,vignat2013proof,chang2011gould,romdhane2008zeros}. When $m=2$, they reduce,  after a rescaling, to the classical Hermite polynomials.

Wronskians of Hermite polynomials constitute a classical and extensively studied family of special
polynomials. They arise in the theory of monodromy-free Schr\"odinger operators \cite{oblomkov1999monodromy}
and have been investigated from many perspectives, including their zeros, recurrence relations, coefficient formulae, and irreducibility properties \cite{felder2026harmonic,felder2012zeros,bonneux2018recurrence,
bonneux2020coefficients,grosu2021irreducibility,gomez2021complete}. It is therefore natural to form Wronskian determinants from the generalized Hermite polynomials defined in \eqref{eq:p-generating-intro} \cite{zhang2022rogue,yang2023rogue,bonneux2020coefficients,bonneux2020asymptotic,grosu2021expansion}. This gives rise to the
corresponding generalized Wronskian--Hermite polynomials. In this paper,
we focus on the subfamily whose index sets are arithmetic progressions.

For integers \(N\ge1\), \(k\ge2\), and \(1\le l\le k-1\), set
\begin{equation}\label{eq:nodes-intro}
    n_i=l+(i-1)k,\qquad i=1,\ldots,N,
\end{equation}
and define the generalized Wronskian--Hermite polynomial
with jump \(k\) by
\begin{equation}\label{eqn:Wronskian-Hermite-poly-with-jump-k}
    W_N^{[m,k,l]}(z)
    =
    c_N^{[m,k,l]}
    \Wr \!\left[
        p_{n_1}^{[m]}(z),p_{n_2}^{[m]}(z),\ldots,p_{n_N}^{[m]}(z)
    \right],
    \qquad
    c_N^{[m,k,l]}
    =
    \frac{\prod_{i=1}^{N}n_i!}
         {\prod_{1\le i<j\le N}(n_j-n_i)}.
\end{equation}
For \(N=0\), we set $W_0^{[m,k,l]}(z)=1.$ For brevity, we refer to \(W_N^{[m,k,l]}(z)\) simply as a generalized
Wronskian--Hermite polynomial.
For \(m=2\), the generalized Wronskian--Hermite polynomial \(W_N^{[m,k,l]}(z)\) reduces to
Wronskian--Hermite polynomials with arithmetic index sets.  For suitable choices of \(m,k,l\), it also recovers, up to rescalings, several
polynomial hierarchies related to the Painlev\'e equations. In particular, for $s \geq 1$,
\(W_N^{[2s+1,2,1]}(z)\) is related to the Yablonskii--Vorob'ev hierarchy, which
can be used to express rational solutions of the second Painlev\'e equation and its hierarchy, and has been widely investigated
\cite{clarkson2003second,fukutani2000special,yang2023rogue,bertola2024exactly,balogh2016hankel,bertola2015zeros,BuckinghamMiller2014PII,BuckinghamMiller2015PIICritical}. When \(k=3\), \(W_N^{[m,3,1]}(z)\) and \(W_N^{[m,3,2]}(z)\) are, respectively, constant multiples
of the generalized Okamoto polynomial hierarchies \(Q_N^{[m]}(z)\) and
\(R_N^{[m]}(z)\). The classical Okamoto polynomials are associated
with rational solutions of the fourth Painlev\'e equation, which have been extensively studied
\cite{Okamoto1986StudiesPainleveIII,kametaka1983poles,noumi1999symmetries,clarkson2003fourth,yang2023rogue,roffelsen2025real}.
The generalized Wronskian–Hermite polynomials also appear in the asymptotic description of
rogue wave patterns for integrable equations, including the nonlinear
Schr\"odinger (NLS) equation, the derivative NLS equation, and their
multi-component generalizations
\cite{yang2021rogue,yang2023rogue,zhang2022rogue,lin2024rogue}.

Among many problems concerning generalized
Wronskian--Hermite polynomials, a particularly important one concerns the
simplicity of their nonzero roots. This problem arose naturally in our
previous studies of rogue wave patterns for integrable equations and led us
to formulate the following conjecture
\cite{yang2023rogue,zhang2022rogue}.

\begin{conjecture}\label{conjecture:nonzero simple of WH}
All nonzero roots of the generalized Wronskian--Hermite polynomial \(W_N^{[m,k,l]}(z)\) are simple for 
\begin{equation}
    N\ge1,\qquad m\ge2,\qquad k\ge2,\qquad 1\le l\le k-1.
\end{equation}
\end{conjecture}

In the 1990s, Veselov conjectured that all nonzero roots of the Wronskian--Hermite polynomials ($m=2$) associated with arbitrary index sets are simple; this conjecture was later recorded in \cite{felder2012zeros}, where the zero distributions of these polynomials were related to Young diagrams.
In contrast to the classical Hermite case,
the analogous statement for generalized Wronskian--Hermite polynomials
fails for arbitrary index sets (see Remark \ref{remark:conjecture not true roe arbitrary m}). This explains why this conjecture is formulated only for the arithmetic-progression index sets \eqref{eq:nodes-intro}.

To the best of our knowledge, only a few special cases of
Conjecture~\ref{conjecture:nonzero simple of WH} have been established. When
\(N=1\), \(W_1^{[m,k,l]}(z)\) is a nonzero constant multiple of the corresponding generalized Hermite polynomial, and Conjecture \ref{conjecture:nonzero simple of WH} follows from the results on
\(d\)-symmetric \(d\)-orthogonal polynomials \cite{romdhane2008zeros}. For
\(N=2\) and \(m=2\), it was proved that the
nonzero roots of Wronskians of two Hermite polynomials are simple
\cite{garcia2015oscillation}. Two Painlev\'{e}-related subfamilies provide further known cases. In particular, the case \(k=2\) and \(m=3\) coincides, up to an
elementary rescaling, with the Yablonskii--Vorob'ev polynomials, all of whose roots are
simple \cite{fukutani2000special}. Similarly, the case \(k=3\) and \(m=2\)
reduces to the classical Okamoto polynomials, for which root simplicity was
proved in \cite{kametaka1983poles,fukutani2000special}.

Despite these partial results,
Conjecture~\ref{conjecture:nonzero simple of WH} remains open in general.
The main result of this paper is to establish the conjecture for three classes of generalized Wronskian--Hermite polynomials.

\begin{theorem}\label{thm:main-intro}
Let \(n_i=l+(i-1)k\), \(i=1,\ldots,N\), and let
\(W_N^{[m,k,l]}(z)\) be defined by \eqref{eqn:Wronskian-Hermite-poly-with-jump-k}.
Then every nonzero root of \(W_N^{[m,k,l]}(z)\) is simple in the following cases:
\begin{enumerate}[(i)]
    \item \(N=2\) and \(m>n_1\);
    \item \(N\ge3\) and \(m>n_N/2\);
    \item \(k=4\) and \(m=2\).
\end{enumerate}
\end{theorem}

\begin{remark}\label{rem:more-than-half}
Theorem~\ref{thm:main-intro} covers, in a natural counting sense, more than
half of the relevant parameter range. Indeed, for a fixed
arithmetic-progression index set, it suffices to consider
\(2\le m\le n_N\), since \(W_N^{[m,k,l]}(z)\) is a monomial whenever \(m>n_N\) (see Lemma \ref{lem:monomial-cases}). When \(N=2\), case~(i) covers \(k=n_2-n_1\) of the \(n_2-1\)
integers in this range. Since \(n_1=l\le k-1\), we have
\[
    k>\frac{n_2-1}{2}.
\]
When \(N\ge3\), case~(ii) covers
\[
    \left\lceil\frac{n_N}{2}\right\rceil >
    \frac{n_N-1}{2}
\]
of the \(n_N-1\) integers in this range.
\end{remark}

The elementary cases (see Lemma \ref{lem:monomial-cases}) in which \(W_N^{[m,k,l]}(z)\) is a monomial are immediate, so we focus on the nontrivial cases. 
We briefly outline the proof of Theorem~\ref{thm:main-intro}. For the first case,  the introduction of the Gould--Hopper polynomials reduces the problem
to showing that a certain linear combination of two strictly hyperbolic
polynomials with strictly interlacing zeros is itself strictly hyperbolic.
The required hyperbolicity and interlacing properties of the two polynomials follow from the theory of \(d\)-symmetric \(d\)-orthogonal
polynomials \cite{romdhane2008zeros}, while the strict hyperbolicity of their
linear combination is then guaranteed by Obreschkoff's theorem
\cite{obreschkoff1963verteilung,borcea2009polya}.
In the second (nontrivial) case, a two-term determinant reduction expresses
the factor containing all nonzero roots of
\(W_N^{[m,k,l]}(z)\) in the form
\[
\det(z^m I_r+M),
\]
where $1 \leq r \leq N-1$ and \(M\) is a diagonal scaling of a Cauchy-type matrix. The desired zero distribution then follows from the
Gantmacher--Krein spectral theory and properties of strictly sign-regular
matrices
\cite{GantmacherKrein2002,AlseidiMargaliotGarloff2019}.
In the third case,
where \(m=2\) and \(k=4\), we derive explicit recurrence relations and
bilinear equations satisfied by these polynomials. These identities provide
the inductive mechanism required to prove the conjecture for all \(N\) in this
case.

The rest of this paper is organized as follows.
Section~\ref{sec:preliminary} collects some preliminary results used in the
proof of Theorem~\ref{thm:main-intro}, while Section~\ref{sec:proof of main results} proves Theorem \ref{thm:main-intro}.
More precisely, Section~\ref{sec:proof of main results} is divided into three
subsections, corresponding respectively to the three cases of the theorem. For the first two cases, the zero distributions are also
investigated. In Section~\ref{sec:WH with partition}, we further study the
multiplicities of $z=0$ as a root of the generalized Wronskian--Hermite polynomials and
give an interpretation of this multiplicity from the viewpoint of partitions.
Section~\ref{sec:WH with Painleve} is devoted to showing that \(W_N^{[2,k,l]}(z)\) can be used to represent rational solutions of higher-order Painlev\'e equations introduced by Noumi and Yamada \cite{noumi1998affine,noumi1999symmetries}. Finally, Section~\ref{sec:WH with application} applies our results to the large-parameter asymptotic analysis of rogue wave patterns for the
multi-component Hirota equation, in which the root-simplicity results established in this paper play an essential role.

\section{Notations and preliminaries} \label{sec:preliminary}
\subsection{Notations}
Throughout this paper, \(\mathbb Z\), \(\mathbb R\), and
\(\mathbb C\) denote the sets of integers, real numbers, and complex
 numbers, respectively. 
For \(a,r\in\mathbb Z\) and \(k\in\mathbb Z_{>0}\), we use the notations
\begin{eqnarray*}
&&    \mathbb Z_{<a}
    :=
    \{q\in\mathbb Z:q<a\}, \quad \mathbb Z_{>a}
    :=
    \{q\in\mathbb Z:q>a\},
\\
&&
\mathbb Z_{\leq a}
    :=
    \{q\in\mathbb Z:q\leq a\}, \quad \mathbb Z_{\geq a}
    :=
    \{q\in\mathbb Z:q \geq a\},
    \qquad
    k\mathbb Z_{<a}+r
    :=
    \{kq+r:q\in\mathbb Z,\ q<a\}.
\end{eqnarray*}
For a finite set \(A\), we denote its cardinality by \(\#A\).
For \(x\in\mathbb R\), \(\lfloor x\rfloor\) denotes the greatest
integer less than or equal to \(x\), and
\(\operatorname{sgn}(x)\) denotes the sign of \(x\).
For a finite vector
\(\mathbf a=(a_1,\ldots,a_n)\), we write
\begin{equation}
    |\mathbf a|
    :=
    \sum_{j=1}^{n}a_j.
\end{equation}
When the ambient dimension is clear, \(\mathbf 0\), \(\mathbf 1\),
and \(\mathbf e_j\) denote, respectively, the zero vector, the
all-ones vector, and the \(j\)-th standard basis vector. Moreover,
\(I_n\) denotes the \(n\times n\) identity matrix.

For sufficiently differentiable functions
\(f_1,\ldots,f_N\) of \(z\), their Wronskian is defined by
\begin{equation*}
    \operatorname{Wr}[f_1,\ldots,f_N](z)
    :=
    \det_{1\le i,j\le N}\left(
        \frac{d^{\,i-1}f_j(z)}{dz^{\,i-1}}
    \right).
\end{equation*}
We adopt the conventions
\begin{equation*}
    \operatorname{Wr}[\varnothing]=1,
    \qquad
    \sum_{j\in\varnothing}a_j=0,
    \qquad
    \prod_{j\in\varnothing}a_j=1.
\end{equation*}
For sufficiently differentiable functions \(f\) and \(g\),
Hirota's bilinear operator is defined by
\cite{hirota2004direct}
\begin{equation}
D_x^m D_t^n f\cdot g
=
\left.
\left(
\frac{\partial}{\partial x}
-
\frac{\partial}{\partial x'}
\right)^m
\left(
\frac{\partial}{\partial t}
-
\frac{\partial}{\partial t'}
\right)^n
\left[
f(x,t)g(x',t')
\right]
\right|_{x'=x,\;t'=t}.
\end{equation}

\subsection{Hyperbolic polynomials and Descartes' rule of signs}

\begin{definition}\label{def:hyperbolic}
A nonzero polynomial $P(z) \in \mathbb{R}[z]$ is said to be \emph{hyperbolic} if all of its roots are real. Furthermore, it is called \emph{strictly hyperbolic} if all of its roots are real and simple.
\end{definition}

\begin{definition}\label{def:proper-position}
For two hyperbolic polynomials $f(z), g(z) \in \mathbb{R}[z]$, we say that $f$ and $g$ are in \emph{proper position} and write $f \ll g$, if their roots interlace and their Wronskian satisfies
\begin{equation}
    \operatorname{Wr}[f,g](z) = f(z)g'(z) - f'(z)g(z) \ge 0 \qquad \text{for all } z \in \mathbb{R}.
\end{equation}
We also note the convention that the roots of the zero polynomial interlace the
roots of any nonzero hyperbolic polynomial, and write \(0\ll f\) and \(f\ll0\).
\end{definition}

\begin{theorem}[Obreschkoff's theorem, {\cite{obreschkoff1963verteilung,borcea2009polya}}]\label{thm:obreschkoff}
Let \(f(z),g(z)\in\R[z]\).  Then 
\(\alpha f(z)+\beta g(z)\) for all \(\alpha,\beta\in\R\) is hyperbolic or identically zero
if and only if either \(f\ll g\), or \(g\ll f\), or \(f=g\equiv0\).
Moreover, \(\alpha f(z)+\beta g(z)\) is strictly hyperbolic for all \(\alpha,\beta\in\R\) with  \(\alpha^2+\beta^2\not=0\) if and
only if \(f(z)\) and \(g(z)\) are strictly hyperbolic with no common roots and
either \(f\ll g\) or \(g\ll f\).
\end{theorem}
For a multivariate extension of this theorem, see  \cite[Theorem~2.9]{borcea2009lee}.

\begin{theorem}[Descartes' rule of signs]
\label{thm:descartes-rule-of-signs}
Let $P(z)\in\mathbb R[z]\setminus\{0\}$ be 
\begin{equation}
    P(z)=a_Nz^N+a_{N-1}z^{N-1}+\cdots+a_0, \qquad a_N\neq0.
\end{equation}
Let \(V(P)\) be the number of sign variations in the coefficient sequence $a_N,a_{N-1},\ldots,a_0,$ after all zero entries are removed. Then the number of real positive roots of
\(P(z)\), counted with multiplicity, is at most \(V(P)\). Moreover, the number of real negative roots of \(P(z)\), counted with multiplicity, is at most
\(V(P(-z))\).
\end{theorem}

\subsection{Multiple orthogonal polynomials}

Multiple orthogonal polynomials are a generalization of orthogonal polynomials and satisfy orthogonality conditions with respect to several measures
\cite{kuijlaars2010multiple,aptekarev1998multiple,VanAsscheCoussement2001}. They originate from Hermite--Pad\'e
approximation and appear in the study of number theory, random matrix theory and integrable systems \cite{bleher2011random,aptekarev1998multiple,kuijlaars2010multiple}.
There are two types of multiple orthogonal polynomials, called type I
and type II. In this paper, we are concerned with those of type~II. Let \(w_1,\ldots,w_d\) be a finite collection of weight
functions on \(\mathbb R\), and let
\[
    \mathbf n=(n_1,\ldots,n_d)\in\mathbb Z_{\geq 0}^d,
    \qquad
    |\mathbf n|=n_1+\cdots+n_d.
\]
A type II multiple orthogonal polynomial associated with these data is a monic polynomial \(P_{\mathbf n}\) of
degree \(|\mathbf n|\) satisfying the orthogonality conditions
\begin{equation}\label{eq:type-II-MOP}
    \int_{-\infty}^{\infty}
    P_{\mathbf n}(x)x^j w_k(x)\,\mathrm{d}x=0,
    \qquad
    j=0,\ldots,n_k-1,\quad k=1,\ldots,d.
\end{equation}
When all the multi-indices are taken along the so-called stepline, the
corresponding type II multiple orthogonal polynomials form a
\(d\)-orthogonal polynomial sequence, also known as a sequence of vector
orthogonal polynomials
\cite{VanIseghem1987,filipuk2015computing,sokal2024multiple}.

\begin{definition}[\cite{VanIseghem1987,Maroni1989}]
Let \(u_0,u_1,\ldots,u_{d-1}\) be \(d\) linear functionals, and let
\(\{P_n\}_{n\geq 0}\) be a polynomial set over $\mathbb{C}$. We say that
\(\{P_n\}_{n\geq 0}\) is a \(d\)-orthogonal polynomial set (\(d\)-OPS) with respect to the $d$-dimensional functional
\begin{equation}
    \mathcal U={}^{t}(u_0,u_1,\ldots,u_{d-1}),
\end{equation}
if, for each \(j=0,1,\ldots,d-1\), it fulfills
\begin{equation}
\begin{cases}
    \langle u_j, P_m P_n\rangle=0,
    \qquad \text{if} \quad m>nd+j,\quad n\geq 0,\\
    \langle u_j, P_n P_{nd+j}\rangle\neq 0,
    \qquad n\geq 0.
\end{cases} 
\end{equation}
The case \(d=1\) is precisely the classical orthogonality condition.
\end{definition}

\begin{definition}
    A polynomial set \(\{P_n\}_{n\geq 0}\) over $\mathbb{C}$ is called \(d\)-symmetric if
\begin{equation}
    P_n(\omega_d z)=\omega_d^n P_n(z), 
    \qquad
    \omega_d=\exp\left(\frac{2\pi \mathrm{i}}{d+1}\right).
\end{equation}
\end{definition}

\begin{lemma}[{\cite{DouakMaroni1992,romdhane2008zeros}}] \label{lem:d-symmetric-recurrence} 
Let \(\{P_n\}_{n\geq 0}\) be a \(d\)-orthogonal polynomial set. Then the following statements are equivalent: 
\begin{enumerate} 
\item[(a)] \(\{P_n\}_{n\geq 0}\) is \(d\)-symmetric; 
\item[(b)] there exists a sequence \(\{\gamma_{n+1}\}_{n\geq 0}\)  with \(\gamma_{n+1}\neq 0\), such that 
\begin{equation} \label{eqn:d-sym-d-othognal,recurrence relations}
\begin{cases} P_{n+d+1}(z)=z P_{n+d}(z)-\gamma_{n+1}P_n(z), \qquad n\geq 0,\\ 
P_n(z)=z^n, \qquad 0\leq n\leq d. 
\end{cases} 
\end{equation} 
\end{enumerate} 
\end{lemma}

\begin{theorem}[{\cite[Theorem~2.2]{romdhane2008zeros}}]
\label{thm:ben-romdhane-zero-theorem}
Let \(\{P_n\}_{n\geq 0}\) be a \(d\)-orthogonal polynomial set satisfying \eqref{eqn:d-sym-d-othognal,recurrence relations} with $\gamma_{n+1}>0$.
Then, for each \(n\geq 1\) and \(j=0,1,\ldots,d\), the following statements hold.
\begin{itemize}
    \item[(a)] If \(z\) is a root of \(P_{n(d+1)+j}(z)\), then
    \(\omega_d^k z\), \(k=1,\ldots,d\), are also zeros of
    \(P_{n(d+1)+j}(z)\), where $\omega_d=\exp\left(2\pi\mathrm{i}/(d+1)\right)$.

    \item[(b)] The origin is a root of \(P_{n(d+1)+j}(z)\) of multiplicity \(j\).

    \item[(c)] The polynomial \(P_{n(d+1)+j}(z)\) has exactly \(n\) distinct positive zeros. We denote them by
    \begin{equation}
        0<x_{n,j;1}<x_{n,j;2}<\cdots<x_{n,j;n}.
    \end{equation}

    \item[(d)] The positive roots of
    \(P_{n(d+1)+j}(z)\), \(j=0,1,\ldots,d\), are interlaced as follows:
    \begin{equation}
    \label{eq:BR-same-level-interlacing}
        0
        <
        x_{n,0;1}
        <
        x_{n,1;1}
        <
        \cdots
        <
        x_{n,d;1}
        <
        x_{n,0;2}
        <
        \cdots
        <
        x_{n,d;2}
        <
        \cdots
        <
        x_{n,0;n}
        <
        \cdots
        <
        x_{n,d;n}.
    \end{equation}
\end{itemize}
\end{theorem}
\begin{remark}
In the proof of Theorem \ref{thm:ben-romdhane-zero-theorem} (see \cite{romdhane2008zeros}), it was shown that the positive roots of \(P_{n(d+1)+j}(z)\) satisfy
\begin{equation} 
x_{n,d;s}<x_{n+1,0;s+1}, \qquad s=1,\ldots,n, \label{eq:BR-root-relation-1} 
\end{equation} 
\begin{equation} 
x_{n+1,j;s}<x_{n,j;s}, \qquad j=0,\ldots,d,\quad s=1,\ldots,n, \label{eq:BR-root-relation-2} 
\end{equation} 
and 
\begin{equation} 
0<x_{n,j;s}<x_{n,j+1;s}, \qquad j=0,\ldots,d-1,\quad s=1,\ldots,n. 
\label{eq:BR-root-relation-3} 
\end{equation}
\end{remark}

\begin{corollary}\label{cor:fixed-level-interlacing}
Under the assumptions of Theorem~\ref{thm:ben-romdhane-zero-theorem}, fix
\(j\; (0\le j\le d)\). Let
\begin{equation}
    0<x_{n,j;1}<x_{n,j;2}<\cdots<x_{n,j;n}
\end{equation}
be the positive roots of \(P_{n(d+1)+j}(z)\). Then, for \(n\geq1\), we have
\begin{equation}
    0<
    x_{n+1,j;1}
    <
    x_{n,j;1}
    <
    x_{n+1,j;2}
    <
    x_{n,j;2}
    <
    \cdots
    <
    x_{n,j;n}
    <
    x_{n+1,j;n+1}.
\end{equation}
\end{corollary}

\begin{proof}
We first consider the case \(j=0\). By
\eqref{eq:BR-root-relation-2},
\begin{equation}
    x_{n+1,0;s}<x_{n,0;s},
    \qquad s=1,\ldots,n.
\end{equation}
On the other hand, repeated use of \eqref{eq:BR-root-relation-3} gives
\begin{equation}
    x_{n,0;s}<x_{n,d;s},
\end{equation}
and then \eqref{eq:BR-root-relation-1} yields
\begin{equation}
    x_{n,d;s}<x_{n+1,0;s+1}.
\end{equation}
Hence, we have
\begin{equation}
    x_{n+1,0;s}<x_{n,0;s}<x_{n+1,0;s+1},
    \qquad s=1,\ldots,n.
\end{equation}

Next, let \(1\leq j\leq d-1 \; (d\ge2)\).  Again,
\eqref{eq:BR-root-relation-2} gives
\begin{equation}
    x_{n+1,j;s}<x_{n,j;s},
    \qquad s=1,\ldots,n.
\end{equation}
For the other inequality, repeated use of \eqref{eq:BR-root-relation-3} from
level \(j\) to level \(d\), then \eqref{eq:BR-root-relation-1}, and finally
repeated use of \eqref{eq:BR-root-relation-3} with \(n\) replaced by \(n+1\)
give
\begin{equation}
    x_{n,j;s}
    <
    x_{n,d;s}
    <
    x_{n+1,0;s+1}
    <
    x_{n+1,j;s+1}.
\end{equation}
Therefore
\begin{equation}
    x_{n+1,j;s}<x_{n,j;s}<x_{n+1,j;s+1},
    \qquad s=1,\ldots,n.
\end{equation}

Finally, let \(j=d\). By \eqref{eq:BR-root-relation-2},
\begin{equation}
    x_{n+1,d;s}<x_{n,d;s},
    \qquad s=1,\ldots,n.
\end{equation}
Moreover, by \eqref{eq:BR-root-relation-1} and repeated use of
\eqref{eq:BR-root-relation-3} with \(n\) replaced by \(n+1\),
\begin{equation}
    x_{n,d;s}
    <
    x_{n+1,0;s+1}
    <
    x_{n+1,d;s+1}.
\end{equation}
Thus
\begin{equation}
    x_{n+1,d;s}<x_{n,d;s}<x_{n+1,d;s+1},
    \qquad s=1,\ldots,n.
\end{equation}
\end{proof}

\begin{corollary}\label{thm:defined-polynomials-roots}
Let \(\{P_n\}_{n\geq0}\) be a \(d\)-symmetric \(d\)-orthogonal polynomial set satisfying the assumptions of
Theorem~\ref{thm:ben-romdhane-zero-theorem}.
For \(j=0,1,\ldots,d\), define the polynomials
\(\{P_n^{(j)}\}_{n\geq0}\) by
\begin{equation}
    P_n^{(j)}(X)=x^{-j}P_{n(d+1)+j}(x),
    \qquad X=x^{d+1}.
\end{equation}
Then, for \(n\geq1\), the polynomial \(P_n^{(j)}(X)\) has \(n\) positive
simple roots. Moreover, if
\begin{equation}
    0<\xi_{n,j;1}<\xi_{n,j;2}<\cdots<\xi_{n,j;n}
\end{equation}
are the roots of \(P_n^{(j)}(X)\), then the roots of \(P_n^{(j)}(X)\) and
\(P_{n+1}^{(j)}(X)\) interlace:
\begin{equation}
    0<
    \xi_{n+1,j;1}
    <
    \xi_{n,j;1}
    <
    \xi_{n+1,j;2}
    <
    \xi_{n,j;2}
    <
    \cdots
    <
    \xi_{n,j;n}
    <
    \xi_{n+1,j;n+1}.
\end{equation}
\end{corollary}

\subsection{Strictly sign-regular matrices}

\begin{definition}
Let \(A\) be an \(n\times n\) real matrix, that is,
\(A\in\mathbb{R}^{n\times n}\). For \(1\le k\le n\), we say that \(A\) is
sign-regular of order \(k\), denoted by \(SR_k\), if all \(k\times k\)
minors of \(A\) are either non-negative or non-positive. It is called strictly sign-regular of order \(k\), denoted by \(SSR_k\), if all
\(k\times k\) minors of \(A\) are nonzero and have the same sign.
The matrix \(A\) is called strictly sign-regular, denoted by \(SSR\), if it is
\(SSR_k\) for all \(k=1,\ldots,n\).
\end{definition}

We recall the classical spectral theorem for strictly
sign-regular matrices.

\begin{theorem}[{\cite[Section 3.3]{AlseidiMargaliotGarloff2019}; see also \cite[Chapter V]{GantmacherKrein2002}}]\label{thm:GK}
If \(A\in\mathbb{R}^{n\times n}\) is \(SSR\), then
all eigenvalues of \(A\) are real, nonzero, and algebraically simple. Moreover,
they can be ordered as
\begin{equation}
    |\lambda_1|>|\lambda_2|>\cdots>|\lambda_n|>0.
\end{equation}
\end{theorem}

We shall also use a weaker, order-wise version of this result. Let \(A\) be
\(SSR_k\), and denote by \(\varepsilon_k\in\{1,-1\}\) the common sign of all
\(k\times k\) minors of \(A\). The eigenvalues are labeled so that
\begin{equation}
    |\lambda_1|\ge |\lambda_2|\ge \cdots \ge |\lambda_n|.
\end{equation}

\begin{theorem}[\cite{AlseidiMargaliotGarloff2019}]\label{thm:SSRk}
Let \(A\in\mathbb{R}^{n\times n}\) be nonsingular. Suppose \(A\) is
\(SSR_k\) for some \(k\in\{1,\ldots,n-1\}\). Then
\(\lambda_1\lambda_2\cdots\lambda_k\) is real and satisfies
\begin{equation}
    \varepsilon_k\lambda_1\lambda_2\cdots\lambda_k>0.
\end{equation}
Moreover,  the eigenvalues satisfy
\begin{equation}
    |\lambda_k|>|\lambda_{k+1}|.
\end{equation}
\end{theorem}

As an immediate consequence, if two consecutive strictly sign-regular conditions hold, then one obtains the following corollary.

\begin{corollary}[\cite{AlseidiMargaliotGarloff2019}]\label{cor:SSR-consecutive}
Let \(A\in\mathbb{R}^{n\times n}\) be nonsingular. Assume \(A\) is
\(SSR_i\) and \(SSR_{i+1}\) for some \(i\in\{1,\ldots,n-2\}\). Let
\(\varepsilon_i\) and \(\varepsilon_{i+1}\) be the common signs of the
\(i\times i\) and \((i+1)\times(i+1)\) minors of \(A\), respectively. Then, we have:
\begin{enumerate}
    \item[(a)] The signed eigenvalue \(\varepsilon_i\varepsilon_{i+1}\lambda_{i+1}\) is real and positive.
    \item[(b)] The eigenvalues satisfy the inequalities
    \begin{equation}
    |\lambda_i|>|\lambda_{i+1}|>|\lambda_{i+2}|.
\end{equation}
\end{enumerate}

\end{corollary}

\section{Proof of Theorem \ref{thm:main-intro}}\label{sec:proof of main results}

We first present some elementary monomial cases where the root structure is immediate. The remaining parts of this section are devoted to the proof of Theorem \ref{thm:main-intro} for other cases.
\begin{lemma}\label{lem:monomial-cases}
If either \(m>n_N\) or \(m\equiv0\pmod{k}\), then the generalized Wronskian--Hermite polynomial $W_N^{[m,k,l]}(z)$ given in \eqref{eqn:Wronskian-Hermite-poly-with-jump-k} has the form:
\begin{equation}
    W_N^{[m,k,l]}(z)=z^\Gamma,
\end{equation}
where
\begin{equation} \label{def:Gamma}
    \Gamma
    =
    \frac{N}{2}\bigl((k-1)(N-1)+2l\bigr).
\end{equation}
\end{lemma}

\begin{remark}\label{remark:conjecture not true roe arbitrary m}
Denote the generalized Wronskian--Hermite polynomials associated with arbitrary index set by
\begin{equation} \label{Wronskian-Hermite hierarchy}
    W_{\Lambda}^{[m]}(z) = \Wr \!\left[
        p_{n_1}^{[m]}(z),p_{n_2}^{[m]}(z),\ldots,p_{n_N}^{[m]}(z)
    \right],
\end{equation}
where $\Lambda=(n_1,n_2,\cdots,n_N) \in \Z^N$ with $0< n_1<n_2<\cdots<n_N$. For \(m=2\), Veselov conjectured that, for any index $\Lambda$, all nonzero roots of \(W_{\Lambda}^{[2]}(z)\) are
simple; this conjecture was explicitly stated in
\cite{felder2012zeros}. However, this conjecture does not hold for $m \ge 3$. For instance, 
\begin{eqnarray}
    && W_{(3,4)}^{[3]}(z)=\frac{1}{144} \left(z^3-12\right)^2,\\
       && W_{(2,3,6,10)}^{[5]}(z) = \frac{1}{11664000}\left(z^5-180\right)^3,\\
       && W_{(2,5,6,9)}^{[5]}(z) = \frac{z}{62208000} \left(z^5+120\right) \left(z^5+720\right)^2,\\
       &&  W_{(1,4,6,8)}^{[3]}(z) = \frac{z}{414720} \left(z^6-60 z^3-720\right)\left(z^3-12\right)^2.
\end{eqnarray}
In this paper, Conjecture \ref{conjecture:nonzero simple of WH} is restricted to the generalized Wronskian--Hermite polynomials $W_N^{[m,k,l]}(z)$, which correspond to $\Lambda = (l,l+k,l+2k, \dots, l+(N-1)k)$. 
\end{remark}

\subsection{The second-order case}
We first consider the second-order Wronskians of the Gould--Hopper polynomials \cite{GouldHopper1962}, which are defined by the generating function 
\begin{equation}\label{eq:GH-generating}
    \sum_{n=0}^{\infty}\mathcal H_n^{[m]}(z)\frac{t^n}{n!}
    =
    \exp\left(zt-c_m t^m\right),
    \qquad
    c_m=\frac{1}{(m-1)!m^2}, \qquad m\ge2.
\end{equation}
Alternatively, they can be defined by the recurrence relation \cite{GouldHopper1962}
\begin{equation}
    \mathcal H_{n+m}^{[m]}(z)
    =
    z\mathcal H_{n+m-1}^{[m]}(z)
    -
    \frac{1}{m}\binom{n+m-1}{m-1}\mathcal H_n^{[m]}(z),
    \qquad n\geq0,
\end{equation}
along with the initial conditions
\begin{equation}
    \mathcal H_n^{[m]}(z)=z^n,
    \qquad 0\leq n\leq m-1.
\end{equation} 
According to \cite{douak1995relation}, they constitute an $(m-1)$-symmetric $(m-1)$-orthogonal polynomial set characterized by positive recurrence coefficients \cite{romdhane2008zeros}. 
Hence, by Theorem~\ref{thm:ben-romdhane-zero-theorem}, for each
\(q\ge 0\) and \(0\le r\le m-1\), all nonzero roots of
\(\mathcal H_{qm+r}^{[m]}(z)\) are simple, and \(z=0\) is a root of multiplicity \(r\).

\begin{proposition}\label{prop:GH-Wronskian}
Let \(m\ge2\) and \(1\le l<m\).  For each \(k\ge2\), define
\begin{equation}
    \widetilde W_{l,k}^{[m]}(z):= \operatorname{Wr}\!\left[\mathcal H_l^{[m]}(z),\mathcal H_{l+k}^{[m]}(z)\right],
\end{equation}
where $\mathcal H_n^{[m]}(z)$ are the Gould--Hopper polynomials defined in \eqref{eq:GH-generating}. Then, all nonzero roots of $\widetilde W_{l,k}^{[m]}(z)$ are simple.
\end{proposition}

\begin{proof}
From \eqref{eq:GH-generating}, we get
\begin{equation}\label{eq:GH-expansion}
    \mathcal H_n^{[m]}(z)
    =
    n!\sum_{j=0}^{\lfloor n/m\rfloor}
    \frac{(-c_m)^jz^{n-mj}}{j!(n-mj)!}.
\end{equation}
Since \(1\leq l<m\), we have
\begin{equation}\label{eq:GH-l}
    \mathcal H_l^{[m]}(z)=z^l.
\end{equation}
Write \(l+k=qm+r\), where $0 \le r\le m-1$. Then
\begin{equation}
    \mathcal H_{l+k}^{[m]}(z)
    =
    (qm+r)!z^rQ_q^{(r)}(z^m),
\end{equation}
where
\begin{equation}\label{eq:Q-definition}
    Q_q^{(r)}(X)
    :=
    \sum_{j=0}^{q}
    \frac{(-c_m)^jX^{q-j}}{j!(r+(q-j)m)!}.
\end{equation}

If \(q=0\), then
\begin{equation}
    \widetilde W_{l,k}^{[m]}(z)=k\,z^{l+r-1}.
\end{equation}
This implies $\widetilde W_{l,k}^{[m]}(z)$ has no nonzero roots. Hence, for the rest of the proof, we
assume \(q\geq1\). Since \(\mathcal H_{l+k}^{[m]}(z)\) is $(m-1)$-symmetric $(m-1)$-orthogonal,  by Theorem~\ref{thm:ben-romdhane-zero-theorem}, all
roots of \(Q_q^{(r)}(X)\) are positive and simple. Hence, \(Q_q^{(r)}(X)\) is strictly hyperbolic.

Differentiating
\(\mathcal H_{l+k}^{[m]}(z)\), we obtain
\begin{align}
    \frac{d}{dz}\mathcal H_{l+k}^{[m]}(z)
    &=
    (qm+r)!
    \left(
        rz^{r-1}Q_q^{(r)}(X)
        +
        mz^{r+m-1}\left(Q_q^{(r)}\right)'(X)
    \right), \quad X=z^m,
\end{align}
and
\begin{align}
    \widetilde W_{l,k}^{[m]}(z)
   =
    (qm+r)!z^{l+r-1}
    \left(
        (r-l)Q_q^{(r)}(X)
        +
        mX\left(Q_q^{(r)}\right)'(X)
    \right).
\end{align}
Next, from \eqref{eq:Q-definition},
\begin{equation}
    X\left(Q_q^{(r)}\right)'(X)
    =
    \sum_{j=0}^{q}
    \frac{(q-j)(-c_m)^jX^{q-j}}{j!(r+(q-j)m)!}
    =
    qQ_q^{(r)}(X)
    +
    c_mQ_{q-1}^{(r)}(X).
\end{equation}
Therefore, since \(r-l+mq=k\), we obtain
\begin{equation}\label{eq:GH-Wr-factorization}
    \widetilde W_{l,k}^{[m]}(z)
    =
    (qm+r)!z^{l+r-1}
    \left(kQ_q^{(r)}(z^m)+mc_mQ_{q-1}^{(r)}(z^m)\right).
\end{equation}

It remains to prove that \(kQ_q^{(r)}(X)+mc_mQ_{q-1}^{(r)}(X)\) has only simple
roots. The case $q=1$ is immediate, since this polynomial has degree one. Thus, in what follows,
we assume $q\geq2.$

Recall that \(Q_q^{(r)}(X)\) and
\(Q_{q-1}^{(r)}(X)\) are strictly hyperbolic and all their roots are positive and simple. By Corollary~\ref{thm:defined-polynomials-roots}, these two polynomials have no common roots, and their roots strictly interlace. Let \(0<a_1<a_2<\cdots<a_q\) be the roots of \(Q_q^{(r)}(X)\), and let \(0<b_1<b_2<\cdots<b_{q-1}\) be the roots of \(Q_{q-1}^{(r)}(X)\). By the strict interlacing property, these roots satisfy \begin{equation}\label{eq:FG-root-order} 0<a_1<b_1<a_2<b_2<\cdots<b_{q-1}<a_q. 
\end{equation}
Set \(F(X):=Q_q^{(r)}(X)\) and \(G(X):=Q_{q-1}^{(r)}(X)\). Then by
\eqref{eq:Q-definition}, both of them have positive leading coefficients, so
\begin{equation}
    F(X)=c_F\prod_{i=1}^{q}(X-a_i),
    \qquad G(X)=c_G\prod_{i=1}^{q-1}(X-b_i),
    \qquad c_F,c_G>0.
\end{equation}

The inequality \(\deg G(X)<\deg F(X)\), together with the facts that \(F(X)\) has only simple roots \(a_1,\ldots,a_q\) and that \(F(X)\) and \(G(X)\) have no common roots, implies that the rational function \(G(X)/F(X)\) has only simple poles at \(a_1,\ldots,a_q\) and vanishes at infinity. Hence, it admits the partial fraction decomposition
\begin{equation}\label{eq:partial-fraction}
    \frac{G(X)}{F(X)}
    =
    \sum_{i=1}^{q}\frac{A_i}{X-a_i}, \quad A_i=\frac{G(a_i)}{F'(a_i)}.
\end{equation}
Moreover, by \eqref{eq:FG-root-order}, we have
\begin{equation}
    G(a_i)
    =
    c_G\prod_{j=1}^{q-1}(a_i-b_j).
\end{equation}
For \(j<i\), \(a_i-b_j\) is positive, while for \(j\geq i\), \(a_i-b_j\) is negative. Thus the number of negative factors in \(G(a_i)\)
is \(q-i\), and
\begin{equation}
    \operatorname{sgn}G(a_i)=(-1)^{q-i}.
\end{equation}
Similarly, we have
\begin{equation}
    F'(a_i)
    =
    c_F\prod_{\substack{j=1\\j\neq i}}^{q}(a_i-a_j),
\end{equation}
and
\begin{equation}
    \operatorname{sgn}F'(a_i)=(-1)^{q-i}.
\end{equation}
Consequently,
\begin{equation}
    A_i=\frac{G(a_i)}{F'(a_i)}>0,
    \qquad i=1,\ldots,q.
\end{equation}

Differentiating \eqref{eq:partial-fraction}, we obtain
\begin{equation}
    \left(\frac{G}{F}\right)'(X)
    =
    -\sum_{i=1}^{q}\frac{A_i}{(X-a_i)^2}<0,
    \qquad
    X\in\mathbb R\setminus\{a_1,\ldots,a_q\}.
\end{equation}
On the other hand, one has
\begin{align}
    \left(\frac{G}{F}\right)'(X)
    =
    \frac{G'(X)F(X)-G(X)F'(X)}{F(X)^2}
     =
    -\frac{\operatorname{Wr}[G,F](X)}{F(X)^2}.
\end{align}
Since \(F(X)^2>0\) for \(X\neq a_i\), we get
\begin{equation}\label{eq:WrFG-negative-offroots}
    \operatorname{Wr}[G,F](X)>0,
    \qquad
    X\in\mathbb R\setminus\{a_1,\ldots,a_q\}.
\end{equation}
Next, we consider the points \(X=a_i\). Since \(F(a_i)=0\), we have
\begin{equation}
    \operatorname{Wr}[G,F](a_i)
    =
    G(a_i)F'(a_i).
\end{equation}
As shown above, \(F'(a_i)\) and \(G(a_i)\) have the same sign. Hence
\begin{equation}
    \operatorname{Wr}[G,F](a_i)>0,
    \qquad i=1,\ldots,q.
\end{equation}
Combining this with \eqref{eq:WrFG-negative-offroots}, we conclude that
\begin{equation}
    \operatorname{Wr}[G,F](X)>0,
    \qquad X\in\mathbb R.
\end{equation}
Together with the strict interlacing property \eqref{eq:FG-root-order}, this yields
\begin{equation}\label{eq:G-proper-F}
    G(X)\ll F(X).
\end{equation}

By Theorem \ref{thm:obreschkoff}, since \(F(X)\) and \(G(X)\) are strictly hyperbolic, have no
common roots, and are in proper position, each nontrivial linear combination
of \(F(X)\) and \(G(X)\)  is strictly hyperbolic. In particular, since \(k>0\) and
\(mc_m>0\),
the polynomial
\begin{equation}
    kQ_q^{(r)}(X)+mc_mQ_{q-1}^{(r)}(X)
    =
    kF(X)+mc_mG(X)
\end{equation}
is strictly hyperbolic. Hence, all its roots are real and simple.

Finally, by \eqref{eq:GH-Wr-factorization}, all nonzero roots of $\widetilde W_{l,k}^{[m]}(z)$ are simple.

\end{proof}

Let \(\alpha=c_m^{-1/m}\exp(\pi\i/m)\).  Then \(\alpha^m=-c_m^{-1}\), and
\eqref{eq:GH-generating} gives
\begin{equation}
    p_n^{[m]}(z)
    =
    \frac{\alpha^n}{n!}\,
    \mathcal H_n^{[m]}(\alpha^{-1}z),
    \qquad n\ge0.
\end{equation}
Thus, we have 
\begin{equation}
    W_2^{[m,k,l]}(z)  = c_2^{[m,k,l]} \Wr\left[p_l^{[m]},p_{l+k}^{[m]}\right](z)
    =
    c_2^{[m,k,l]}\frac{\alpha^{2l+k-1}}{l!(l+k)!}
    \Wr\left[\mathcal H_l^{[m]},\mathcal H_{l+k}^{[m]}\right](\alpha^{-1}z).
\end{equation}  
Therefore, all nonzero roots of $W_2^{[m,k,l]}(z)$ are simple. This completes the proof of case (i) in Theorem \ref{thm:main-intro}.

We next describe the zero distributions of the generalized Wronskian--Hermite polynomials
\(W_2^{[m,k,l]}(z)\).

\begin{proposition}
\label{prop:GH-Wr-sign-distribution}
Let \(m\geq2\), \(1\leq l<m\), and \(k\geq2\). Write
\begin{equation}
    l+k=qm+r,
    \qquad q\geq0,
    \qquad 0\leq r<m.
\end{equation}
Let $Q_{-1}^{(r)}(X):=0,$
and set
\begin{equation}
    R_q^{(r,l,k)}(X)
    :=
    kQ_q^{(r)}(X)+mc_mQ_{q-1}^{(r)}(X).
\end{equation}
Then the following statements hold.

\begin{enumerate}[(i)]
\item If \(q=0\), then \(R_q^{(r,l,k)}(X)\) has no nonzero roots.

\item If \(q\geq1\) and \(r>l\), then all roots of
\(R_q^{(r,l,k)}(X)\) are positive and simple.

\item If \(q\geq1\) and \(r=l\), then \(X=0\) is a simple root of
\(R_q^{(l,l,k)}(X)\), and all other roots are positive and simple.

\item If \(q\geq1\) and \(r<l\), then \(R_q^{(r,l,k)}(X)\) has one simple negative root and \(q-1\) simple positive roots. 
\end{enumerate}
\end{proposition}

\begin{proof}
If \(q=0\), then
\begin{equation}
    R_0^{(r,l,k)}(X)
    =
    kQ_0^{(r)}(X)
    =
    \frac{k}{r!}.
\end{equation}
Thus, \(R_0^{(r,l,k)}(X)\) is a nonzero constant and has no nonzero roots.
Now assume \(q\geq1\). As shown in the proof of
Proposition~\ref{prop:GH-Wronskian}, the polynomial
\(R_q^{(r,l,k)}(X)\) has \(q\) real simple roots. It remains to determine their signs.

Using
\begin{equation}
    R_q^{(r,l,k)}(X)
    =
    (r-l)Q_q^{(r)}(X)
    +
    mX\left(Q_q^{(r)}\right)'(X),
\end{equation}
where
\begin{equation}
    Q_q^{(r)}(X)=\sum_{s=0}^{q}A_sX^s,
    \qquad
    A_s=
    \frac{(-c_m)^{q-s}}{(q-s)!(r+sm)!},
\end{equation}
we obtain
\begin{equation}
    R_q^{(r,l,k)}(X)
    =
    \sum_{s=0}^{q}C_sX^s,
\end{equation}
where
\begin{equation}
    C_s
    =
    (ms+r-l)A_s
    =
    \frac{(ms+r-l)(-c_m)^{q-s}}{(q-s)!(r+sm)!},
    \qquad s=0,1,\ldots,q.
\end{equation}

First, suppose \(r>l\), then we have
\begin{equation}
    ms+r-l>0,
    \qquad s=0,1,\ldots,q.
\end{equation}
Hence
\begin{equation}
    \operatorname{sgn}C_s=(-1)^{q-s},
    \qquad s=0,1,\ldots,q.
\end{equation}
The coefficient of \(X^s\) in \(R_q^{(r,l,k)}(-X)\) is \((-1)^sC_s\), and
therefore
\begin{equation}
    \operatorname{sgn}\left((-1)^sC_s\right)=(-1)^q,
    \qquad s=0,1,\ldots,q.
\end{equation}
Thus, the coefficient sequence of \(R_q^{(r,l,k)}(-X)\) has no sign variation.
By Descartes' rule of signs (see Theorem \ref{thm:descartes-rule-of-signs}), \(R_q^{(r,l,k)}(X)\) has no negative roots. Since $C_0 
    \neq0,$ zero is not a root. As all \(q\) roots are real and simple, all of them must be positive.

Next, suppose \(r=l\), then we have $C_0=0$ and $C_1 \neq0$.
Thus \(X=0\) is a simple root of \(R_q^{(l,l,k)}(X)\). Put
\begin{equation}
    S(X):=\frac{R_q^{(l,l,k)}(X)}{X}.
\end{equation}
For \(s=1,\ldots,q\), we have
\begin{equation}
    \operatorname{sgn}C_s=(-1)^{q-s}.
\end{equation}
The coefficient of \(X^{s-1}\) in \(S(-X)\) is \((-1)^{s-1}C_s\), and hence
\begin{equation}
    \operatorname{sgn}\left((-1)^{s-1}C_s\right)=(-1)^{q-1},
    \qquad s=1,\ldots,q.
\end{equation}
Therefore, \(S(-X)\) has no sign variation. By Descartes' rule of signs,
\(S\) has no negative roots. Since \(R_q^{(l,l,k)}\) is strictly hyperbolic and
has one simple root at \(X=0\), all its remaining \(q-1\) roots are positive and
simple.

Finally, suppose $r < l$ and define $\delta := l - r$, which satisfies $1 \le \delta \le m - 1$. For each $s = 1, \ldots, q$, we clearly have
\begin{equation}
    ms + r - l = ms - \delta > 0.
\end{equation}
It then follows that the signs of the coefficients are given by
\begin{equation}
    \operatorname{sgn}C_s = (-1)^{q-s} \quad \text{for } s = 1, \ldots, q, \qquad \text{and} \qquad \operatorname{sgn}C_0 = (-1)^{q+1}.
\end{equation}
Consequently, the coefficient sequence of $R_q^{(r,l,k)}(X)$, when arranged in descending order of powers, exhibits exactly $q-1$ sign variations. By Descartes' rule of signs, $R_q^{(r,l,k)}(X)$ has at most $q-1$ positive roots.
On the other hand, for \(s=1,\ldots,q\), the coefficient of \(X^s\) in
\(R_q^{(r,l,k)}(-X)\) has sign
\begin{equation}
    \operatorname{sgn}\left((-1)^sC_s\right)=(-1)^q,
\end{equation}
whereas its constant coefficient has sign
\begin{equation}
    \operatorname{sgn}C_0=(-1)^{q+1}.
\end{equation}
Therefore the coefficient sequence of \(R_q^{(r,l,k)}(-X)\), arranged in
descending powers, has exactly one sign variation. By Descartes' rule of signs,
\(R_q^{(r,l,k)}(X)\) has at most one negative root.
Moreover, the fact that $C_0 \neq 0$ ensures that zero is not a root. Since $R_q^{(r,l,k)}(X)$ has $q$ distinct real roots in total, where at most $q-1$ of them are positive and at most one is negative, this forces it to have exactly $q-1$ positive roots and exactly one negative root.

\end{proof}

\begin{corollary}
\label{cor:W2-ray-distribution}
Under the assumptions of
Proposition~\ref{prop:GH-Wr-sign-distribution}, define two families of rays by
\begin{equation}
    \mathcal R_{\mathrm{odd}}^{[m]}
    :=
    \left\{
    z\in\mathbb C\setminus\{0\}:
    \arg(z)\equiv\frac{(2\nu+1)\pi}{m}\pmod{2\pi},
    \;
    \nu=0,1,\ldots,m-1
    \right\},
\end{equation}
and
\begin{equation}
    \mathcal R_{\mathrm{even}}^{[m]}
    :=
    \left\{
    z\in\mathbb C\setminus\{0\}:
    \arg(z)\equiv\frac{2\pi\nu}{m}\pmod{2\pi},
    \;
    \nu=0,1,\ldots,m-1
    \right\}.
\end{equation}
Then, the nonzero roots of \(W_2^{[m,k,l]}(z)\) are distributed as follows.

\begin{enumerate}[(i)]
\item If \(q=0\), then there is no nonzero root.

\item If \(q\geq1\) and \(r\ge l\), then all nonzero roots lie on
\(\mathcal R_{\mathrm{odd}}^{[m]}\).

\item If \(q\geq1\) and \(r<l\), there are \((q-1)m\) roots lying on
\(\mathcal R_{\mathrm{odd}}^{[m]}\), while the remaining $m$ roots lie on \(\mathcal R_{\mathrm{even}}^{[m]}\).
\end{enumerate}
\end{corollary}

\begin{example}
We illustrate Corollary~\ref{cor:W2-ray-distribution} in
Fig.~\ref{fig:WHZeros-N2-k11-l1} by taking
\begin{equation}
    N=2,\quad k=11,\quad l=1.
\end{equation}
The red points denote the roots lying on
\(\mathcal R_{\mathrm{odd}}^{[m]}\), the blue points denote the roots lying on
\(\mathcal R_{\mathrm{even}}^{[m]}\), and the black point denotes the root at the origin. For example, when \(m=2\), we have \(l+k=12=6m+0\), so that
\(q=6\) and \(r=0<l\). Hence, Corollary~\ref{cor:W2-ray-distribution}
implies that \(10\) nonzero roots lie on \(\mathcal R_{\mathrm{odd}}^{[2]}\), while
\(2\) nonzero roots lie on \(\mathcal R_{\mathrm{even}}^{[2]}\), as shown in
Fig.~\ref{fig:WHZeros-N2-k11-l1} (a).

\begin{figure}[H]
    \centering
    \includegraphics[width=0.9\linewidth]{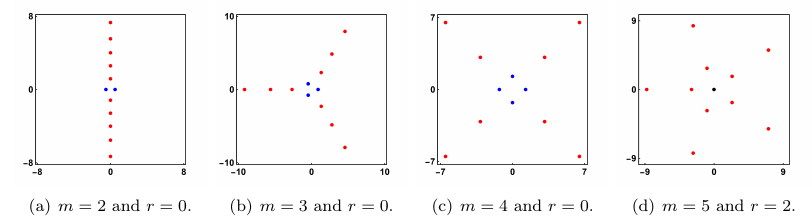}
    \caption{Zero distributions of the generalized Wronskian--Hermite polynomials $W_{2}^{[m,11,1]}(z)$ for $m=2,3,4,5$.}
    \label{fig:WHZeros-N2-k11-l1}
\end{figure}
\end{example}

\subsection{The higher-order case}
We consider the case for \(N\ge3\).   Let
\(k\ge2\), \(1\le l\le k-1\), and
\begin{equation}\label{eq:higher-order-ni}
    n_i=l+(i-1)k,\qquad i=1,\ldots,N.
\end{equation}
We assume that
\begin{equation}\label{eq:large-m-assumption}
    2\le m\le n_N,\qquad m\not\equiv0\pmod{k},\qquad\text{and} \qquad 2m>n_N.
\end{equation}
Since \(m\le n_N\) and \(m>n_N/2>n_1=l\), the set \(\{i:n_i\ge m\}\) is nonempty but does not contain \(1\). Denote by \(r\) the number of such indices, i.e.,
\begin{equation}\label{eq:r-i0-definition}
    r=\#\{i:n_i\ge m\},
\end{equation}
and
$i_0=N-r+1,$ then \(1\le r\le N-1\). Moreover, since \(n_1<n_2<\cdots<n_N\), the indices satisfying \(n_i\ge m\) are precisely \(i=i_0,\ldots,N\).
Define
\begin{equation}\label{eq:xi-y}
    \xi_j=n_{i_0+j-1},
    \qquad
    y_j=\xi_j-m,
    \qquad j=1,\ldots,r.
\end{equation}
Since \(n_N<2m\), these integers satisfy
\begin{equation}\label{eq:y-xi-separation-large-m}
    0\le y_1<\cdots<y_r<m\le \xi_1<\cdots<\xi_r.
\end{equation}

\begin{proposition}\label{prop:two-term-reduction}
Under the assumptions \eqref{eq:large-m-assumption} with $N\ge3$, we have
\begin{equation}\label{eq:two-term-factor}
    W_N^{[m,k,l]}(z)=z^{\Gamma-rm}\det(z^mI_r+M),
\end{equation}
where \(\Gamma\) is defined in
\eqref{def:Gamma},  \(r\) is given in \eqref{eq:r-i0-definition} and \(M=(M_{ab})_{a,b=1}^r\) is the  \(r\times r\)  real matrix with
\begin{equation}\label{eq:M-definition}
    M_{ab}
    =
    \frac{\xi_a!}{y_a!}\ell_{i_0+b-1}(y_a),
    \qquad a,b=1,\ldots,r,
\end{equation}
and
\begin{equation}
    \ell_h(u)
    =
    \prod_{\substack{1\le i\le N\\ i\ne h}}
    \frac{u-n_i}{n_h-n_i}.
\end{equation}
\end{proposition}

\begin{proof}
From the definition of $p_n^{[m]}(z)$ in \eqref{eq:p-generating-intro}, we have
\begin{equation}\label{eq:p-expansion}
    p_n^{[m]}(z)
    =
    \sum_{s=0}^{\lfloor n/m\rfloor}
    \frac{z^{n-ms}}{(n-ms)!s!},
    \qquad n\ge0.
\end{equation}
We keep the convention \(p_n^{[m]}(z)=0\) for \(n<0\).
Since \(n_i\le n_N<2m\) for \(i=1,2,\cdots,N\), each entry
\(p_{n_i-j+1}^{[m]}(z)\) in $W_N^{[m,k,l]}(z)$ has at most two terms:
the unshifted term \(s=0\) and the one-shifted term
\(s=1\). 

Define the polynomials $\Phi_q(u)$ by
\begin{equation}
    \Phi_0(u)=1,
    \qquad
    \Phi_q(u)=u(u-1)\cdots(u-q+1)
    =
    \prod_{s=0}^{q-1}(u-s),
    \qquad q\ge1.
\end{equation}
In particular, if
\(u_0\in\mathbb{Z}_{\ge0}\) and \(q>u_0\), then the product contains a zero factor, so \(\Phi_q(u_0)=0\).

Set $q=j-1$. For the unshifted part,  if \(q\le n_i\), then we have
\begin{equation}
    \frac{z^{n_i-q}}{(n_i-q)!}
    =
    \frac{z^{n_i-q}}{n_i!} \Phi_q(n_i).
\end{equation}
If \(q>n_i\), then
\(\Phi_q(n_i)=0\).
For the shifted part, suppose that \(n_i\ge m\), and write
\(y=n_i-m\). If \(q\le y\), then
\begin{equation}
    \frac{z^{n_i-m-q}}{(n_i-m-q)!}
    =
    \frac{z^{n_i-q-m}}{n_i!} 
    \frac{n_i!}{y!}\Phi_q(y).
\end{equation}
If \(q>y\), then \(\Phi_q(y)=0\). If
\(n_i<m\), there is no shifted contribution.
Thus
\begin{equation}
    W_N^{[m,k,l]}(z)=c_N^{[m,k,l]}\det_{1\le i,j\le N}\left[p_{n_i-j+1}^{[m]}(z)\right]
    =
    \frac{z^\Gamma}{\prod_{1\le i<j\le N}(n_j-n_i)}\det\left( E(z)\right),
    \label{eq:det-factorized-large-m}
\end{equation}
where $\Gamma$ is given in \eqref{def:Gamma} and 
the entries of the matrix \(E(z)\) are given by
\begin{equation}
    E_{ij} = 
\begin{cases} 
\Phi_{j-1}(n_i), & i < i_0, \\ 
\Phi_{j-1}(n_i) + z^{-m}\dfrac{\xi_{i-i_0+1}!}{y_{i-i_0+1}!} \Phi_{j-1}(y_{i-i_0+1}), & i \ge i_0.
\end{cases}
\end{equation}

Define the $N \times N$ matrix $G=\left(g_{ij}\right)$ by
\begin{equation}
    g_{ij}=\Phi_{j-1}(n_i).
    \label{eq:G-definition-large-m}
\end{equation}
Since \(\Phi_0,\Phi_1,\ldots,\Phi_{N-1}\) are monic polynomials of degree
\(0,1,\ldots,N-1\),
\begin{equation}
    \det \left(G\right)
    =
    \prod_{1\le i<j\le N}(n_j-n_i).
    \label{eq:G-vandermonde-large-m}
\end{equation}
Combining \eqref{eq:det-factorized-large-m} with the normalization
\(c_N^{[m,k,l]}\), we obtain
\begin{equation}
    W_N^{[m,k,l]}(z)
    =
    z^\Gamma\frac{\det \left(E(z)\right)}{\det \left(G\right)}.
    \label{eq:W-normalized-extracted-large-m}
\end{equation}
Denote
\begin{equation}
    E(z)=G+z^{-m}B,
    \label{eq:E-G-B-large-m}
\end{equation}
where the entries of matrix $B$ are defined as follows:
\begin{equation}
    B_{ij}=0,\qquad i<i_0,
\end{equation}
and
\begin{equation}
    B_{i_0+a-1,j}
    =
    \frac{\xi_a!}{y_a!}\Phi_{j-1}(y_a),
    \qquad a=1,\ldots,r,\quad j=1,\ldots,N.
    \label{eq:B-entry-large-m}
\end{equation}

We now perform a matrix factorization of \(B\). 
For \(h=1,\ldots,N\), let
\begin{equation}
    \ell_h(u)
    =
    \prod_{\substack{1\le i\le N\\ i\neq h}}
    \frac{u-n_i}{n_h-n_i}.
    \label{eq:lagrange-basis-large-m}
\end{equation}
Then \(\ell_h(n_i)=\delta_{hi}\). Since \(\deg\Phi_q=q\le N-1\), Lagrange
interpolation gives, for \(q=0,\ldots,N-1\),
\begin{equation}
    \Phi_q(t)
    =
    \sum_{h=1}^{N}\Phi_q(n_h)\ell_h(t).
    \label{eq:falling-factorial-interpolation-large-m}
\end{equation}
Since the first $i_0-1$ rows of $B$ are zero and
\begin{equation}
    B_{i_0+a-1,j}=\frac{\xi_a!}{y_a!}\Phi_{j-1}(y_a)=
    \sum_{h=1}^{N}
    \frac{\xi_a!}{y_a!}\ell_h(y_a)\Phi_{j-1}(n_h),
\end{equation}
we have
\begin{equation*}
    B=KG
\end{equation*}
with
\begin{equation}
    K=
    \begin{pmatrix}
    O & O\\
    K_1&M
    \end{pmatrix},
    \label{eq:K-block-form-large-m}
\end{equation}
where each \(O\) denotes a zero block of the appropriate size,
\begin{equation}
    K_1=
    \left(
    \frac{\xi_a!}{y_a!}\ell_h(y_a)
    \right)_{1\le a\le r,\ 1\le h\le i_0-1},
\end{equation}
and
\begin{equation}
    M
    =\left(
    \frac{\xi_a!}{y_a!}\ell_{i_0+b-1}(y_a)\right)_{1\le a,b\le r}.
    \label{eq:M-active-block-large-m}
\end{equation}
Therefore
\begin{equation}
    E(z)=G+z^{-m}KG=(I_N+z^{-m}K)G.
    \label{eq:E-KG-factorization-large-m}
\end{equation}
Since
\begin{equation}
    I_N+z^{-m}K
    =
    \begin{pmatrix}
    I_{i_0-1}&O\\
    z^{-m}K_1&I_r+z^{-m}M
    \end{pmatrix},
\end{equation}
we have
\begin{equation}
    \det\left(E(z)\right)
    =
    \det\left(G\right)\det\left(I_r+z^{-m}M\right).
\end{equation}
Substituting this into \eqref{eq:W-normalized-extracted-large-m} gives
\begin{align}
    W_N^{[m,k,l]}(z)
    =
    z^\Gamma\det(I_r+z^{-m}M) =
    z^{\Gamma-rm}\det(z^mI_r+M).
    \label{eq:W-reduction-large-m}
\end{align}

It remains to prove that the exponent \(\Gamma-rm\) is nonnegative. Note that
\begin{align}
    \sum_{i=1}^{N}n_i-rm
    =
    \sum_{i=1}^{i_0-1}n_i+\sum_{a=1}^{r}(\xi_a-m) =
    \sum_{i=1}^{i_0-1}n_i+\sum_{a=1}^{r}y_a.
\end{align}
Since \(m\not\equiv0\pmod{k}\) and \(n_i\equiv l\pmod{k}\), we have
\begin{equation}
    y_a=\xi_a-m\equiv l-m\not\equiv l\pmod{k},
\end{equation}
which implies \(y_a\neq n_i\).
Moreover, since the \(n_i\)'s are distinct and the \(y_a\)'s are distinct,
the \(N\) nonnegative integers
\(n_1,\ldots,n_{i_0-1},y_1,\ldots,y_r\) are distinct. Consequently,
their sum is at least \(\binom{N}{2}\), and 
\begin{equation}
    \Gamma-rm
    =
    \sum_{i=1}^{N}n_i-rm-\binom{N}{2}
    \ge0.
\end{equation}
This completes the proof.
\end{proof}

\begin{proposition}\label{prop:reduced-roots}
Under the assumptions of Proposition \ref{prop:two-term-reduction}, the polynomial
\(\det(XI_r+M)\) has roots \(X_1,\ldots,X_r\) satisfying
\begin{equation}
    X_\nu\in\R\setminus\{0\}\quad
    \text{ and } \quad
    |X_\mu|\ne |X_\nu|,
    \quad \mu\ne\nu,
\end{equation}
where $1\le\mu,\nu\le r.$
Consequently, all roots of \(\det(X I_r+M)\) are nonzero and simple.
\end{proposition}

\begin{proof}
Denote
\begin{equation}
    P(u)=\prod_{i=1}^{N}(u-n_i).
\end{equation}
Then, for \(h=1,\ldots,N\), we have
\begin{equation}
    \ell_h(u)
    =
    \frac{P(u)}{(u-n_h)P'(n_h)}.
\end{equation}
Using \(n_{i_0+b-1}=\xi_b\), we obtain
\begin{equation}
    M_{ab}
    =
    \frac{\xi_a!}{y_a!}
    \frac{P(y_a)}
    {(y_a-\xi_b)P'(\xi_b)}.
    \label{eq:M-cauchy-raw-large-m}
\end{equation}
Since the factorial term is positive and $y_a - \xi_b < 0$ from \eqref{eq:y-xi-separation-large-m}, we next determine the signs of $P(y_a)$ and $P'(\xi_b)$, respectively.

Let us write $m = sk + \mu$, where $s \in \mathbb{Z}_{\ge0}$ and $1 \le \mu \le k-1$. Note that the condition $1 \le \mu \le k-1$ comes from $m \not\equiv 0 \pmod{k}$. For each $a = 1, \ldots, r$, we set $j_a = i_0 + a - 1$, which implies $\xi_a = n_{j_a}$ and leads to
\begin{equation}
    y_a = n_{j_a} - m= l + (j_a - s - 1)k - \mu.
\end{equation}

To determine the sign of $P(y_a)$, we first count the number of nodes $n_i$ less than $y_a$. The inequality $n_i < y_a$ can be explicitly written as
\begin{equation}
    l+(i-1)k < l+(j_a-s-1)k-\mu.
\end{equation}
Let $q_a = j_a - s - 1$. Since $y_a \ge 0$, $q_a$ must be nonnegative; indeed, if $q_a < 0$, then $$y_a = l + q_ak - \mu \le l - k - \mu < 0,$$ which yields a contradiction. Moreover, since $0 < \mu < k$, the inequalities
\begin{equation}
    (q_a-1)k < q_ak-\mu < q_ak
\end{equation}
imply that 
$$0\le y_a<n_{1}, \quad\text{ for }\quad q_a =0,$$
and
$$n_{q_a}<y_a<n_{q_{a}+1}, \quad\text{ for }\quad q_a \ge 1$$
Consequently, the number of nodes $n_i$ strictly less than $y_a$ is
\begin{equation}
    q_a = j_a - s - 1.
\end{equation}
Furthermore, there is no node $n_i$ equal to $y_a$. 
As a result, the product $P(y_a) = \prod_{i=1}^{N}(y_a-n_i)$ contains exactly $q_a$ positive factors and $N-q_a$ negative factors, which shows that
\begin{equation}
    \operatorname{sgn}P(y_a) = (-1)^{N-q_a}. \label{eq:sign-P-y-large-m}
\end{equation}

Similarly, considering $j_b = i_0 + b - 1$ so that $\xi_b = n_{j_b}$, the derivative evaluated at the node $n_{j_b}$ yields
\begin{equation}
    P'(\xi_b) = \prod_{\substack{1\le i\le N\\ i\neq j_b}}(n_{j_b}-n_i).
\end{equation}
Here, the factors corresponding to $i > j_b$ are strictly negative, and there are precisely $N-j_b$ such factors. This implies
\begin{equation}
    \operatorname{sgn}P'(\xi_b) = (-1)^{N-j_b}. \label{eq:sign-P-prime-large-m}
\end{equation} 

Substituting \eqref{eq:sign-P-y-large-m} and \eqref{eq:sign-P-prime-large-m} into \eqref{eq:M-cauchy-raw-large-m}, since $q_a = j_a - s - 1$, $j_a = i_0 + a - 1$ and $j_b = i_0 + b - 1$, the sign of the matrix entries can be computed as
\begin{equation}
    \operatorname{sgn}M_{ab} = (-1)^{a+b-s}. \label{eq:sign-M-large-m}
\end{equation}

Let $\eta = (-1)^s$ and
$$S = \operatorname{diag}(1, -1, 1, -1, \ldots, (-1)^{r-1}).$$ Define the matrix $A = \eta SMS$, whose entries are given by
\begin{equation}
    A_{ab} = \eta(-1)^{a-1}M_{ab}(-1)^{b-1}.
\end{equation}
By virtue of \eqref{eq:sign-M-large-m}, the sign of $A_{ab}$ can be evaluated as
\begin{equation}
    (-1)^s(-1)^{a-1}(-1)^{b-1}(-1)^{a+b-s}  = 1,
\end{equation}
which ensures that all entries of $A$ are positive.

We now present the explicit expressions for these positive entries.  Combining $y_a - \xi_b < 0$, \eqref{eq:M-cauchy-raw-large-m} and the aforementioned sign cancellation, we obtain
\begin{equation}
    A_{ab} = |M_{ab}| = \frac{\xi_a!}{y_a!} \frac{|P(y_a)|}{(\xi_b-y_a)|P'(\xi_b)|} = \frac{u_a v_b}{\xi_b-y_a}, \label{eq:A-cauchy-type-large-m}
\end{equation}
where the positive factors $u_a$ and $v_b$ are defined by
\begin{equation}
    u_a = \frac{\xi_a!|P(y_a)|}{y_a!} > 0 \qquad \text{and} \qquad v_b = \frac{1}{|P'(\xi_b)|} > 0.
\end{equation}

We now prove that \(A\) is SSR. Fix \(q\in\{1,\ldots,r\}\), and choose arbitrary
row and column indices
\begin{equation}
    1\le a_1<\cdots<a_q\le r,
    \qquad
    1\le b_1<\cdots<b_q\le r.
\end{equation}
By \eqref{eq:A-cauchy-type-large-m}, we have
\begin{align}
    \det\left(A_{a_s b_t}\right)
    &=
    \left(\prod_{s=1}^{q}u_{a_s}\right)
    \left(\prod_{t=1}^{q}v_{b_t}\right)
    \det\left(
    \frac{1}{\xi_{b_t}-y_{a_s}}
    \right),
\end{align}
where 
\begin{equation}
    \left(\prod_{s=1}^{q}u_{a_s}\right)
    \left(\prod_{t=1}^{q}v_{b_t}\right)>0.
\end{equation}
By the Cauchy determinant formula,
\begin{equation}
    \det\left(
    \frac{1}{\xi_{b_t}-y_{a_s}}
    \right)
    =
    \frac{
    \prod_{1\le s<t\le q}(y_{a_s}-y_{a_t})
    \prod_{1\le s<t\le q}(\xi_{b_t}-\xi_{b_s})
    }{
    \prod_{s,t=1}^{q}(\xi_{b_t}-y_{a_s})
    }.
    \label{eq:cauchy-determinant-large-m}
\end{equation}
By
\eqref{eq:y-xi-separation-large-m},
\begin{equation}
    y_{a_1}<\cdots<y_{a_q}<\xi_{b_1}<\cdots<\xi_{b_q}.
\end{equation}
Therefore each factor in the denominator of
\eqref{eq:cauchy-determinant-large-m} is positive, and each factor in
\begin{equation}
    \prod_{1\le s<t\le q}(\xi_{b_t}-\xi_{b_s})
\end{equation}
is positive. On the other hand, each factor in
\begin{equation}
    \prod_{1\le s<t\le q}(y_{a_s}-y_{a_t})
\end{equation}
is negative. There are \(\binom q2\) such factors. Hence, we have
\begin{equation}
    \operatorname{sgn}
    \det\left(A_{a_s b_t}\right)
    =
    (-1)^{\binom q2}.
\end{equation}
Therefore, for every
\(q=1,\ldots,r\), all \(q\times q\) minors of \(A\) are nonzero and have the
same sign. Hence, \(A\) is an SSR matrix. In particular, \(\det A\neq0\).

By Theorem~\ref{thm:GK}, the eigenvalues of \(A\) are real, nonzero,
algebraically simple, and have distinct moduli. Since
\begin{equation}
    A=\eta SMS,
    \qquad
    S^{-1}=S,
    \qquad
    \eta\in\{1,-1\},
\end{equation}
the eigenvalues of \(M\) are also real, nonzero, algebraically simple, and have distinct moduli.

By denoting the eigenvalues of $M$ as $-X_1, -X_2, \ldots, -X_r$, we have
\begin{equation}
    \det(X I_r+M) = \prod_{\nu=1}^{r}(X -X_\nu),
\end{equation}
where $X_\nu \in \mathbb{R}\setminus\{0\}$ and $|X_\mu| \neq |X_\nu|$ for all $\mu \neq \nu$,
which immediately implies that all roots of $\det(X I_r+M)$ are simple.

\end{proof}

According to Propositions
\ref{prop:two-term-reduction}-\ref{prop:reduced-roots}, there exist real nonzero numbers \(X_1,\ldots,X_r\) such that
\begin{equation}\label{eq:Nge3-X-properties}
    X_\mu\ne X_\nu,\qquad |X_\mu|\ne |X_\nu|,
    \qquad \mu\ne\nu,
\end{equation}
and
\begin{equation}\label{eq:Nge3-factorization}
    W_N^{[m,k,l]}(z)
    =
    z^{\Gamma-rm}\prod_{\nu=1}^{r}(z^m-X_\nu).
\end{equation}
This completes the proof of case (ii) in Theorem \ref{thm:main-intro}.

The factorization obtained in Propositions~\ref{prop:two-term-reduction}-\ref{prop:reduced-roots} yields a precise geometric description of the nonzero roots.
\begin{corollary}
\label{cor:alternating-ring-structure}
Under the assumptions of Proposition \ref{prop:two-term-reduction}, let 
\begin{equation}
\label{eq:alternating-ring-factorization}
    W_N^{[m,k,l]}(z)
    =
    z^{\Gamma-rm}
    \prod_{\nu=1}^{r}(z^m-X_\nu),
\end{equation}
where \(X_1,X_2,\ldots,X_r\in\mathbb R\setminus\{0\}\) are labeled so that
\begin{equation}
    |X_1|>|X_2|>\cdots>|X_r|>0.
\end{equation}
Then the following statements hold.

\begin{enumerate}[(i)]
\item The signs of \(X_1,X_2,\ldots,X_r\) alternate, namely
\begin{equation}
    X_\nu X_{\nu+1}<0,
    \qquad
    \nu=1,\ldots,r-1.
\end{equation}

\item The nonzero roots of \(W_N^{[m,k,l]}(z)\) form an alternating
concentric-ring structure: they lie on \(r\) distinct concentric circles
with radii
\begin{equation}
    \rho_\nu=|X_\nu|^{1/m},
    \qquad
    \rho_1>\rho_2>\cdots>\rho_r>0,
\end{equation}
form a regular \(m\)-gon on each circle, and the regular \(m\)-gons on
adjacent circles are rotated by the angle \(\pi/m\) relative to each other.
\end{enumerate}
\end{corollary}

\begin{proof}
The case \(r=1\) is immediate, so we assume \(r\ge2\) in subsequent discussions.

We adopt the notation from the proof of Proposition~\ref{prop:reduced-roots}, where the matrix $M$ is transformed into $A = \eta SMS$. Here, $\eta \in \{1, -1\}$ and $S = S^{-1}$ is a diagonal matrix such that $A$ is a nonsingular, strictly sign-regular matrix. Furthermore, for each $q = 1, 2, \ldots, r$, all $q \times q$ minors of $A$ share the common sign $\varepsilon_q = (-1)^{\binom{q}{2}}$. Let the eigenvalues of $A$ be ordered such that
\begin{equation}
    |\lambda_1| > |\lambda_2| > \cdots > |\lambda_r| > 0.
\end{equation}
where $\lambda_i \in \R \; (1\le i \le r).$

We now determine their signs sequentially. First, since $\varepsilon_1 = 1$, applying Theorem~\ref{thm:SSRk} with $k = 1$ yields $\lambda_1 > 0$. Next, for each $i = 1, \ldots, r - 2$, Corollary~\ref{cor:SSR-consecutive} guarantees that $\varepsilon_i \varepsilon_{i+1} \lambda_{i+1} > 0$. Since the sign product simplifies to
\begin{equation}
    \varepsilon_i \varepsilon_{i+1} = (-1)^{\binom{i}{2} + \binom{i+1}{2}} = (-1)^i,
\end{equation}
we immediately obtain
\begin{equation}
    \operatorname{sgn}(\lambda_\nu) = (-1)^{\nu-1} \qquad \text{for } \nu = 1,2, \ldots, r - 1.
\end{equation}
Finally, to determine the sign of the remaining eigenvalue $\lambda_r$, we recall that $\det A=\lambda_1 \lambda_2 \cdots \lambda_r $. Given that the common sign of the $r \times r$ minors is $\operatorname{sgn}(\det A) = \varepsilon_r = (-1)^{\binom{r}{2}}$, it follows that
\begin{equation}
    \operatorname{sgn}(\lambda_r) =  (-1)^{\binom{r}{2}-\binom{r-1}{2}} = (-1)^{r-1}.
\end{equation}
Consequently, we conclude that the signs of all eigenvalues alternate according to
\begin{equation}
    \operatorname{sgn}(\lambda_\nu) = (-1)^{\nu-1} \qquad \text{for } \nu = 1,2, \ldots, r.
\end{equation}

Since $A = \eta SMS$ and the matrix $SMS$ is similar to $M$, the eigenvalues of $M$ are given by $\mu_\nu = \eta \lambda_\nu$ for $\nu = 1, \ldots, r$ under the same ordering by moduli. In the aforementioned factorization, we have $X_\nu = -\mu_\nu = -\eta \lambda_\nu$, which ensures that the signs of $X_\nu$ likewise alternate, yielding
\begin{equation}
    X_\nu X_{\nu+1} < 0 \qquad \text{for } \nu = 1, \ldots, r-1.
\end{equation}

For each \(\nu\), the equation \(z^m=X_\nu\) has \(m\) distinct roots forming a regular \(m\)-gon on the circle
\begin{equation}
    |z|=\rho_\nu,
    \qquad
    \rho_\nu=|X_\nu|^{1/m}.
\end{equation}
Since adjacent \(X_\nu\)'s have opposite signs, the corresponding regular
\(m\)-gons are rotated by \(\pi/m\) relative to each other. Moreover,
\(|X_1|>\cdots>|X_r|>0\) implies
\begin{equation}
    \rho_1>\cdots>\rho_r>0.
\end{equation}
\end{proof}

\begin{example}
We give some examples in Fig. \ref{fig:WHZeros-N8-k2-l1-m9-15} to illustrate Corollary~\ref{cor:alternating-ring-structure}.
For \(N=8\), \(k=2\), and \(l=1\), the nonzero roots of
the Yablonskii--Vorob'ev polynomial hierarchy \(W_{8}^{[m,2,1]}(z)\) lie on one or more distinct concentric circles. On each
circle, the roots form a regular \(m\)-gon, and the regular \(m\)-gons on two adjacent circles are rotated by an angle \(\pi/m\) relative to each other. 
\begin{figure}[H]
    \centering
    \includegraphics[width=0.9\linewidth]{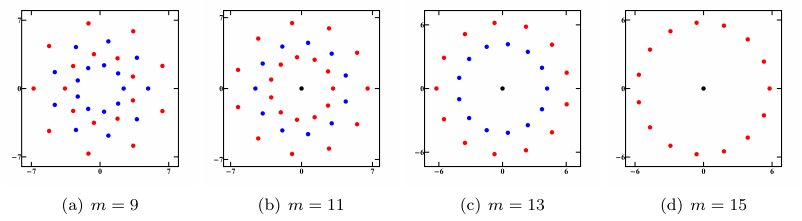}
    \caption{Zero distributions of the Yablonskii--Vorob'ev polynomial hierarchies $W_{8}^{[m,2,1]}(z)$ for $m=9,11,13,15$.}
    \label{fig:WHZeros-N8-k2-l1-m9-15}
\end{figure}

We also present another example, namely the Okamoto polynomial hierarchies $W_{10}^{[m,3,2]}(z)$, in Fig. \ref{fig:WHZeros-N10-k3-l2-m16-20}. For
\(m=16,17,19,20\), direct calculations give
\(r=5,5,4,4\), respectively. Hence, the nonzero roots lie on \(5,5,4,4\) distinct concentric circles.
\begin{figure}[H]
    \centering
    \includegraphics[width=0.9\linewidth]{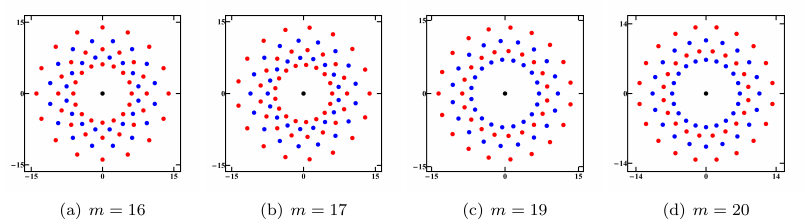}
    \caption{Zero distributions of the Okamoto polynomial hierarchies $W_{10}^{[m,3,2]}(z)$ for $m=16,17,19,20$.}
    \label{fig:WHZeros-N10-k3-l2-m16-20}
\end{figure}

\end{example}

\subsection{The case \texorpdfstring{\(m=2\) and \(k=4\)}{m=2 and k=4}}
\label{subsec:k4-m2-proof}
In this subsection, we prove Conjecture~\ref{conjecture:nonzero simple of WH} for the case 
\begin{equation}
    m=2,\qquad k=4,\qquad \text{and} \qquad l=1,2,3.
\end{equation}

\begin{proposition}
\label{prop:k4-W-recurrence}
For $N\geq1$ and $l \in \{1,2,3\}$, the generalized Wronskian--Hermite polynomials \(W_N^{[2,4,l]}(z)\) satisfy the recurrence relation
\begin{eqnarray}
\label{eq:k4-W-recurrence}
    W_{N+1}^{[2,4,l]}(z)W_{N-1}^{[2,4,l]}(z)&=&8z\left(W_N^{[2,4,l]}(z)\left(W_N^{[2,4,l]}\right)''(z)-\left(\left(W_N^{[2,4,l]}\right)'(z)\right)^2\right) \nonumber\\
    &+&8W_N^{[2,4,l]}(z)\left(W_N^{[2,4,l]}\right)'(z)+z\left(z^2+12N+4l-6\right)\left(W_N^{[2,4,l]}(z)\right)^2,
\end{eqnarray}
with initial polynomials
\begin{equation}
\label{eq:k4-W-initial-values}
    W_0^{[2,4,l]}(z)=1,\qquad W_1^{[2,4,1]}(z)=z,\qquad W_1^{[2,4,2]}(z)=z^2+2,\qquad W_1^{[2,4,3]}(z)=z^3+6z.
\end{equation}
\end{proposition}

\begin{proof}
Let \(l\in\{1,2,3\}\) and \(N\geq1\) be fixed. Denote 
\begin{equation}
    \Delta_N^{(l)}(z)=\operatorname{Wr} \left[p_{n_1}^{[2]}(z),p_{n_2}^{[2]}(z),\ldots,p_{n_{N-1}}^{[2]}(z),p_{n_N}^{[2]}(z)\right],\;
    \Sigma_N^{(l)}(z)=\operatorname{Wr} \left[p_{n_1}^{[2]}(z),p_{n_2}^{[2]}(z),\ldots,p_{n_{N-1}}^{[2]}(z),p_{n_{N+1}}^{[2]}(z)\right],
\end{equation}
where $n_j=l+4(j-1),\; j=1,2,\ldots,N+1$, and $$ \operatorname{Wr} [f_1,f_2, \dots, f_N] = \det \left( f_j^{(i-1)}(z) \right)_{i,j=1}^N.$$ By the Jacobi formula for determinants, we obtain
\begin{equation}\label{eqn:Jacobi formula for determinants}
    \Delta_{N+1}^{(l)}(z)\Delta_{N-1}^{(l)}(z)
    =
    \Delta_N^{(l)}(z)\left(\Sigma_N^{(l)}\right)'(z)
    -
    \left(\Delta_N^{(l)}\right)'(z)\Sigma_N^{(l)}(z).
\end{equation}
Furthermore, recalling the relation $W_N^{[2,4,l]}(z)=c_N^{[2,4,l]}\Delta_N^{(l)}(z)$ and multiplying both sides by \(c_{N+1}^{[2,4,l]}c_{N-1}^{[2,4,l]}\), we obtain
\begin{equation}
\label{eq:k4-normalized-Jacobi-corrected}
    W_{N+1}^{[2,4,l]}(z)W_{N-1}^{[2,4,l]}(z)
    =
    W_N^{[2,4,l]}(z)\left(\widetilde{\Sigma}_N^{(l)}\right)'(z)
    -
    \left(W_N^{[2,4,l]}\right)'(z)\widetilde{\Sigma}_N^{(l)}(z),
\end{equation}
where
\begin{equation}
    \widetilde{\Sigma}_N^{(l)}(z)
    =
    \frac{c_{N+1}^{[2,4,l]}c_{N-1}^{[2,4,l]}}{c_N^{[2,4,l]}}\Sigma_N^{(l)}(z) = \frac{(n_N+1)(n_N+2)(n_N+3)(n_N+4)}{4N} c_N^{[2,4,l]}\Sigma_N^{(l)}(z).
\end{equation}

From the definition of $p_n^{[2]}(z)$ from \eqref{eq:p-generating-intro}, its differential and recurrence relations are governed by
\begin{equation}
    \left(p_n^{[2]}\right)'(z)=p_{n-1}^{[2]}(z),
    \qquad
    (n+1)p_{n+1}^{[2]}(z)=zp_n^{[2]}(z)+2\left(p_n^{[2]}\right)'(z).
\end{equation}
Consequently, we deduce that
\begin{equation}
    z\left(p_n^{[2]}\right)'(z)=n p_n^{[2]}(z)-2p_{n-2}^{[2]}(z),
\end{equation}
and
\begin{equation}
    z^2p_n^{[2]}(z)=(n+1)(n+2)p_{n+2}^{[2]}(z)-(4n+2)p_n^{[2]}(z)+4p_{n-2}^{[2]}(z).
\end{equation}
Applying the latter identity repeatedly yields
\begin{equation}
\begin{aligned}
    z^4p_n^{[2]}(z)
    ={}&(n+1)(n+2)(n+3)(n+4)p_{n+4}^{[2]}(z)-4(n+1)(n+2)(2n+3)p_{n+2}^{[2]}(z) \\
    &+12(2n^2+2n+1)p_n^{[2]}(z)
      -16(2n-1)p_{n-2}^{[2]}(z)+16p_{n-4}^{[2]}(z).
\end{aligned}
\end{equation}
Combining the last two formulas, we get
\begin{equation}
\label{eq:k4-one-step-column-identity}
    \left(32z\frac{d}{dz}+z^4+(8n+12)z^2+8n^2+8n+12\right)p_n^{[2]}(z)
    =(n+1)(n+2)(n+3)(n+4)p_{n+4}^{[2]}(z)+16p_{n-4}^{[2]}(z).
\end{equation}

We now define the auxiliary sum of determinants $\mathcal{S}_N^{(l)}(z)$ as
\begin{equation}
\label{eq:k4-SN-definition}
    \mathcal S_N^{(l)}(z)=\sum_{j=1}^{N}\operatorname{Wr} \left[p_{n_1}^{[2]},\ldots,p_{n_{j-1}}^{[2]},\left(32z\frac{d}{dz}+z^4+(8n_j+12)z^2+8n_j^2+8n_j+12\right)p_{n_j}^{[2]},p_{n_{j+1}}^{[2]},\ldots,p_{n_N}^{[2]}\right].
\end{equation}
On the one hand, substituting \eqref{eq:k4-one-step-column-identity} with $n=n_j$, the replaced column in the $j$-th Wronskian takes the form
\begin{equation}
    (n_j+1)(n_j+2)(n_j+3)(n_j+4)p_{n_j+4}^{[2]}(z)+16p_{n_j-4}^{[2]}(z).
\end{equation}
For \(1\leq j\leq N-1\), we have $p_{n_j+4}^{[2]}=p_{n_{j+1}}^{[2]}$, which implies that the first term creates two identical columns, causing the corresponding Wronskian to vanish. Similarly, for \(2\leq j\leq N\), the relation $p_{n_j-4}^{[2]}=p_{n_{j-1}}^{[2]}$
ensures that the second term also yields two identical columns. Finally, for $j=1$, we have
\begin{equation}
    p_{n_1-4}^{[2]}=p_{l-4}^{[2]}=0,
    \qquad
    l=1,2,3.
\end{equation}
As a result, only the term  \(p_{n_N+4}^{[2]}=p_{n_{N+1}}^{[2]}\) in the last summand survives, which gives
\begin{equation}
\label{eq:k4-SN-to-Sigma}
    \mathcal S_N^{(l)}(z)
    =
    (n_N+1)(n_N+2)(n_N+3)(n_N+4)\Sigma_N^{(l)}(z).
\end{equation}

On the other hand, we can express \(\mathcal S_N^{(l)}(z)\) in terms of \(\Delta_N^{(l)}(z)\) and its derivative. 
To facilitate this calculation, we denote
\begin{equation}
    B_j(z)=z^4+(8n_j+12)z^2+8n_j^2+8n_j+12.
\end{equation}
Applying the Leibniz rule, for  \(r=0,1,\ldots,N-1\), the $(r+1)$-th row of the \(j\)-th replaced column in \(\mathcal S_N^{(l)}\) can be expanded as
\begin{equation}\label{eqn:expantion of the j-th column}
\begin{aligned}
    \frac{d^r}{dz^r}\left[32z\left(p_{n_j}^{[2]}\right)' +B_j p_{n_j}^{[2]} \right] &=
    32z\left(p_{n_j}^{[2]}\right)^{(r+1)} 
    +32r\left(p_{n_j}^{[2]}\right)^{(r)} 
    +B_j \left(p_{n_j}^{[2]}\right)^{(r)} +rB_j' \left(p_{n_j}^{[2]}\right)^{(r-1)} 
     \\
    &+\binom{r}{2}B_j'' \left(p_{n_j}^{[2]}\right)^{(r-2)} 
    +\binom{r}{3}B_j^{(3)} \left(p_{n_j}^{[2]}\right)^{(r-3)} 
    +\binom{r}{4}B_j^{(4)} \left(p_{n_j}^{[2]}\right)^{(r-4)},
\end{aligned}
\end{equation}
where, for \(r<q\), the term containing 
\(\left(p_{n_j}^{[2]}\right)^{(r-q)}(z)\) is understood to be zero, and the derivatives of $B_j(z)$ are given by
\begin{equation}
    B_j'(z)=4z^3+(16n_j+24)z,\quad
    B_j''(z)=12z^2+16n_j+24,\quad
    B_j^{(3)}(z)=24z,\quad B_j^{(4)}(z)=24.
\end{equation}

For the subsequent determinant manipulations, given a column vector \(\mathbf{a}=(a_0,\ldots,a_{N-1})^T\), we denote by
\(\Delta_N^{(l)}[j;\mathbf{a}]\) the determinant obtained from \(\Delta_N^{(l)}\) by replacing its \(j\)-th column with \(\mathbf{a}\). Differentiating the Wronskian \(\Delta_N^{(l)}(z)\) yields 
\begin{equation}
\label{eq:k4-column-derivative-rule}
    \left(\Delta_N^{(l)}\right)'(z)=\sum_{j=1}^{N}
    \Delta_N^{(l)}\!\left[j;\left(\left(p_{n_j}^{[2]}\right)^{(r+1)}(z)\right)_{r=0}^{N-1}\right],
\end{equation}
where we use the convention
\begin{equation}
    \left(\alpha\left(r\right)\right)_{r=0}^{N-1}= \left(\alpha(0),\alpha(1),\cdots,\alpha(N-1)\right)^T.
\end{equation}
If the column vector \(\mathbf{a}\) does not depend on \(j\), then we have
\begin{equation}
\label{eq:k4-diagonal-row-rule}
    \sum_{j=1}^{N}
    \Delta_N^{(l)}\!\left[j;\left(a_r\left(p_{n_j}^{[2]}\right)^{(r)}(z)\right)_{r=0}^{N-1}\right]
    =
    \left(\sum_{r=0}^{N-1}a_r\right)\Delta_N^{(l)}(z).
\end{equation}
Similarly, if the column vector \(\mathbf{a}\) does not depend on \(j\), for \(s\geq1\), we get
\begin{equation}
\label{eq:k4-lower-row-rule}
    \sum_{j=1}^{N}
    \Delta_N^{(l)}\!\left[j;\left(a_r\left(p_{n_j}^{[2]}\right)^{(r-s)}(z)\right)_{r=0}^{N-1}\right]
    =
    0,
\end{equation}
where the corresponding element of the
replacement column is understood to be $0$ for $r<s$. Moreover, we also have
\begin{equation}
\label{eq:k4-r-upper-row-rule}
    \sum_{j=1}^{N}
    \Delta_N^{(l)}\!\left[j;\left(r\left(p_{n_j}^{[2]}\right)^{(r+1)}(z)\right)_{r=0}^{N-1}\right]
    =
    (N-1)\left(\Delta_N^{(l)}\right)'(z).
\end{equation}
Furthermore, from the relation
\begin{equation}
    2\left(p_n^{[2]}\right)''(z)+z\left(p_n^{[2]}\right)'(z)-np_n^{[2]}(z)=0,
\end{equation}
we get
\begin{equation}
\label{eq:k4-differentiated-Hermite-corrected}
    2\left(p_n^{[2]}\right)^{(s+2)}(z)
    +z\left(p_n^{[2]}\right)^{(s+1)}(z)
    +(s-n)\left(p_n^{[2]}\right)^{(s)}(z)
    =
    0.
\end{equation}
Taking \(n=n_j\) and \(s=r-1\), for \(r\geq1\), gives
\begin{equation}
    r n_j\left(p_{n_j}^{[2]}\right)^{(r-1)}(z)
    =
    2r\left(p_{n_j}^{[2]}\right)^{(r+1)}(z)
    +rz\left(p_{n_j}^{[2]}\right)^{(r)}(z)
    +r(r-1)\left(p_{n_j}^{[2]}\right)^{(r-1)}(z).
\end{equation}
Using \eqref{eq:k4-diagonal-row-rule}-\eqref{eq:k4-r-upper-row-rule}, we obtain
\begin{equation}
\label{eq:k4-nj-first-lower-shift}
\sum_{j=1}^{N}
    \Delta_N^{(l)}\!\left[j;\left(rn_j\left(p_{n_j}^{[2]}\right)^{(r-1)}(z)\right)_{r=0}^{N-1}\right]
    =
    2(N-1)\left(\Delta_N^{(l)}\right)'(z)
    +\frac{N(N-1)}{2}z\Delta_N^{(l)}(z).
\end{equation}
Similarly, taking \(n=n_j\) and \(s=r-2\), for \(r\geq2\), we have
\begin{equation}
    n_j\left(p_{n_j}^{[2]}\right)^{(r-2)}(z)
    =
    2\left(p_{n_j}^{[2]}\right)^{(r)}(z)
    +z\left(p_{n_j}^{[2]}\right)^{(r-1)}(z)
    +(r-2)\left(p_{n_j}^{[2]}\right)^{(r-2)}(z),
\end{equation}
which implies that
\begin{equation}
\label{eq:k4-nj-second-lower-shift}
\begin{aligned}
     \sum_{j=1}^{N}
    \Delta_N^{(l)}\!\left[j;\left(\binom{r}{2}n_j\left(p_{n_j}^{[2]}\right)^{(r-2)}(z)\right)_{r=0}^{N-1}\right]=
    2\sum_{r=0}^{N-1}\binom{r}{2}\Delta_N^{(l)}(z)
    =
    \frac{N(N-1)(N-2)}{3}\Delta_N^{(l)}(z).
\end{aligned}
\end{equation}
Thus, we can collect all contributions in the expansion of \(\mathcal S_N^{(l)}(z)\) by using \eqref{eqn:expantion of the j-th column}, which results in
\begin{equation}
\label{eq:k4-SN-collected-before-sums}
    \mathcal S_N^{(l)}(z)
    =
    32Nz\left(\Delta_N^{(l)}\right)'(z)
    +
    \Bigg[
        \sum_{j=1}^{N}B_j(z)
        +8N(N-1)z^2
        +16N(N-1)
        +\frac{16N(N-1)(N-2)}{3}
    \Bigg]\Delta_N^{(l)}(z),
\end{equation}
where
\begin{equation}
\begin{aligned}
    \sum_{j=1}^{N}B_j(z)
    =
    Nz^4
    +\left(8\sum_{j=1}^{N}n_j+12N\right)z^2
    +8\sum_{j=1}^{N}n_j^2
    +8\sum_{j=1}^{N}n_j
    +12N.
\end{aligned}
\end{equation}
Since
\begin{equation}
    \sum_{j=1}^{N}n_j
    =
    \sum_{j=1}^{N}\left(l+4(j-1)\right)
    =
    N(l+2N-2),
\end{equation}
and
\begin{equation}
    \sum_{j=1}^{N}n_j^2
    =
    \sum_{j=1}^{N}\left(l+4(j-1)\right)^2=
    \frac{N}{3}\left(16N^2+(12l-24)N+3l^2-12l+8\right),
\end{equation}
substituting these sums into \eqref{eq:k4-SN-collected-before-sums}, we obtain
\begin{equation}
\label{eq:k4-SN-final}
\begin{aligned}
    \mathcal S_N^{(l)}(z)
    =
    32Nz\left(\Delta_N^{(l)}\right)'(z)
    +
    \Big[
        Nz^4
        +4N(6N+2l-3)z^2
        +4N\left(12N^2+(8l-12)N+2l^2-6l+3\right)
    \Big]\Delta_N^{(l)}(z).
\end{aligned}
\end{equation}

Finally, combining \eqref{eq:k4-SN-to-Sigma} and \eqref{eq:k4-SN-final}, we have
\begin{equation}
\label{eq:k4-Sigma-tilde-corrected}
    \widetilde{\Sigma}_N^{(l)}(z)
    =
    8z\left(W_N^{[2,4,l]}\right)'(z)
    +
    \left[
        \frac{z^4}{4}
        +(6N+2l-3)z^2
        +12N^2+(8l-12)N+2l^2-6l+3
    \right]W_N^{[2,4,l]}(z),
\end{equation}
and
\begin{equation}\label{eq:k4-Sigma-tilde-Prime-corrected}
\begin{aligned}
    \left(\widetilde{\Sigma}_N^{(l)}\right)'(z)
    =
    8\left(W_N^{[2,4,l]}\right)'(z)
    +8z\left(W_N^{[2,4,l]}\right)''(z)
    +\left(z^3+2(6N+2l-3)z\right)W_N^{[2,4,l]}(z) \\
    +
    \left[
        \frac{z^4}{4}
        +(6N+2l-3)z^2
        +12N^2+(8l-12)N+2l^2-6l+3
    \right]\left(W_N^{[2,4,l]}\right)'(z).
\end{aligned}
\end{equation}
Substituting \eqref{eq:k4-Sigma-tilde-corrected} and \eqref{eq:k4-Sigma-tilde-Prime-corrected} into \eqref{eq:k4-normalized-Jacobi-corrected}, we obtain \eqref{eq:k4-W-recurrence}.
This completes the proof.
\end{proof}

\begin{proposition}
\label{prop:k4-adjacent-Hirota}
For \(l=1,2,3\) and \(N\ge1\), we have
\begin{equation}\label{eq:k4-adjacent-Hirota}
    \left(D_z^2-\frac{z}{2}D_z-\frac{l+3N}{2}\right)W_N^{[2,4,l]}(z)\cdot W_{N+1}^{[2,4,l]}(z)=0,
\end{equation} 
where $D_z$ is the Hirota operator.
\end{proposition}

\begin{proof}
The conclusion follows immediately from Proposition~\ref{prop:single-block-WH-adjacent-bilinear} whose proof is provided in Section \ref{sec:WH with Painleve}.
\end{proof}

\begin{theorem}\label{thm:k4-m2-simple-nonzero}
For \(l=1,2,3\) and \(N\ge1\), all nonzero roots of \(W_N^{[2,4,l]}(z)\) are simple.
\end{theorem}

\begin{proof}
Fix \(l\in\{1,2,3\}\). For brevity, we denote
\begin{equation}
    W_N(z)=W_N^{[2,4,l]}(z).
\end{equation}
For \(N=1\), from \eqref{eq:k4-W-initial-values}, we have
\begin{equation}
    W_1^{[2,4,1]}(z)=z,\qquad
    W_1^{[2,4,2]}(z)=z^2+2,\qquad
    W_1^{[2,4,3]}(z)=z^3+6z.
\end{equation}
It is clear that all nonzero roots of \(W_1(z)\) are simple.
Assume that, for some \(N\ge1\), all nonzero roots of \(W_N(z)\) are simple. We prove that all nonzero roots of \(W_{N+1}(z)\) are simple by induction.

Suppose that \(z_0\neq0\) is a multiple root of \(W_{N+1}(z)\), then there exist an integer \(s\geq2\) and a polynomial \(h(z)\) such that
\begin{equation}
\label{eq:k4-WNplus1-local-factor}
    W_{N+1}(z)=(z-z_0)^s h(z),
    \qquad
    h(z_0)\neq0,
\end{equation}
which implies
\begin{equation}
\label{eq:k4-WNplus1-root-derivative}
    W_{N+1}(z_0)=0,
    \qquad
    W_{N+1}'(z_0)=0.
\end{equation}

By \eqref{eq:k4-adjacent-Hirota} in Proposition \ref{prop:k4-adjacent-Hirota}, we have
\begin{equation}
    \left(D_z^2-\frac{z}{2}D_z-\frac{l+3N}{2}\right)W_N(z)\cdot W_{N+1}(z)=0.
\end{equation}
Expanding the Hirota operators gives
\begin{equation}
\label{eq:k4-Hirota-expanded-direct}
\begin{aligned}
    &W_N''(z)W_{N+1}(z)-2W_N'(z)W_{N+1}'(z)+W_N(z)W_{N+1}''(z) \\
    &\quad
    -\frac{z}{2}\left(W_N'(z)W_{N+1}(z)-W_N(z)W_{N+1}'(z)\right)
    -\frac{l+3N}{2}W_N(z)W_{N+1}(z)=0.
\end{aligned}
\end{equation}

We claim that $W_N(z_0) = 0$. Assume for contradiction that $W_N(z_0) \neq 0$. Since $W_N(z)$ is analytic at $z=z_0$, 
the local expansion of  $W_N(z)W_{N+1}''(z)$ around $z=z_0$ is 
\begin{equation}
    W_N(z)W_{N+1}''(z) = s(s-1)W_N(z_0)h(z_0)(z-z_0)^{s-2} + O\left((z-z_0)^{s-1}\right).
\end{equation}
As $s \ge 2$, $W_N(z_0) \neq 0$ and $h(z_0) \neq 0$, the coefficient of $(z-z_0)^{s-2}$ is nonzero. However, direct calculation reveals that all other terms in \eqref{eq:k4-Hirota-expanded-direct} are $O\left((z-z_0)^{s-1}\right)$.
This is a contradiction. Hence, we must have
\begin{equation}
    W_N(z_0) = 0.
\end{equation}

Finally, evaluating the recurrence relation \eqref{eq:k4-W-recurrence} in Proposition \ref{prop:k4-W-recurrence} at \(z=z_0\), and using
\begin{equation}
    W_{N+1}(z_0)=0,
    \qquad
    W_N(z_0)=0,
\end{equation}
we obtain
\begin{equation}
    -8z_0\left(W_N'(z_0)\right)^2=0.
\end{equation}
Since \(z_0\neq0\), it follows that
\begin{equation}
    W_N'(z_0)=0.
\end{equation}
Thus, \(z_0\) is a multiple nonzero root of \(W_N(z)\), contradicting the induction hypothesis.
Therefore, \(W_{N+1}(z)\) has no multiple nonzero roots. 

By induction, all nonzero roots of \(W_N^{[2,4,l]}(z)\) are simple for \(N\geq1\). 
\end{proof}

\section{Generalized Wronskian--Hermite polynomials associated with Young diagrams}\label{sec:WH with partition}
In the previous section, we studied the nonzero roots of the generalized
Wronskian--Hermite polynomials. In this section, we focus on the multiplicity of the root of
\(W_N^{[m,k,l]}(z)\) at the origin. More precisely, we compare two
descriptions of this multiplicity: the algebraic description in terms of the
integers \(N_i\) in \cite{lin2024rogue}, and the Young-diagram description in terms of the
\(m\)-core of the associated partition in \cite{bonneux2020coefficients}. We show that these two theorems
are equivalent by identifying the integers \(N_i\) with the bead numbers on
the nonzero runners of the \(k\)-abacus of \(\mathrm{core}_m(Y)\).

\subsection{Young diagrams, \texorpdfstring{\(m\)-cores}{m-cores}, and the Schur-function representation}
\label{subsec:young-diagrams-mcores-schur}
We first recall the basic notions that will be used to describe the multiplicity of zero roots of \(W_N^{[m,k,l]}(z)\).

\begin{definition}[Integer partition]
A partition of a positive integer $N$ is a decreasing sequence 
of positive integers $\lambda = (\lambda_1,\lambda_2,\ldots,\lambda_\ell)$ such that
\begin{equation}
\lambda_1 \ge \lambda_2 \ge \cdots \ge \lambda_\ell > 0,
\qquad
|\lambda| := \sum_{i=1}^\ell \lambda_i = N.
\end{equation}
By convention, the only partition of zero is the empty partition
\(\varnothing\), with length \(\ell(\varnothing)=0\).
\end{definition}
A partition $\lambda=(\lambda_1,\ldots,\lambda_\ell)$ can be represented graphically
by its \emph{Young diagram} (or Ferrers diagram), denoted by $[\lambda]$.
It consists of $\ell$ left-justified rows of boxes, where the $k$-th ($1\le k \le \ell$) row contains
$\lambda_k$ boxes.

\begin{definition}[Hook]
Let \(x\) be a box in a Young diagram. The hook of \(x\) is the set consisting of \(x\) itself, together with all boxes to the right of \(x\) in the same row and all boxes below \(x\) in the same column. The number of boxes in this set is called the hook length of \(x\). A hook with hook length \(m\) is called an \(m\)-hook. The rightmost box in the same row of the hook is called its hand, and the lowest box in the same column of the hook is called its foot.
\end{definition}

\begin{definition}[Rim \(m\)-hook, \cite{james2006representation}]\label{def:rim-m-hook}
Let \(H\) be an \(m\)-hook of a Young diagram. The corresponding rim \(m\)-hook is the part of the rim of the Young diagram which is of the same length as \(H\). It consists of the boxes on the rim between the hand and the foot of \(H\), including the hand and the foot. Removing such a rim \(m\)-hook from the Young diagram again yields a Young diagram.
\end{definition}

\begin{definition}[$m$-core, \cite{james2006representation}]\label{def:core-of-lambda}
A Young diagram  is called an $m$-core if it contains no $m$-hooks.
\end{definition}

\begin{remark}
The \(m\)-core of a partition \(\lambda\) can be obtained as follows. Repeatedly remove a rim \(m\)-hook from the Young diagram of \(\lambda\) in such a way that the remaining diagram is still a partition. When no further rim \(m\)-hooks can be removed, the resulting partition is called the \(m\)-core of \(\lambda\). In this paper, we denote the \(m\)-core of \(\lambda\) by \(\mathrm{core}_m(\lambda)\). Moreover, for any partition, the \(m\)-core is uniquely determined (see Theorem~2.7.16 in \cite{james2006representation}).
\end{remark}

\begin{remark}[Computing the $m$-core of a partition via $\beta$-numbers and the $m$-abacus, \cite{james2006representation}]\label{rem:compute-m-core}
Let $\lambda=(\lambda_1,\lambda_2,\ldots,\lambda_\ell)$ be a partition and let $m\ge 2$ be an integer.
The $m$-core $\mathrm{core}_m(\lambda)$ can be obtained by the following abacus algorithm, which is equivalent to
successively removing rim $m$-hooks.

\begin{enumerate}[(Step 1)]
\item Choose $r$ and form the $\beta$-numbers.
Choose any integer $r\ge \ell(\lambda)$ and extend $\lambda$ by zeros to length $r$ (i.e., set $\lambda_i=0$ for $i>\ell(\lambda)$), where $\ell(\lambda)$ denotes the length of the partition $\lambda$.
Define the $\beta$-numbers
\begin{equation}\label{eq:beta-numbers-abacus}
\beta_i=\lambda_i+r-i,\qquad i=1,2,\ldots,r.
\end{equation}
Throughout this paper, we take $r=\ell(\lambda)$ unless otherwise specified.

\item Place beads on the $m$-abacus.
Consider $m$ runners indexed by $0,1,\ldots,m-1$.
For each $\beta_i$, place a bead at position $\beta_i$ on the runner labeled by $\beta_i \bmod m$.
Equivalently, runner $a$ contains positions
\[
a,\ a+m,\ a+2m,\ \ldots \qquad (a=0,1,\ldots,m-1).
\]

\item Slide beads upward.
On each runner, repeatedly apply the move
\[
\beta \longmapsto \beta-m
\]
whenever the position $\beta-m\ge 0$ on the same runner is empty.
Continue until no bead can be moved upward on any runner.
Denote the resulting set of positions by $\beta^{(\infty)}_1 > \beta^{(\infty)}_2 > \cdots > \beta^{(\infty)}_r.$

\item Read off the $m$-core from the final $\beta$-numbers.
Define
\begin{equation}
\tilde{\lambda}^{(\infty)}_i=\beta^{(\infty)}_i-r+i,\qquad i=1,2,\ldots,r.
\end{equation}
Let \(\lambda^{(\infty)}\) be the partition consisting of all positive parts among
\(\tilde{\lambda}^{(\infty)}_1,\tilde{\lambda}^{(\infty)}_2,\ldots,\tilde{\lambda}^{(\infty)}_r\).
Then
\begin{equation}\label{eq:core-from-final-beta}
    \mathrm{core}_m(\lambda)=\lambda^{(\infty)}.
\end{equation}
\end{enumerate}

\end{remark}

\begin{remark}
Each successful upward move of a bead corresponds to removing one rim $m$-hook from the Young diagram of $\lambda$.
Hence the total number of moves equals the $m$-weight
\begin{equation}
    w_m(\lambda)=\frac{|\lambda|-|\mathrm{core}_m(\lambda)|}{m}.
\end{equation}
\end{remark}

For an infinite vector $\mathbf{x}=(x_1,x_2,\ldots)$ and a Young diagram $Y=(i_1,i_2,\ldots,i_\ell)$, the
corresponding Schur function $s_Y(\mathbf{x})$ is defined by the following determinant \cite{macdonald1998symmetric}
\begin{equation}\label{eq:schur-jacobi-trudi}
s_Y(\mathbf{x})
=\det_{1\le j,k\le \ell}\!\Bigl[\,S_{\,i_j-j+k}(\mathbf{x})\,\Bigr],
\end{equation}
where $S_n(\mathbf{x})$ denotes the elementary Schur polynomials \cite{yang2023rogue}. If we choose
\begin{equation}\label{eq:young-diagram-Y}
Y=\left((N-1)(k-1)+l,\; (N-2)(k-1)+l,\; \ldots,\; (k-1)+l,\; l\right),
\end{equation}
then the generalized Wronskian--Hermite polynomials $W_N^{[m,k,l]}(z)$ can be expressed as \cite{clarkson2003second,yang2023rogue}
\begin{equation}\label{eq:WH-as-schur}
W_N^{[m,k,l]}(z)
=\frac{s_Y\!\bigl(z,0,\ldots,0,1_m,0,\ldots\bigr)}{s_Y(1,0,0,\ldots)},
\end{equation}
where $1_m$ indicates that the entry $1$ is placed in the $m$-th position of the vector, while all other entries
(except the first entry $z$) are set to zero.
Note that $Y$ given in \eqref{eq:young-diagram-Y} is a $k$-core partition, since its $\beta$-numbers are  (see Corollary $1.2$ in \cite{olsson2009core})
\begin{equation}
    \beta(Y)=\{\,l+jk: j=0,1,\dots,N-1\,\}.
\end{equation}

\subsection{Two theorems on the multiplicity of the zero root}
We first recall the algebraic description of the root structure of \(W_N^{[m,k,l]}(z)\). In this description, the multiplicity of the zero root
is encoded by an explicitly defined integer \(\Gamma_0\).

\begin{theorem}[\cite{lin2024rogue}, Theorem $2$]\label{lemma:root structure of WH polynomials}
Let \(m,k,l\) be positive integers satisfying $m > 1$, $k \nmid m$, and $1 \leq l < k$, the generalized Wronskian--Hermite polynomials $W_N^{[m,k,l]}(z)$ in \eqref{eqn:Wronskian-Hermite-poly-with-jump-k} are monic with degree
\begin{equation} \label{degree Gamma}
    \Gamma = \frac{N}{2} \left( (k - 1)(N - 1) + 2l \right),
\end{equation}
and can be expressed as
\begin{equation} \label{W_N structure}
    W_N^{[m,k,l]}(z) = z^{\Gamma_0} \widehat{W}_N^{[m,k,l]}(\hat{z}),
\end{equation}
where $\hat{z} = z^m$, $\widehat{W}_N^{[m,k,l]}(\hat{z})$ is a monic polynomial with real coefficients and $\widehat{W}_N^{[m,k,l]}(0)\neq0$. Here, $\Gamma_0$ is a nonnegative integer given by
\begin{equation} \label{Gamma_0 formula}
    \Gamma_0 = \sum_{i=1}^{k-1} \frac{N_i}{2} \left( (k - 1)(N_i - 1) + 2i \right) - \sum_{1 \leq i < j \leq k-1} N_i N_j,
\end{equation}
where the integers $N_i$ $(1\le i\le k-1)$ are determined by an explicit arithmetic procedure.
For completeness, the calculations of $N_i$ are summarized in \ref{app:Ni-procedure}.
\end{theorem}

The next theorem gives a different description
of the same multiplicity, but in terms of the \(m\)-core of the associated Young diagram.

\begin{theorem}[\cite{bonneux2020coefficients}, Theorem $6$]\label{thm:WH-mcore-factorization}
Fix $m\ge2$ and let $W_N^{[m,k,l]}(z)$ be the generalized Wronskian--Hermite polynomial defined by
\eqref{eq:WH-as-schur}, where the associated Young diagram is
\begin{equation}\label{eq:WH-young-diagram-repeat}
Y=\left((N-1)(k-1)+l,\; (N-2)(k-1)+l,\; \ldots,\; (k-1)+l,\; l\right).
\end{equation}
Then $W_N^{[m,k,l]}(z)$ has $|Y|$ zeros (counted with multiplicity), and there exists a polynomial $q_N^{[m,k,l]}(\zeta)$ such that
\begin{equation}\label{eq:WH-mcore-factorization}
W_N^{[m,k,l]}(z)=z^{|\mathrm{core}_m(Y)|}\,q_N^{[m,k,l]}(z^m),
\end{equation}
where $|X|$ denotes the size of a partition $X$ and $q_N^{[m,k,l]}(0)\neq0$.
\end{theorem}

\subsection{Connection between the two theorems on the multiplicity of the zero root}
We now establish a connection between the algebraic and Young-diagram
descriptions above. We first express the size of the \(m\)-core in terms of
the bead numbers on the \(k\)-abacus, and then show that these bead numbers
coincide with the integers \(N_i\) appearing in the algebraic formula. 
\begin{lemma}\label{lem:size-of-core-by-runner-beads}
Let $\lambda=\mathrm{core}_m(Y),$ where $Y$ is the partition associated with $W_N^{[m,k,l]}(z)$.
Represent $\lambda$ on the $k$-abacus, and for each runner $i\in\{0,1,\dots,k-1\}$, let $b_i$
denote the number of beads on runner $i$.
Assume that the $\beta$-numbers of $\lambda$ are taken with $r=\ell(\lambda),$
then 
\begin{equation}\label{eq:core-size-ni-lemma}
|\lambda|
=
\sum_{i=1}^{k-1} \frac{b_i}{2} \left( (k - 1)(b_i - 1) + 2i \right)
-\sum_{1 \le i < j \le k-1} b_i b_j.
\end{equation}
\end{lemma}

\begin{proof}
If $\lambda=\varnothing$, the formula is immediate. We therefore assume $\lambda\not=\varnothing$. Take $r=\ell(\lambda)$, then the $\beta$-numbers of $\lambda$ are defined by
\begin{equation}
    \beta_i=\lambda_i+r-i>0,\qquad i=1,\dots,r.
\end{equation}
In particular, position \(0\) is empty.
Since $Y$ is a $k$-core, it follows that $\lambda$ is also a $k$-core \cite{gramain2012core}. Consequently, in the corresponding $k$-abacus display, the beads on each runner are pushed up to their highest possible positions. In addition, runner $0$ contains no bead, so that $b_0=0$.
Hence, runner $i$ contains beads exactly at the positions
\begin{equation}
i,\ i+k,\ i+2k,\ \dots,\ i+(b_i-1)k,
\qquad i=1,\dots,k-1.
\end{equation}
Therefore, the sum of the $\beta$-numbers of $\lambda$ is
\begin{equation}
\sum_{j=1}^{r}\beta_j
=
\sum_{i=1}^{k-1}\frac{b_i}{2}\left(2i +k(b_i-1)\right).
\end{equation}

On the other hand, since
\begin{equation}
\beta_j=\lambda_j+r-j,\qquad j=1,\dots,r,
\end{equation}
we have
\begin{equation}
\sum_{j=1}^{r}\beta_j
=
|\lambda|+\frac{r(r-1)}{2}.
\end{equation}
Thus
\begin{equation} \label{eqn: calculation steps in proof for lambda sum}
|\lambda|
=
\sum_{i=1}^{k-1}\frac{b_i}{2}\left(2i +k(b_i-1)\right)-\frac{r(r-1)}{2}.
\end{equation}

Since runner $0$ is empty, we have
\begin{equation}
r=\sum_{i=1}^{k-1}b_i.
\end{equation}
Hence
\begin{equation}
\frac{r(r-1)}{2}
=
\frac{1}{2}\sum_{i=1}^{k-1}b_i^2
+\sum_{1\le i<j\le k-1}b_ib_j
-\frac{1}{2}\sum_{i=1}^{k-1}b_i.
\end{equation}
Substituting this into \eqref{eqn: calculation steps in proof for lambda sum} yields
\begin{equation}
|\lambda|
=
\sum_{i=1}^{k-1} \frac{b_i}{2} \left( (k - 1)(b_i - 1) + 2i \right)
-\sum_{1 \le i < j \le k-1} b_i b_j.
\end{equation}
\end{proof}

In what follows, we establish the connection between Theorem \ref{lemma:root structure of WH polynomials} and Theorem \ref{thm:WH-mcore-factorization}, both of which describe the root structures of $W_N^{[m,k,l]}(z)$.

\begin{theorem}\label{thm:relation between N_i and n_i}
Let \(W_N^{[m,k,l]}(z)\) be a generalized Wronskian--Hermite polynomial, and let \(Y\) be the associated partition. The \(m\)-core \(\mathrm{core}_m(Y)\) can be represented on the \(k\)-abacus by choosing the \(\beta\)-numbers with
\begin{equation}
r=\ell\bigl(\mathrm{core}_m(Y)\bigr),
\end{equation}
where $\ell(\lambda)$ denotes the length of the partition $\lambda$. 
Then the number \(b_i\) of beads on runner \(i\) agrees with the integers \(N_i\) in Theorem~\ref{lemma:root structure of WH polynomials}; namely,
\begin{equation}
b_i=N_i,\qquad i=1,\dots,k-1.
\end{equation}
\end{theorem}

\begin{proof}
Since \(Y\) is a \(k\)-core, its \(m\)-core \(\mathrm{core}_m(Y)\) is again a \(k\)-core \cite{gramain2012core}. Hence, \(\mathrm{core}_m(Y)\) is represented on the \(k\)-abacus by a bead configuration in which no bead can be moved upward. In this setting, runner \(0\) is empty, so \(b_0=0\) (see the proof of Lemma \ref{lem:size-of-core-by-runner-beads}). It remains to prove that $b_i=N_i,\; i=1,\dots,k-1.$

First, we recall the definition of \(N_i\). In the proof of Theorem~\ref{lemma:root structure of WH polynomials} (see Appendix~A of \cite{lin2024rogue}), the generalized Wronskian--Hermite polynomial \(W_N^{[m,k,l]}(z)\) is redefined as a determinant by introducing a parameter $a$. After expanding its entries and performing suitable row operations, the coefficients of the highest-power terms in \(a\) for the entries in the first column are constant multiples of certain powers of \(z\). These powers can then be divided into \(k-1\) segments, each of which is of the form
\begin{equation}
    z^i,\; z^{i+k},\; z^{i+2k},\; \cdots,
\end{equation}
where \(1\le i\le k-1\). Then, \(N_i\) is defined as the length of the segment whose first term is \(z^i\), and hence $N_i \ge 0$.

To establish the relation between \(b_i\) and \(N_i\), we describe two procedures involved: one for deriving \(\mathrm{core}_m(Y)\), and the other for obtaining the above \(k-1\) segments of powers of \(z\).

For each \(j\geq1\), let \(\alpha_j\) be the unique integer in
\(\{0,1,\ldots,m-1\}\) such that
\begin{equation}
    \alpha_j
    \equiv
    l+(j-1)k
    \pmod m.
\end{equation} 
As proved in \cite{lin2024rogue}, the set
\begin{equation}
    \mathcal A
    =
    \{\alpha_j\mid 1\leq j\leq N\}
\end{equation}
contains
at most \(\hat m\) elements, where $\hat m=m/\gcd(m,k)$.
Then, we assume that \(N_0\) is the remainder when \(N\) is divided by \(\hat{m}\), i.e.,
\begin{equation}
N = k_N \hat{m} + N_0, \quad 0 \le N_0 \le \hat{m}-1,
\end{equation}
where $k_N$ is a nonnegative integer. 

In order to determine the lowest power of $z$ in $W_N^{[m,k,l]}(z)$, one can introduce the following determinant \cite{yang2023rogue,zhang2022rogue,lin2024rogue}
\begin{equation}\label{eqn:redefine WH polynomials}
\widetilde{W}_{N}^{[m,k,l]}(z;a)
=
c_{N}^{[m,k,l]} \mathrm{Wr}\left[S_{l}^{[m]}(z;a),S_{l+k}^{[m]}(z;a), \cdots, S_{l+(N-1)k}^{[m]}(z;a)\right],
\end{equation}
where 
\begin{equation}
\sum_{j=0}^{\infty} S_j^{[m]}(z;a)\,\epsilon^j
=
\exp\!\left(z\epsilon+a\epsilon^m\right).
\end{equation}
Then, by using
\begin{equation}\label{eqn:expansion of S_j(a,z)}
S_{j}^{[m]}(z;a)
=
\sum_{i=0}^{\lfloor j/m \rfloor}
\frac{a^{i}}{i!(j-im)!}\,
z^{\,j-im},
\end{equation}
where \(\lfloor x \rfloor\) denotes the greatest integer less than or equal to \(x\), each entry of the determinant in \eqref{eqn:redefine WH polynomials} can be expanded as a polynomial in both \(a\) and \(z\). To determine the lowest power of \(z\) in \eqref{eqn:redefine WH polynomials}, it suffices to examine the highest-power term in \(a\) \cite{lin2024rogue}. 

A natural way of extracting the highest-power term in \(a\) is to retain, in each entry, only its highest-power term in \(a\), which corresponds to the lowest power of \(z\). However, in many cases, this causes two or more rows to become proportional.
One must therefore perform suitable row operations. As a result, the powers of \(z\) arising from the highest-power term of \(a\) in \eqref{eqn:redefine WH polynomials} can eventually be arranged as follows, with the constant factors in each entry omitted. For the first column, they are
{\footnotesize
\begin{equation}\label{eqn:power of z after row operation}
z^{\alpha_1},
\; \cdots, \;
z^{\alpha_{\hat m}};
\;
z^{\alpha_1+m},
\; \cdots, \;
z^{\alpha_{\hat m}+m};
\; \cdots; \;
z^{\alpha_1+(k_N-1)m},
\; \cdots, \;
z^{\alpha_{\hat m}+(k_N-1)m};
\;
z^{\alpha_1+k_N m},
\; \cdots, \;
z^{\alpha_{N_0}+k_N m},
\end{equation}}
while for the second column, they are
{\scriptsize\begin{equation}
z^{\alpha_1-1},\cdots,
z^{\alpha_{\hat m}-1};
z^{\alpha_1+m-1},
\cdots,
z^{\alpha_{\hat m}+m-1};
\cdots;
z^{\alpha_1+(k_N-1)m-1},
\cdots,
z^{\alpha_{\hat m}+(k_N-1)m-1};
z^{\alpha_1+k_N m-1},\cdots,
z^{\alpha_{N_0}+k_N m-1}.
\end{equation}}
The remaining columns have similar structures, namely, the power decreases by one from column to column. Here, the \(\alpha_j\) \((1\le j\le \hat m)\) are distinct, and if some entry is \(z^0\), then all subsequent entries in the same row are zero.

On the other hand, the associated partition of \(W_N^{[m,k,l]}(z)\) is
\begin{equation}
Y=\left((N-1)(k-1)+l,\; (N-2)(k-1)+l,\; \ldots,\; (k-1)+l,\; l\right).
\end{equation}
Since \(r=\ell(Y)=N\), one can compute \(\beta\)-numbers of \(Y\) as
\begin{equation}\label{eqn:beta number for redefine WH polynomials}
\beta=(l+(N-1)k, \; l+(N-2)k, \; \cdots l+2k,\; l+k,\; l).
\end{equation}
We place the \(\beta\)-numbers of \(Y\) on the \(m\)-abacus (see Fig. \ref{fig:m-abacus-empty} for an empty \(m\)-abacus). Initially, the bead positions are
\begin{equation}
l,\ l+k,\ l+2k,\ \dots,\ l+(N-1)k.
\end{equation}
To obtain the \(m\)-core (see Remark \ref{rem:compute-m-core}), one repeatedly slides each bead upward along its runner until no bead can be moved any further. 

\begin{figure}[htbp]
\centering
\begin{tikzpicture}[x=1.2cm,y=0.95cm]

    \draw[thick] (0,0.3)--(0,-2.7);
    \draw[thick] (1,0.3)--(1,-2.7);
    \draw[thick] (2,0.3)--(2,-2.7);
    \draw[thick] (4,0.3)--(4,-2.7);

    \node[above] at (0,0.45) {\(0\)};
    \node[above] at (1,0.45) {\(1\)};
    \node[above] at (2,0.45) {\(2\)};
    \node[above] at (3,0.45) {\(\cdots\)};
    \node[above] at (4,0.45) {\(m-1\)};

    \node[below left] at (0,0) {\scriptsize \(0\)};
    \node[below left] at (1,0) {\scriptsize \(1\)};
    \node[below left] at (2,0) {\scriptsize \(2\)};
    \node at (3,0) {\(\cdots\)};
    \node[below left] at (4,0) {\scriptsize \(m-1\)};

    \node[below left] at (0,-1) {\scriptsize \(m\)};
    \node[below left] at (1,-1) {\scriptsize \(m+1\)};
    \node[below left] at (2,-1) {\scriptsize \(m+2\)};
    \node at (3,-1) {\(\cdots\)};
    \node[below left] at (4,-1) {\scriptsize \(2m-1\)};

    \node[below left] at (0,-2) {\scriptsize \(2m\)};
    \node[below left] at (1,-2) {\scriptsize \(2m+1\)};
    \node[below left] at (2,-2) {\scriptsize \(2m+2\)};
    \node at (3,-2) {\(\cdots\)};
    \node[below left] at (4,-2) {\scriptsize \(3m-1\)};

    \draw (0,0) circle (2.2pt);
    \draw (1,0) circle (2.2pt);
    \draw (2,0) circle (2.2pt);
    \draw (4,0) circle (2.2pt);

    \draw (0,-1) circle (2.2pt);
    \draw (1,-1) circle (2.2pt);
    \draw (2,-1) circle (2.2pt);
    \draw (4,-1) circle (2.2pt);

    \draw (0,-2) circle (2.2pt);
    \draw (1,-2) circle (2.2pt);
    \draw (2,-2) circle (2.2pt);
    \draw (4,-2) circle (2.2pt);

    \node at (0,-3) {\(\vdots\)};
    \node at (1,-3) {\(\vdots\)};
    \node at (2,-3) {\(\vdots\)};
    \node at (4,-3) {\(\vdots\)};

\end{tikzpicture}
\caption{An empty \(m\)-abacus display. The positions on runner \(i\) are \(i,\, i+m,\, i+2m,\, \dots\). Each empty circle represents a bead position.}
\label{fig:m-abacus-empty}
\end{figure}
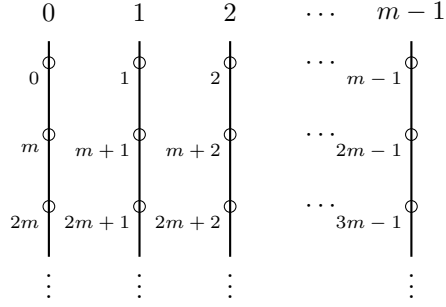

Recall that runner \(i\) consists of the positions
\begin{equation}
    i,\; i+m,\; i+2m,\; \cdots,
\end{equation}
then, a bead at position \(j\) on runner \(i\) must satisfy
\begin{equation}
    j \equiv i \pmod{m}.
\end{equation}
Hence, these beads are distributed among at most \(\hat m\) runners, namely,
\begin{equation}
    \alpha_1,\; \alpha_2,\; \alpha_3,\; \cdots,\; \alpha_{\hat m}.
\end{equation}
Note that the initial bead positions, namely the \(\beta\)-numbers in \eqref{eqn:beta number for redefine WH polynomials} are of the form
\begin{equation}
l+(j-1)k, \qquad j=1,2,\ldots,N.
\end{equation}
It follows that the final bead positions are taken from the set
\begin{equation}\label{eqn:final bead positions for redefine WH polynomials}
\left\{
\alpha_1,
\; \cdots, \;
\alpha_{\hat m}
\; ; \;
\alpha_1+m,
\; \cdots, \;
\alpha_{\hat m}+m
\; ; \; \cdots \; ; \;
\alpha_1+(k_N-1)m,
\; \cdots, \;
\alpha_{\hat m}+(k_N-1)m
\; ; \;
\alpha_1+k_Nm,
\; \cdots, \;
\alpha_{N_0}+k_Nm
\right\}.
\end{equation}

It is clear that the final bead positions coincide with the powers of \(z\) given in \eqref{eqn:power of z after row operation}. To facilitate the subsequent analysis, we rearrange the sequence in \eqref{eqn:final bead positions for redefine WH polynomials} in strictly increasing order as
\begin{equation}\label{eqn:rearrange sequence hat_beta_n}
    0\le\widehat{\beta}^{(\infty)}_1< \widehat{\beta}^{(\infty)}_2< \widehat{\beta}^{(\infty)}_3<\cdots< \widehat{\beta}^{(\infty)}_{N-1}< \widehat{\beta}^{(\infty)}_N.
\end{equation}

In the calculations of \(N_i\), the powers of \(z\) arising from the highest-power term of \(a\) in \eqref{eqn:redefine WH polynomials} can be arranged as  \cite{lin2024rogue}
\begin{equation}\label{eq:det-beta-infinity-structure}
\begin{vmatrix}
z^{\widehat{\beta}^{(\infty)}_1} & z^{\widehat{\beta}^{(\infty)}_1-1} & \cdots & z & 1 & 0 & \cdots & 0 \\
z^{\widehat{\beta}^{(\infty)}_2} & z^{\widehat{\beta}^{(\infty)}_2-1} & \cdots & z^{\widehat{\beta}^{(\infty)}_2-\widehat{\beta}^{(\infty)}_1+1} & z^{\widehat{\beta}^{(\infty)}_2-\widehat{\beta}^{(\infty)}_1} & z^{\widehat{\beta}^{(\infty)}_2-\widehat{\beta}^{(\infty)}_1-1} & \cdots & 0 \\
\vdots & \vdots & \ddots & \vdots & \vdots & \vdots & \ddots & \vdots \\
z^{\widehat{\beta}^{(\infty)}_{N}} & z^{\widehat{\beta}^{(\infty)}_{N}-1} & \cdots & z^{\widehat{\beta}^{(\infty)}_{N}-\widehat{\beta}^{(\infty)}_1+1} & z^{\widehat{\beta}^{(\infty)}_{N}-\widehat{\beta}^{(\infty)}_1} & z^{\widehat{\beta}^{(\infty)}_{N}-\widehat{\beta}^{(\infty)}_1-1} & \cdots & z^{\widehat{\beta}^{(\infty)}_{N}-N+1}
\end{vmatrix}_{N\times N},
\end{equation}
where each entry \(z^s\) is taken to be zero whenever \(s<0\). Indeed, before omitting the constant factors in each entry, the determinant corresponding to \eqref{eq:det-beta-infinity-structure} is a Wronskian. We note that this determinant can be further reduced whenever the sequence \eqref{eqn:rearrange sequence hat_beta_n} begins with an arithmetic progression of the form
\begin{equation}
0,1,2,\ldots,n-1.
\end{equation}
To illustrate this, we distinguish three cases.
\begin{itemize}
    \item If \(\widehat{\beta}^{(\infty)}_1 \neq 0\), then \eqref{eq:det-beta-infinity-structure} cannot be reduced further, and hence \eqref{eqn:rearrange sequence hat_beta_n} is already the final sequence.
    
    \item If \(\widehat{\beta}^{(\infty)}_i=i-1\) for all \(1\le i\le N\), then \eqref{eq:det-beta-infinity-structure} can be reduced completely, and consequently \eqref{eqn:rearrange sequence hat_beta_n} reduces to the empty set.
    
    \item If \(\widehat{\beta}^{(\infty)}_1=0\), and there exists an integer \(n\) such that
    \begin{equation}
        \widehat{\beta}^{(\infty)}_i=i-1,\qquad 1\le i\le n,
    \end{equation}
    while
    \begin{equation}
        n<\widehat{\beta}^{(\infty)}_{n+1}<\widehat{\beta}^{(\infty)}_{n+2}<\cdots<\widehat{\beta}^{(\infty)}_N,
    \end{equation}
    then \eqref{eq:det-beta-infinity-structure} can be reduced to a determinant of size \(\widehat{N}\times \widehat{N}\), where \(\widehat{N}=N-n\), whose first column is
    \begin{equation}
        z^{\widehat{\beta}^{(\infty)}_{n+1}-n},
        \; z^{\widehat{\beta}^{(\infty)}_{n+2}-n},
        \; \cdots,\;
        z^{\widehat{\beta}^{(\infty)}_{N}-n}.
    \end{equation}
\end{itemize}

To compute \(b_i\), let
\begin{equation}\label{eqn: beta-numbers of final positions}
\beta^{(\infty)}=
\left(\widehat{\beta}^{(\infty)}_N,\widehat{\beta}^{(\infty)}_{N-1},\ldots,\widehat{\beta}^{(\infty)}_2,\widehat{\beta}^{(\infty)}_1
\right),
\end{equation}
then we recover the $m$-core partition \(\mathrm{core}_m(Y)\) as
\begin{equation}\label{eq:standard-recover-core-lemma}
    \tilde{\lambda}^{(\infty)}_i
    =
    \widehat{\beta}^{(\infty)}_{N-i+1}-N+i,
    \qquad i=1,2,\ldots,N.
\end{equation}
Let \(\lambda^{(\infty)}\) be the partition consisting of all positive parts among
\(\tilde{\lambda}^{(\infty)}_1,\tilde{\lambda}^{(\infty)}_2,\ldots,\tilde{\lambda}^{(\infty)}_N\). Then
\begin{equation}
    \mathrm{core}_m(Y)=\lambda^{(\infty)},
\end{equation}
which shows that the \(\beta\)-numbers in \eqref{eqn: beta-numbers of final positions} are reducible whenever the sequence \eqref{eqn:rearrange sequence hat_beta_n} begins with an arithmetic progression of the form
\begin{equation}
0,1,2,\ldots,n-1.
\end{equation}
Then we distinguish the following three cases.
\begin{itemize}
    \item If \(\widehat{\beta}^{(\infty)}_1 \neq 0\), then we have
    \begin{equation}
        \tilde{\lambda}^{(\infty)}_i
        =
        \widehat{\beta}^{(\infty)}_{N-i+1}-N+i>0,
        \qquad i=1,2,\ldots,N.
    \end{equation}
    Therefore, the sequence \eqref{eqn:rearrange sequence hat_beta_n} cannot be reduced further and is already in its final form.
    \item If \(\widehat{\beta}^{(\infty)}_i=i-1\) for all \(1\le i\le N\), then
     it follows that
    \begin{equation}
            \tilde{\lambda}^{(\infty)}_i
            =
            \widehat{\beta}^{(\infty)}_{N-i+1}-N+i=0,
            \qquad i=1,2,\ldots,N.
        \end{equation}
    Therefore, the sequence \eqref{eqn:rearrange sequence hat_beta_n} can be reduced completely, and thus reduces to the empty set. In this case, $\mathrm{core}_m(Y)=\emptyset$.
    \item If \(\widehat{\beta}^{(\infty)}_1=0\), and there exists an integer \(n\) with \(1\le n < N\) such that
    \begin{equation}
        \widehat{\beta}^{(\infty)}_i=i-1,
        \qquad i=1,2,\ldots,n,
    \end{equation}
    while
    \begin{equation}
        n<\widehat{\beta}^{(\infty)}_{n+1}<\widehat{\beta}^{(\infty)}_{n+2}<\cdots<\widehat{\beta}^{(\infty)}_N,
    \end{equation}
    then, by \eqref{eq:standard-recover-core-lemma}, one has
    \begin{equation}
        \tilde{\lambda}^{(\infty)}_i
        =
        \widehat{\beta}^{(\infty)}_{N-i+1}-N+i>0,
        \qquad i=1,2,\ldots,N-n,
    \end{equation}
    and
    \begin{equation}
        \tilde{\lambda}^{(\infty)}_i
        =
        \widehat{\beta}^{(\infty)}_{N-i+1}-N+i=0,
        \qquad i=N-n+1,N-n+2,\ldots,N.
    \end{equation}
    Consequently, the sequence \eqref{eqn:rearrange sequence hat_beta_n} can be reduced by deleting the initial arithmetic progression
    \begin{equation}
        0,1,\ldots,n-1,
    \end{equation}
    and the remaining sequence
    \begin{equation}
        \left(
        \widehat{\beta}^{(\infty)}_N-n,\,
        \widehat{\beta}^{(\infty)}_{N-1}-n,\,
        \ldots,\,
        \widehat{\beta}^{(\infty)}_{n+1}-n
        \right)
    \end{equation}
    is its final form. Hence,
    \begin{equation}
        \lambda^{(\infty)}
        =
        \left(
        \widehat{\beta}^{(\infty)}_N-N+1,\,
        \widehat{\beta}^{(\infty)}_{N-1}-N+2,\,
        \ldots,\,
        \widehat{\beta}^{(\infty)}_{n+1}-n
        \right),
    \end{equation}
    and thus
    \begin{equation}
        \mathrm{core}_m(Y)=\lambda^{(\infty)}.
    \end{equation}
\end{itemize}

Hence, regardless of the initial form of the sequence \eqref{eqn:rearrange sequence hat_beta_n}, these two reduction procedures yield the same final sequence. Moreover, recall that \(\mathrm{core}_m(Y)\) is also a \(k\)-core. Therefore, the final sequence, namely the \(\beta\)-numbers of \(\mathrm{core}_m(Y)\), can be represented on the \(k\)-abacus by a bead configuration in which no bead can be moved upward, and runner \(0\) is empty. It follows that, for each \(i=1,2,\ldots,k-1\), the beads on runner \(i\) must occupy the positions
\begin{equation}
    i,\; i+k,\; i+2k,\; \ldots,\; i+(b_i-1)k,
\end{equation}
where \(b_i\ge 0\) denotes the number of beads on runner \(i\), and an empty runner corresponds to the empty set. Hence, the final \(\beta\)-numbers are exactly the disjoint union of these runner positions:
\begin{equation}\label{eqn:arithmetic progressions with first term i}
    \bigsqcup_{i=1}^{k-1}
    \left\{
    i,\; i+k,\; \ldots,\; i+(b_i-1)k
    \right\}.
\end{equation}
Therefore, the final sequence can be decomposed into at most \(k-1\) arithmetic progressions with first terms \(i\in\{1,\ldots,k-1\}\) and common difference \(k\).

On the other hand, in the determinant reduction in \cite{lin2024rogue}, the integers \(N_i\) are defined as the lengths of the segments obtained by decomposing the same final sequence into arithmetic progressions, which are the same as \eqref{eqn:arithmetic progressions with first term i}.  Hence,
\begin{equation}
    b_i=N_i,\qquad i=1,2,\ldots,k-1.
\end{equation}
This completes the proof.

\end{proof}

This theorem connects Theorem~\ref{lemma:root structure of WH polynomials} and Theorem~\ref{thm:WH-mcore-factorization}, both of which describe the root structure of \(W_N^{[m,k,l]}(z)\). Consequently, it provides an alternative way to determine the order of the rogue wave pattern at the origin in Theorem~\ref{thm:Rogue wave patterns}. On the other hand, it yields an explicit procedure for determining the \(m\)-core of the partition \(Y\).

\begin{corollary}
\label{cor:m-core-from-Ni}
Let \(Y\) be the partition defined by
\begin{equation}
Y=
\left((N-1)(k-1)+l,\; (N-2)(k-1)+l,\; \ldots,\; (k-1)+l,\; l\right),
\end{equation}
where \(1\le l<k\), \(k\ge2\),  \(m\ge2\),  and
\(k\nmid m\).
Let \(N_i\; (i=1,2,\ldots,k-1)\) be the integers obtained from the arithmetic procedure in \ref{app:Ni-procedure}. Denote $r=\sum_{i=1}^{k-1}N_i$ and
\begin{equation}
B=
\bigcup_{i=1}^{k-1}
\{\,i+jk: 0\le j<N_i\,\}.
\end{equation}
If \(r=0\), then \(B=\varnothing\) and
\begin{equation}
    \mathrm{core}_m(Y)=\varnothing.
\end{equation}
If \(r\ge1\), we can write the elements of \(B\) in decreasing order as
\begin{equation}
\beta_1>\beta_2>\cdots>\beta_r.
\end{equation}
Then, \(\mathrm{core}_m(Y)\) can be obtained from the sequence
\begin{equation}
(\beta_1-r+1,\; \beta_2-r+2,\; \ldots,\;\beta_{r-1}-1,\; \beta_r),
\end{equation}
after deleting all zeros.
\end{corollary}

\section{Generalized Wronskian--Hermite polynomials associated with Painlev\'e equations}
\label{sec:WH with Painleve}
The six Painlev\'e equations P\(_{\mathrm I}\)-P\(_{\mathrm {VI}}\) are nonlinear second-order ordinary differential equations characterized by the Painlev\'e property.
In general, their solutions are transcendental. 
However, for special parameters, Painlev\'e equations (except \(P_{\mathrm I}\)) admit algebraic or rational solutions. For example, rational solutions of \(P_{\mathrm{II}}\) are expressed in terms of the Yablonskii--Vorob'ev polynomials \cite{Vorobev1965RationalPII,Yablonskii1959RationalPII}, while those of \(P_{\mathrm{IV}}\) are described by the generalized Hermite polynomials or generalized Okamoto polynomials \cite{Okamoto1986StudiesPainleveIII,noumi1999symmetries}. Moreover, these three families of special polynomials can be represented as Wronskian--Hermite polynomials \cite{clarkson2003second,clarkson2003fourth,KajiwaraOhta1996PII,KajiwaraOhta1998PIV}. In addition, the zero distributions of the polynomials associated with rational solutions of Painlev\'e equations have also been extensively studied, especially their large-degree asymptotics \cite{BuckinghamMiller2014PII,BuckinghamMiller2015PIICritical,bertola2015zeros,buckingham2020large,BothnerMiller2020PIIIAsymptotics,BuckinghamMiller2022PIVIsomonodromy,BalogounBertola2025PVHankel}.

The Noumi--Yamada systems, introduced by Noumi and Yamada, form a family of higher-order generalizations of the Painlev\'e equations and include the symmetric forms of P\(_{\mathrm{IV}}\) and P\(_{\mathrm{V}}\) as special cases \cite{noumi1998affine,noumi2004painleve,tsuda2005universal,gomez2021complete}. In this section, we investigate the relation between the Noumi--Yamada systems and the generalized Wronskian--Hermite polynomials \(W_N^{[2,k,l]}(z)\).

\subsection{The Noumi--Yamada systems and their rational solutions}
\begin{definition}[\cite{noumi1998affine}]
\label{def:Noumi--Yamada-Ak}
Let \(k\ge3\) be an integer and let $\alpha_0, \alpha_1, \ldots, \alpha_{k-1}$ be complex constants satisfying  $\sum_{i=0}^{k-1}\alpha_i=1$.  The Noumi--Yamada system of type \(A_{k-1}^{(1)}\),
denoted by \(P(A_{k-1})\), is defined as follows.
If \(k=2g+1\) is odd, then
\begin{equation}\label{eq:PA-even}
    P(A_{2g}):\qquad
    \frac{df_i}{dx}
    =
    f_i
    \left(
        \sum_{j=1}^{g} f_{i+2j-1}
        -
        \sum_{j=1}^{g} f_{i+2j}
    \right)
    +
    \alpha_i, \quad i=0,1,2,\cdots,k-1.
\end{equation}
If \(k=2g+2\) is even, then
\begin{equation}\label{eq:PA-odd}
\begin{aligned}
    P(A_{2g+1}):\qquad
    \frac{x}{2}\frac{df_i}{dx}
    &=
    f_i
    \left(
        \sum_{1\le j\le r\le g} f_{i+2j-1}f_{i+2r}
        -
        \sum_{1\le j\le r\le g} f_{i+2j}f_{i+2r+1}
    \right) \\
    &\quad
    +
    \left(
        \frac12-\sum_{j=1}^{g}\alpha_{i+2j}
    \right)f_i
    +
    \alpha_i\sum_{j=1}^{g}f_{i+2j}, \quad i=0,1,2,\cdots,k-1.
\end{aligned}
\end{equation}
Here, \(f_i\) and \(\alpha_i\) are \(k\)-periodic chains, i.e., $f_{i+k}=f_i, \; \alpha_{i+k}=\alpha_i$. Moreover, we adopt the normalization
\begin{equation}
    \sum_{i=0}^{k-1}f_i=x.
\end{equation}

\end{definition}
Note that \(P(A_2)\) and \(P(A_3)\) are the symmetric forms of the P\(_{\mathrm {IV}}\) and P\(_{\mathrm {V}}\), respectively. Under the dependent-variable and parameter
transformations
\begin{equation}
\label{eq:f-from-tau-PAkminus1}
    f_i(x)
    =
    \frac{x}{k}
    -
    \frac{d}{dx}
    \log
    \frac{\tau_{i+1}(x)}{\tau_{i-1}(x)},
    \qquad
    \alpha_i
    =
    \frac{K_i-K_{i-1}+1}{k},
    \qquad
    i=0,\ldots,k-1,
\end{equation}
the Noumi--Yamada system \(P(A_{k-1})\) takes the bilinear form
\cite{tsuda2005universal}
\begin{equation}
\label{eq:PAkminus1-tau-bilinear}
    \left(
        D_x^2+\frac{x}{k}D_x+\frac{K_i}{k}
    \right)
    \tau_i(x)\cdot\tau_{i+1}(x)=0,
    \qquad
    i=0,\ldots,k-1,
\end{equation}
where \(D\) denotes Hirota's bilinear differential operator
\cite{hirota2004direct}, and the sequences \(\{\tau_i(x)\}_{i\in\mathbb Z}\) and
\(\{K_i\}_{i\in\mathbb Z}\) are \(k\)-periodic.

\begin{definition}[\cite{tsuda2005universal}]
A subset \(M \subseteq\mathbb Z\) is called a Maya diagram if
\begin{equation}
    m\in M \quad (m\ll0),
    \qquad \text{and} \qquad
    m\notin M \quad (m\gg0).
\end{equation}
\end{definition}
For a Maya diagram $M=\{\cdots,m_3,m_2,m_1\}$ satisfying $m_1>m_2>m_3>\cdots$, there exists a unique partition \(\lambda=(\lambda_1,\lambda_2,\cdots)\) such that $m_i-m_{i+1}=\lambda_i-\lambda_{i+1}+1$. For a vector $\mathbf a=(a_1,a_2,\ldots,a_k)\in\mathbb Z^k$,  we can define its associated Maya diagram $M(\mathbf a)$ by
\begin{equation}
\label{eq:Maya-from-charge}
    M(\mathbf a)
    =
    (k\mathbb Z_{<a_1}+1)
    \cup
    (k\mathbb Z_{<a_2}+2)
    \cup
    \cdots
    \cup
    (k\mathbb Z_{<a_k}+k),
\end{equation}
where $k\mathbb Z_{<a_r}+r =  \{kq+r:\ q<a_r,\; q\in\mathbb Z\}.$
Let $\lambda(\mathbf{a})$ denote the partition corresponding to $M(\mathbf{a})$. Then we have
\begin{equation}
    \lambda(\mathbf a+\mathbf 1)=\lambda(\mathbf a),
    \quad
    \mathbf 1=(1,1,\ldots,1).
\end{equation}
Moreover, $\lambda(\mathbf{a})$ is a \(k\)-reduced partition \cite{tsuda2005universal}, which is also called a \(k\)-core partition.

\begin{theorem}[{\cite[Theorem~5.5]{noumi2004painleve}} and {\cite[Theorem~2.4]{tsuda2005universal}}]
\label{thm:Tsuda-Schur-tau}
Let \(k\ge3\) be an integer and \(\mathbf a=(a_1,\ldots,a_k)\in\mathbb Z^k\).  Denote
\begin{equation}
    \mathbf a^{(0)}=\mathbf a, \quad \text{and} \quad 
    \mathbf a^{(i)}
    =
    (a_1+1,a_2+1,\ldots,a_i+1,a_{i+1},\ldots,a_{k-1}, a_k),
    \qquad
    1\le i\le k-1.
\end{equation}
Moreover, define
\begin{equation}
\label{eq:Tsuda-Schur-tau}
    \tau_i(x)
    =
    s_{\lambda\left(\mathbf a^{(i)}\right)}
    \left(x,-\frac{k}{2},0,0,\ldots\right),
    \qquad i=0,\ldots,k-1,
\end{equation}
where $s_{\lambda}\left(\mathbf y\right)$ denotes the Schur function defined in \eqref{eq:schur-jacobi-trudi}, and $\tau_{i+k}(x)=\tau_i(x)$.  Then,
\( \tau_i(x) \) satisfies
\begin{equation}
\label{eq:Tsuda-bilinear}
    \left(
        D_x^2+\frac{x}{k}D_x+\frac{K_i}{k}
    \right)
    \tau_i(x)\cdot\tau_{i+1}(x)=0, 
\end{equation}
where
\begin{equation}
\label{eq:Tsuda-K}
    K_i=ka_{i+1}-|\mathbf a|,
    \qquad
    |\mathbf a|=a_1+\cdots+a_k.
\end{equation}
\end{theorem}

\subsection{The \texorpdfstring{\(k\)-reduced}{k-reduced} Wronskian--Hermite polynomials}
\label{subsec:block-WH-Painleve}

We first recall the definition of the Wronskian--Hermite polynomial associated with an arbitrary
index set. For \(N=0\), we take
\(\mathbf n\) to be the empty sequence. For $N\ge1$, we define 
\begin{equation}
    \mathbf n=(n_1,n_2,\ldots,n_N)\in\mathbb{Z}^N,
    \qquad
    0<n_1<n_2<\cdots<n_N.
\end{equation}
The Wronskian--Hermite polynomial
associated with nonempty \(\mathbf n\) is defined by
\begin{equation}
\label{eq:general-index-WH}
    W_{\mathbf n}^{[2]}(z)
    =
    C_{\mathbf n}^{[2]}
    \operatorname{Wr}
    \left[
        p_{n_1}^{[2]}(z),
        p_{n_2}^{[2]}(z),
        \ldots,
        p_{n_N}^{[2]}(z)
    \right],
\end{equation}
where \(p_n^{[2]}(z)\) is defined in \eqref{eq:p-generating-intro} and
\begin{equation}
\label{eq:general-index-WH-normalization}
    C_{\mathbf n}^{[2]}
    =
    \frac{\prod_{j=1}^{N}n_j!}
    {\prod_{1\le i<j\le N}(n_j-n_i)}.
\end{equation}
When \(\mathbf n\) is empty, $W_{\mathbf n}^{[2]}(z)$ is defined to be $1$.

We now consider a special class of index sets whose associated
partitions are \(k\)-reduced. Let \(k\ge3\) and
\begin{equation}
    \mathbf N=(N_1,\ldots,N_{k-1})
    \in\mathbb Z_{\ge0}^{k-1},
    \qquad
    N=N_1+\cdots+N_{k-1}.
\end{equation}
We define
\begin{equation}
\label{eq:k-reduced-index-WH}
    I_k(\mathbf N)
    =
    \bigcup_{r=1}^{k-1}
    \{r+kq:q\in\mathbb Z,\ 0\le q<N_r\},
\end{equation}
and denote the elements of
\(I_k(\mathbf N)\) in increasing order by
\begin{equation}
    I_k(\mathbf N)
    =
    \{\tilde n_1<\tilde n_2<\cdots<\tilde n_N\}.
\end{equation}
We define the \(k\)-reduced Wronskian--Hermite polynomial $\mathcal W_{\mathbf N}^{[2,k]}(z)$ by
\begin{equation}
\label{eq:k-reduced-WH}
    \mathcal W_{\mathbf N}^{[2,k]}(z)
    =
    C_{\mathbf N}^{[2,k]}
    \operatorname{Wr}
    \left[
        p_{\tilde n_1}^{[2]}(z),
        p_{\tilde n_2}^{[2]}(z),
        \ldots,
        p_{\tilde n_N}^{[2]}(z)
    \right],
    \qquad
    C_{\mathbf N}^{[2,k]}
    =
    \frac{\prod_{j=1}^{N}\tilde n_j!}
    {\prod_{1\le i<j\le N}(\tilde n_j-\tilde n_i)} .
\end{equation}

\begin{remark}
For \(k=3\) and \(\mathbf N=(N_1,N_2)\), the index set is
\begin{equation}
    I_3(\mathbf N)
    =
    \{1+3q:q\in\mathbb Z,\ 0\le q<N_1\}
    \cup
    \{2+3q:q\in\mathbb Z,\ 0\le q<N_2\}.
\end{equation}
In this case, the \(3\)-reduced Wronskian--Hermite polynomial $\mathcal W_{(N_1,N_2)}^{[2,3]}(z)$
corresponds to the generalized Okamoto polynomial \cite{Yang2023Yang,clarkson2003fourth}.
In particular,
\begin{equation}
    \mathcal W_{(N,0)}^{[2,3]}(z)
    =
    W_N^{[2,3,1]}(z),
    \qquad
    \mathcal W_{(0,N)}^{[2,3]}(z)
    =
    W_N^{[2,3,2]}(z).
\end{equation}
More generally, let \(k\ge3\) and \(1\le l\le k-1\). If $\mathbf N=N\mathbf e_l$,
where \(\mathbf e_l\) denotes the \(l\)-th standard basis vector of
\(\mathbb Z^{k-1}\), then
\begin{equation}
    I_k(N\mathbf e_l)
    =
    \{l+kq:q\in\mathbb Z,\ 0\le q<N\}.
\end{equation}
Therefore, the generalized Wronskian--Hermite polynomial
\(W_N^{[2,k,l]}(z)\) is recovered as
\begin{equation}
\label{eq:single-block-recovery}
    W_N^{[2,k,l]}(z)
    =
    \mathcal W_{N\mathbf e_l}^{[2,k]}(z).
\end{equation}

\end{remark}

It is also useful to recall the Schur-function representation of
\(\mathcal W_{\mathbf N}^{[2,k]}(z)\).  Define the partition 
\begin{equation}
    \lambda(\mathbf N)=\left(\lambda_1(\mathbf N),\lambda_2(\mathbf N),\cdots,\lambda_N(\mathbf N)\right),
\end{equation}
by
\begin{equation}
\label{eq:k-reduced-partition-from-beta}
    \lambda_j(\mathbf N)
    =
    \tilde n_{N+1-j}-N+j,
    \qquad
    j=1,\ldots,N.
\end{equation}
Note that \(\lambda(\mathbf N)\) is \(k\)-reduced.
Using the Jacobi-Trudi formula \cite{macdonald1998symmetric},
we obtain
\begin{equation}
\label{eq:k-reduced-WH-as-Schur}
    \mathcal W_{\mathbf N}^{[2,k]}(z)
    =
    \frac{
        s_{\lambda(\mathbf N)}(z,1,0,\ldots)
    }{
        s_{\lambda(\mathbf N)}(1,0,0,\ldots)
    },
\end{equation}
where
\begin{equation}
\label{eq:Schur-Wronskian-k-reduced}
    s_{\lambda(\mathbf N)}(z,1,0,\ldots)
    =
    \operatorname{Wr}
    \left[
        p_{\tilde n_1}^{[2]}(z),
        p_{\tilde n_2}^{[2]}(z),
        \ldots,
        p_{\tilde n_N}^{[2]}(z)
    \right],
\end{equation}
and
\begin{equation}
\label{eq:Schur-denominator-k-reduced}
    s_{\lambda(\mathbf N)}(1,0,0,\ldots)
    =
    \frac{
        \prod_{1\le i<j\le N}(\tilde n_j-\tilde n_i)
    }{
        \prod_{j=1}^{N}\tilde n_j!
    }.
\end{equation}

\subsection{The \texorpdfstring{\(k\)-reduced}{k-reduced}
Wronskian--Hermite polynomials and the Noumi--Yamada system
\texorpdfstring{\(P(A_{k-1})\)}{P(A\_{k-1})}}
\label{subsec:k-reduced-WH-Painleve}
In this subsection, we relate the \(k\)-reduced Wronskian--Hermite polynomials $\mathcal W_{\mathbf N}^{[2,k]}(z)$ in
\eqref{eq:k-reduced-WH} to the bilinear form of the Noumi--Yamada system \(P(A_{k-1})\).

For \(0\le i\le k-1\), we define
\begin{equation}
\label{eq:shifted-block-vector}
    \mathbf N^{(0)}=\mathbf N,
    \qquad
    \mathbf N^{(i)}
    =
    (N_1+1,\ldots,N_i+1,N_{i+1},\ldots,N_{k-1}),
    \qquad
    1\le i\le k-1.
\end{equation}
where
\begin{equation}
    \mathbf N=(N_1,\ldots,N_{k-1})\in\mathbb Z_{\ge0}^{k-1},
    \qquad
    N=N_1+\cdots+N_{k-1}.
\end{equation}
Moreover, denote
\begin{equation}
\label{eq:k-reduced-Ti-def}
    T_i(z)
    =
    \mathcal W_{\mathbf N^{(i)}}^{[2,k]}(z),
    \qquad
    i=0,\ldots,k-1,
\end{equation}
where we extend \(T_i\) by \(T_{i+k}=T_i\).

\begin{theorem}
\label{thm:k-reduced-WH-NY-bilinear}
The \(k\)-reduced Wronskian--Hermite polynomials \(T_i(z)\) satisfy
\begin{equation}
\label{eq:k-reduced-WH-NY-z-bilinear}
    \left(
        D_z^2-\frac{z}{2}D_z-\frac{K_i}{2}
    \right)
    T_i(z)\cdot T_{i+1}(z)=0,
    \qquad
    i=0,\ldots,k-1,
\end{equation}
where 
\begin{equation}
\label{eq:k-reduced-WH-Ki}
    K_i=
    \begin{cases}
        kN_{i+1}-N+1, & 0\le i\le k-2,\\
        -N-k+1, & i=k-1.
    \end{cases}
\end{equation}
\end{theorem}

\begin{proof}
Choose $\mathbf a
    =
    (N_1,\ldots,N_{k-1},-1)\in\mathbb Z^k .$
For \(0\le i\le k-1\), define
\begin{equation}
\label{eq:shifted-charge-k-reduced-proof}
    \mathbf a^{(0)}=\mathbf a, \quad \text{and} \quad 
    \mathbf a^{(i)}
    =
    (N_1+1,\ldots,N_i+1,N_{i+1},\ldots,N_{k-1},-1).
\end{equation}
We first compute the Maya diagram of \(\mathbf a^{(i)}\) from
\eqref{eq:Maya-from-charge}.  By definition, we have
\begin{equation}
\label{eq:Maya-ai-direct-first}
    M\left(\mathbf a^{(i)}\right)
    =
    \bigcup_{r=1}^{i}
    \left(k\mathbb Z_{<N_r+1}+r\right)
    \cup
    \bigcup_{r=i+1}^{k-1}
    \left(k\mathbb Z_{<N_r}+r\right)
    \cup
    \left(k\mathbb Z_{<-1}+k\right).
\end{equation}
The right-hand side of \eqref{eq:Maya-ai-direct-first} contains all negative integers but does not contain \(0\). Therefore, we may rewrite it as
\begin{equation}
\label{eq:Maya-ai-direct-second}
    M\left(\mathbf a^{(i)}\right)
    =
    \mathbb Z_{<0}
    \cup
    \bigcup_{r=1}^{i}
    \left\{
        r+kq:
        q\in\mathbb Z,\ 0\le q<N_r+1
    \right\}\cup
    \bigcup_{r=i+1}^{k-1}
    \left\{
        r+kq:
        q\in\mathbb Z,\ 0\le q<N_r
    \right\}.
\end{equation}
Hence, by the definition of \(I_k(\mathbf N)\), we have
\begin{equation}
\label{eq:Maya-ai-block-index}
    M\left(\mathbf a^{(i)}\right)
    =
    \mathbb Z_{<0}
    \cup
    I_k\left(\mathbf N^{(i)}\right).
\end{equation}
Moreover, the partition \(\lambda\left(\mathbf N^{(i)}\right)\) associated with the index set \(I_k\left(\mathbf N^{(i)}\right)\) coincides with the partition \(\lambda\left(\mathbf a^{(i)}\right)\) associated with the Maya diagram \(M\left(\mathbf a^{(i)}\right)\). Hence, we have
\begin{equation}
\label{eq:k-reduced-WH-as-Schur-proof}
    T_i(z)
    =
    \mathcal W_{\mathbf N^{(i)}}^{[2,k]}(z)
    =
    C_{\mathbf N^{(i)}}^{[2,k]}
    s_{\lambda\left(\mathbf a^{(i)}\right)}(z,1,0,\ldots),
    \qquad
    i=0,\ldots,k-1,
\end{equation}
where \(C_{\mathbf N^{(i)}}^{[2,k]}\) is the normalization constant in
\eqref{eq:k-reduced-WH}.

We now apply Theorem~\ref{thm:Tsuda-Schur-tau} to the initial vector $\mathbf a=(N_1,\ldots,N_{k-1},-1).$
Since $|\mathbf a|=N-1,$
the constants in \eqref{eq:Tsuda-K} are
\begin{equation}
    K_i=ka_{i+1}-|\mathbf a|,
    \qquad
    i=0,\ldots,k-1.
\end{equation}
Therefore, we obtain \eqref{eq:k-reduced-WH-Ki}.

By Theorem~\ref{thm:Tsuda-Schur-tau}, the Schur functions
\begin{equation}
\label{eq:hat-tau-k-reduced-proof}
    \widehat \tau_i(x)
    =
    s_{\lambda\left(\mathbf a^{(i)}\right)}
    \left(x,-\frac{k}{2},0,\ldots\right),
    \qquad
    i=0,\ldots,k-1,
\end{equation}
extended periodically by $\widehat\tau_{i+k}(x)=\widehat\tau_i(x),$ satisfy
\begin{equation}\label{eq:hat-tau-bilinear-proof}
    \left(
        D_x^2+\frac{x}{k}D_x+\frac{K_i}{k}
    \right)
    \widehat \tau_i(x)\cdot\widehat \tau_{i+1}(x)=0,
    \qquad
    i=0,\ldots,k-1.
\end{equation}
Choose \(\rho\in\mathbb C\) such that
\begin{equation}
\label{eq:rho-k-reduced-proof}
    \rho^2=-\frac{2}{k}.
\end{equation}
By the property of Schur functions, for any partition \(\lambda\), we have
\begin{equation}
\label{eq:Schur-homogeneity-proof}
    s_{\lambda}(\rho x,1,0,\ldots)
    =
    \rho^{|\lambda|}
    s_{\lambda}\left(x,\rho^{-2},0,\ldots\right),
\end{equation}
which results in
\begin{equation}
\label{eq:Schur-scaling-rho-proof}
    s_{\lambda\left(\mathbf a^{(i)}\right)}(\rho x,1,0,\ldots)
    =
    \rho^{|\lambda\left(\mathbf a^{(i)}\right)|}
    s_{\lambda\left(\mathbf a^{(i)}\right)}
    \left(x,-\frac{k}{2},0,\ldots\right).
\end{equation}
Combining \eqref{eq:k-reduced-WH-as-Schur-proof} with
\eqref{eq:Schur-scaling-rho-proof}, we obtain
\begin{equation}
\label{eq:T-rho-hat-tau-proof}
    T_i(\rho x)
    =
    C_{\mathbf N^{(i)}}^{[2,k]}
    \rho^{|\lambda\left(\mathbf a^{(i)}\right)|}
    \widehat \tau_i(x),
    \qquad
    i=0,\ldots,k-1.
\end{equation}
Using \(z=\rho x\) and \(\rho^2=-2/k\), 
\eqref{eq:hat-tau-bilinear-proof} becomes
\begin{equation}
    \left(
        D_z^2-\frac{z}{2}D_z-\frac{K_i}{2}
    \right)
    T_i(z)\cdot T_{i+1}(z)=0.
\end{equation}
\end{proof}

Using similar arguments as above, we can find the bilinear equations satisfied by the generalized Wronskian--Hermite polynomials
\begin{equation}
    W_N^{[2,k,l]}(z) =\mathcal W_{N\mathbf e_l}^{[2,k]}(z), \quad 1\le l\le k-1.
\end{equation} 
\begin{proposition}
\label{prop:single-block-WH-adjacent-bilinear}
Let \(k\ge3\), \(1\le l\le k-1\) and \(N\ge0\).  Then, the generalized Wronskian--Hermite polynomials \(W_N^{[2,k,l]}(z)\) satisfy
\begin{equation}
\label{eq:single-block-WH-adjacent-z-bilinear}
    \left(
        D_z^2-\frac{z}{2}D_z-\frac{(k-1)N+l}{2}
    \right)
    W_N^{[2,k,l]}(z)\cdot W_{N+1}^{[2,k,l]}(z)=0.
\end{equation}
\end{proposition}

\subsection{Rational solutions to the fifth Painlev\'e equation}
\label{subsec:4-reduced-WH-PV}
 In this subsection, we consider the \(4\)-reduced Wronskian--Hermite polynomials.  In this case, the Noumi--Yamada system \(P(A_3)\) is 
the symmetric form of the fifth Painlev\'e equation 
\begin{equation}
\label{eq:PV-general-form}
\frac{d^2 y}{dt^2}
=
\left(\frac{1}{2y}+\frac{1}{y-1}\right)
\left(\frac{dy}{dt}\right)^2
-\frac{1}{t}\frac{dy}{dt}
+\frac{(y-1)^2(\alpha y^2+\beta)}{t^2y}
+\frac{\gamma y}{t}
-\frac{y(y+1)}{2(y-1)},
\end{equation}
where \(\alpha,\beta\) and \(\gamma\) are complex constants.

\begin{theorem}[{\cite{masuda2002determinant}}]
\label{thm:MOK-PV-symmetric-form}
Let \(\alpha_0,\alpha_1,\alpha_2,\alpha_3\in\mathbb C\) satisfy
\begin{equation}
    \alpha_0+\alpha_1+\alpha_2+\alpha_3=1.
\end{equation}
Let \(f_0,f_1,f_2,f_3\) be the symmetric variables of
the fifth Painlev\'e equation satisfying
\begin{equation}
\label{eq:MOK-PV-symmetric-form}
    t\frac{df_i}{dt}
    =
    f_i f_{i+2}(f_{i+1}-f_{i+3})
    +
    \left(\frac12-\alpha_{i+2}\right)f_i
    +
    \alpha_i f_{i+2},
    \qquad i=0,1,2,3,
\end{equation}
where all subscripts are understood modulo \(4\). Then
\begin{equation}
\label{eq:MOK-PV-y}
    y(t)=-\frac{f_3(t)}{f_1(t)}
\end{equation}
satisfies the fifth Painlev\'e equation \eqref{eq:PV-general-form} with
\begin{equation}
\label{eq:MOK-PV-parameter-relation}
    \alpha=\frac12\alpha_1^2,
    \qquad
    \beta=-\frac12\alpha_3^2,
    \qquad
    \gamma=\alpha_0-\alpha_2.
\end{equation}
\end{theorem}

\begin{theorem}[\cite{clarkson2024rational}]\label{thm:Classification of rational solutions of PV}
The fifth Painlev\'e equation \eqref{eq:PV-general-form} has a rational solution if and only if one of the following holds:
\begin{enumerate}
    \item[(i)] $\alpha=\frac{1}{2}p^2,\;
        \beta=-\frac{1}{2}(p+2q+1+\mu)^2,\;
        \gamma=\mu,
        \; p\geq 1$;

    \item[(ii)] $\alpha=\frac{1}{2}(p+\mu)^2,\;
        \beta=-\frac{1}{2}(q+\varepsilon\mu)^2,\;
        \gamma=p+\varepsilon q,
        \; \varepsilon=\pm1,$    provided that \(p\neq0\) or \(q\neq0\);

    \item[(iii)] $\alpha=\frac{1}{2}\left(p+\frac{1}{2}\right)^2,\;
        \beta=-\frac{1}{2}\left(q+\frac{1}{2}\right)^2,\;
        \gamma=\mu,$    provided that \(p\neq0\) or \(q\neq0\);
\end{enumerate}
where \(p,q\in\mathbb Z\), and \(\mu\) is an arbitrary complex constant, together with the solutions obtained through the symmetries
\begin{align}
    \mathcal{S}_1 : \quad & y_1\left(t;\hat  \alpha_1,\hat  \beta_1,\hat  \gamma_1\right) = y\left(-t; \alpha, \beta, \gamma\right), \quad \left(\hat \alpha_1,\hat  \beta_1, \hat \gamma_1\right) = \left(\alpha, \beta, -\gamma\right), \label{eq:sym1} \\[1em]
    \mathcal{S}_2 : \quad & y_2\left(t;\hat  \alpha_2,\hat  \beta_2,\hat  \gamma_2\right) = \frac{1}{y\left(t; \alpha, \beta, \gamma\right)}, \quad \left(\hat \alpha_2,\hat  \beta_2,\hat  \gamma_2\right) = \left(-\beta, -\alpha, -\gamma\right), \label{eq:sym2}
\end{align}
where \(y(t; \alpha, \beta, \gamma)\) is a solution of \eqref{eq:PV-general-form}.
\end{theorem}

The following proposition follows from Theorem~\ref{thm:k-reduced-WH-NY-bilinear}
with \(k=4\), the reduction from the Noumi--Yamada system \(P(A_3)\) to the
symmetric form of the fifth Painlev\'e equation in
Theorem~\ref{thm:MOK-PV-symmetric-form}, together with the change of
variables \(z=\rho x\), \(\rho^2=-1/2\), and
\(t=x^2/4=-z^2/2\). We note that the following proposition relates the \(4\)-reduced Wronskian--Hermite
polynomials with a special class of rational solutions of the fifth Painlev\'e equation.

\begin{proposition}
\label{prop:4-reduced-WH-PV-z-form}
Let \(\mathbf N=(N_1,N_2,N_3)\in\mathbb Z_{\ge0}^3\).  For
\(i=0,1,2,3\), let
\begin{equation}
    T_i(z)
    =
    \mathcal W_{\mathbf N^{(i)}}^{[2,4]}(z),
\end{equation}
where \(\mathbf N^{(i)}\) is defined in \eqref{eq:shifted-block-vector}.
Define
\begin{equation}
\label{eq:4-reduced-WH-PV-z-solution}
    W(z)
    =
    -
    \frac{
        \dfrac{z}{2}
        +
        \dfrac{d}{dz}
        \log\left(\dfrac{T_0(z)}{T_2(z)}\right)
    }{
        \dfrac{z}{2}
        -
        \dfrac{d}{dz}
        \log\left(\dfrac{T_0(z)}{T_2(z)}\right)
    }.
\end{equation}
Then \(W(z)\) is a rational solution of
\begin{equation}
\label{eq:PV-pullback-z-form}
    \frac{d^2 W}{dz^2}
    =
    \left(
        \frac{1}{2W}
        +
        \frac{1}{W-1}
    \right)
    \left(\frac{dW}{dz}\right)^2
    -
    \frac{1}{z}\frac{dW}{dz}
    +
    \frac{4(W-1)^2}{z^2W}
    \left(
        \alpha W^2+\beta
    \right)
    -
    2\gamma W
    -
    \frac{z^2W(W+1)}{2(W-1)} ,
\end{equation}
where
\begin{equation}
\label{eq:4-reduced-WH-PV-parameters}
    \alpha
    =
    \frac12
    \left(
        N_2-N_1+\frac14
    \right)^2,
    \qquad
    \beta
    =
    -\frac12
    \left(
        N_3+\frac34
    \right)^2,
    \qquad
    \gamma
    =
    N_1+N_2-N_3+1.
\end{equation}
\end{proposition}

\begin{remark}
Note that \(W(z)\) is an even rational function, and hence there exists a rational function $R$ such that
\begin{equation}
    W(z)=R(z^2).
\end{equation}
Let \(y(t)=R(-2t)\). Then \(y(t)\) is a rational solution of  the fifth Painlev\'e equation
\eqref{eq:PV-general-form}.
\end{remark}

\begin{remark}
For either choice \(\varepsilon=\pm1\), take
\begin{equation}
    p=N_1-N_3,
    \qquad
    q=\varepsilon\,(N_2+1),
    \qquad
    \mu=N_3-N_2-\frac14,
\end{equation}
then \(p,q\in\mathbb Z\).  Moreover, since \(N_2\ge0\), we have
\(q=\varepsilon(N_2+1)\neq0\). Hence, the condition \(p\neq0\) or
\(q\neq0\) in case (ii)  of Theorem~\ref{thm:Classification of rational solutions of PV} is satisfied.  Furthermore, we have
\begin{equation}
    \alpha=\frac12(p+\mu)^2
    =
    \frac12\left(N_2-N_1+\frac14\right)^2,
\end{equation}
\begin{equation}
    \beta=-\frac12(q+\varepsilon\mu)^2
    =
    -\frac12\left(N_3+\frac34\right)^2,
\end{equation}
and
\begin{equation}
    \gamma=p+\varepsilon q
    =
    N_1+N_2-N_3+1.
\end{equation}
Therefore, the parameters \eqref{eq:4-reduced-WH-PV-parameters} are contained in
case (ii) of Theorem~\ref{thm:Classification of rational solutions of PV}.
\end{remark}

\section{Rogue wave patterns associated with the generalized Wronskian--Hermite polynomials}\label{sec:WH with application}
Rogue waves are localized nonlinear waves with unusually large amplitudes that ``come from nowhere and disappear without a trace''. These characteristics imply that rogue waves may have destructive effects on ships and offshore platforms.  Since their first recorded observation at an oil platform, rogue waves have attracted considerable attention and have also been observed in many other physical systems, especially in nonlinear optics \cite{solli2007optical}. From a mathematical point of view, rogue waves are usually described by rational solutions of integrable nonlinear wave equations on a nonzero background, such as the well-known Peregrine soliton of the nonlinear Schr\"odinger (NLS) equation. Since the discovery of higher-order rogue waves, much progress has been made in deriving explicit solutions and analyzing their dynamics in a wide range of integrable systems. 

Rogue waves may exhibit rich and highly regular patterns on the spatial-temporal plane. The study of such rogue wave patterns is important because it may provide critical information for predicting subsequent rogue waves based on preceding ones. Under suitable parameter choices, higher-order rogue wave solutions may form regular geometric structures such as circular rogue wave clusters \cite{kedziora2011circular}. Moreover, Yang and Yang proved analytically that these rogue wave patterns of the NLS equation are closely related to the root structures of the Yablonskii--Vorob'ev polynomial hierarchy. It has further been verified that these patterns are universal whenever the Schur polynomials involved in the $\tau$ functions have an index jump of two \cite{yang2021universal}, whereas the corresponding rogue wave patterns can be related to the Okamoto polynomial hierarchies when the Schur polynomials involved in the $\tau$ functions have an index jump of three \cite{yang2023rogue}. Recently, for specific integrable systems such as the vector NLS equation \cite{zhang2022rogue} and multi-component derivative NLS equations \cite{lin2024rogue}, it has been established that rogue wave solutions whose $\tau$ functions involve Schur polynomials with arbitrary index jumps are related to generalized Wronskian--Hermite polynomials.

In this section, we consider the multi-component (or \(M\)-component) Hirota equation
\begin{equation}
\label{eqn: vector-Hirota}
v_{j,t}
=
v_{j,xxx}
-3\left(\sum_{k=1}^{M} c_k |v_k|^2 \right) v_{j,x}
-3 v_j \left(\sum_{k=1}^{M} c_k v_k^{*} v_{k,x}\right),
\qquad j=1,2,\ldots,M,
\end{equation}
where \(*\) denotes complex conjugation and \(c_k=\pm1\).
Notably, by imposing suitable reductions, the multi-component Hirota equation reduces to the multi-component Sasa-Satsuma equation, which is also a significant higher-order generalization of the NLS equation and has been widely studied \cite{zhang2025dark,liu2023painleve}.
For example, in the case \(M=2\), the coupled Hirota equation can be reduced to the Sasa-Satsuma equation by imposing the reductions \(v_2=v_1^*\) and \(c_1=c_2\).
For the multi-component Hirota equation \eqref{eqn: vector-Hirota}, studies deriving rogue wave solutions by means of the KP reduction method are still lacking, especially for those solutions expressed in terms of Schur polynomials with an index jump of \(M+1\). 

\subsection{Rogue wave solutions}
We first introduce some notations required for the rogue wave solutions and their patterns. 
The Schur polynomials $S_n(\mathbf{x})$ are defined by
\begin{equation} \sum_{n=0}^{\infty}S_n(\mathbf{x})\lambda^n=\exp\left(\sum_{k=1}^{\infty}x_k\lambda^k\right),
\end{equation}
where $\mathbf{x}=(x_1,x_2,x_3,\cdots)$. To be more specific, we have
\begin{equation}\label{Schur polynomials}
  S_{0}(\mathbf{x})=1, \quad S_{1}(\mathbf{x})=x_{1}, \quad S_{2}(\mathbf{x})=\frac{1}{2} x_{1}^{2}+x_{2},  \ldots, \quad S_{j}(\mathbf{x})=\sum_{l_{1}+2 l_{2}+\cdots+r l_{r}=j}\left(\prod_{i=1}^{r} \frac{x_{i}^{l_{i}}}{l_{i} !}\right).
\end{equation}
Further, we define $S_j(\mathbf{x})\equiv0$ for $ j<0 $.

\begin{lemma}[\cite{zhang2022rogue}]\label{lemma:multiple roots}
Let $M$ be a positive integer and $\lambda_1>0, r_j \not = 0,k_j$ be real constants, with $k_i\neq k_j$ when $i\neq j$, $i,j=1,2,\dots, M$. 
Let $\mathcal{R}_M(z)$ be a rational function defined by
\begin{equation}\label{dimension reduction VS rat}
\mathcal{R}_M(z)= \sum_{j=1}^M\frac{r_j }{(z+  k_j)^2} + 1.
\end{equation}
Then $\mathcal{R}_M(z)=0$ has a pair of complex conjugate roots with nonzero imaginary parts of multiplicity $M$
\begin{equation}
     \lambda_1  \cos\left[  \frac{\pi}{M+1}\right] -  k_1  \pm  \mathrm{i}  \lambda_1 \sin\left[  \frac{\pi}{M+1}\right],
\end{equation}
if the parameters  $r_j,k_j$ satisfy the conditions
\begin{eqnarray}
  k_j  &=& k_1 +  \lambda_1 \left( \sin\left[\frac{\pi}{M+1}\right]  \cot\left[\frac{j\pi}{M+1}\right] -  \cos\left[\frac{\pi}{M+1}\right]\right), \quad j=2,\dots, M,
\end{eqnarray}
and
\begin{eqnarray}
  r_j
  = (-1)^{j+1}  \prod_{\substack{i=1 \\ i \not=j}}^M (k_j-k_i)^{-1}  \left(\lambda_1 \frac{\sin\left[  \frac{\pi}{M+1}\right] }{\sin\left[\frac{j\pi}{M+1}\right]} \right)^{M+1},\quad j=1,2,\dots, M.
\end{eqnarray}

\end{lemma}

In what follows, we construct the rogue wave solutions of the multi-component Hirota equation \eqref{eqn: vector-Hirota}. First, we define the rational function $\mathcal{F}_M(p)$ associated with the multi-component Hirota equation as
\begin{equation}\label{eqn:definition of F(p)}
\mathcal{F}_M(p) = -\sum_{j=1}^M \frac{c_j \rho_j^2}{p-\mathrm{i} k_j} + p, 
\end{equation}
where $c_j = \pm 1$, and $\rho_j > 0, k_j$ are real constants for $j=1,2,\dots,M$. It follows from Lemma \ref{lemma:multiple roots} that if the parameters $\{c_j, \rho_j, k_j\}_{j=1}^M$ satisfy the following constraints:
\begin{equation}\label{eqn:constraints of multiple zeros}
\begin{aligned}
 k_j &= k_1 + \lambda_1 \left( \sin\left[\frac{\pi}{M+1}\right] \cot\left[\frac{j \pi}{M+1}\right] - \cos\left[\frac{\pi}{M+1}\right]\right), \\
 c_j \rho_j^2 &= (-1)^{j} \prod_{\substack{i=1 \\ i \not=j}}^M (k_j-k_i)^{-1}  \left(\lambda_1 \frac{\sin\left[  \frac{\pi}{M+1}\right] }{\sin\left[\frac{j\pi}{M+1}\right]} \right)^{M+1},
 \end{aligned}
\end{equation}
then $\mathcal{F}_M^{\prime}(p)=0$ possesses a pair of non-imaginary roots with multiplicity $M$, given by
\begin{equation} \label{root of multiplicity M}
 p = \pm \lambda_1 \sin\left[\frac{\pi}{M+1}\right] - \mathrm{i} \lambda_1 \cos\left[\frac{\pi}{M+1}\right] + \mathrm{i} k_1.
\end{equation}
Next, we define the auxiliary function $p(\kappa)$ by
\begin{equation}\label{eqn:definition of p(kappa)}
\begin{aligned}
 \mathcal{F}_M(p(\kappa)) 
= \frac{\mathcal{F}_M(p(0))}{M+1} \sum_{n=1}^{M+1} \exp\left[\cos\left(\frac{2 n \pi}{M+1}\right) \kappa\right] \cos\left[\sin \left(\frac{2 n \pi}{M+1}\right)\kappa\right],
\end{aligned}
\end{equation}
where $p(0)$ is given by \eqref{root of multiplicity M}.
Then, one can obtain the following rogue wave solutions and rational solutions of the multi-component Hirota equation \eqref{eqn: vector-Hirota}.

\begin{theorem}\label{thm:rogue wave solutions}
Let $M$ be a positive integer,  $\rho_j >0, k_j$ be real constants, and $c_j=\pm1 $, where $j=1,2,\cdots,M$. Assume  $ \rho_j $ and $ k_j$ are given by \eqref{eqn:constraints of multiple zeros}. Let $\mathcal{F}_M(p)$ and $p(\kappa)$ be the functions defined by \eqref{eqn:definition of F(p)} and \eqref{eqn:definition of p(kappa)}, respectively, and let $p_0$ be a root of multiplicity $M$ of $\mathcal{F}_M^{\prime}(p)=0$ satisfying
$\Re(p_0)\Im(p_0)\neq0$. Then, the $M$-component Hirota equation \eqref{eqn: vector-Hirota} admits $\mathcal{N}$-th order rogue wave solutions
\begin{equation} \label{eqn:N-order rogue wave solutions of MHirota}
 v_{j,\mathcal{N}}(x,t) = \rho_{j}\frac{g_{j,\mathcal{N}}(x,t)}{f_{\mathcal{N}}(x,t)}  e^{\mathrm{i}\left(k_j x-w_j t\right)} , \quad j=1,2,\cdots,M,
\end{equation}
where $w_j=k_j^3+3k_j\sum_{i=1}^{M}c_i \rho_i^2+3\sum_{i=1}^{M}c_i\rho_i^2 k_i$, $\mathcal{N} = \left(N_1, N_2, \ldots, N_M\right)$ with $N_j~(j=1,2,\cdots,M)$ being nonnegative integers,
and $f_{\mathcal{N}}$ and $g_{j,\mathcal{N}}$ are given by
\begin{equation}
f_{\mathcal{N}}(x,t)=\tau_{\mathbf{n}_0}\left(x-3\sum_{i=1}^{M}c_i \rho_i^2 t,t\right), \quad g_{j,\mathcal{N}}(x,t)=\tau_{\mathbf{n}_j}\left(x-3\sum_{i=1}^{M}c_i \rho_i^2 t,t\right),
\end{equation}
with $\mathbf{n}_0=(0,0, \ldots, 0) \in \mathbb{R}^M$ and $\mathbf{n}_j= \mathbf{e}_j$ being the standard unit vector in $\mathbb{R}^M$. Here, $\tau_{\mathbf{n}}$ is given by the following $K \times K \, (1\leq K \leq M) $ block determinant
\begin{equation} \label{tau-block matrix-theorem}
\tau_{\mathbf{n}}=\det\left(\begin{array}{llll}
\tau_{\mathbf{n}}^{[I_1,I_1]} & \tau_{\mathbf{n}}^{[I_1,I_2]}&\cdots&\tau_{\mathbf{n}}^{[I_1,I_K]} \\
\tau_{\mathbf{n}}^{[I_2,I_1]} & \tau_{\mathbf{n}}^{[I_2,I_2]}&\cdots&\tau_{\mathbf{n}}^{[I_2,I_K]} \\ \vdots & \vdots& \ddots &\vdots \\\tau_{\mathbf{n}}^{[I_K,I_1]} & \tau_{\mathbf{n}}^{[I_K,I_2]}&\cdots&\tau_{\mathbf{n}}^{[I_K,I_K]}
\end{array}\right)_{N \times N},
\end{equation}
where
\begin{eqnarray}
&&\mathbf{n}=\left(n_1, n_2, \ldots, n_M\right),
\\
&& 1\leq I_1<I_2<\cdots<I_K\leq M,
\\
&& \tau_{\mathbf{n}}^{[I, J]}=\left(m_{(M+1)(i-1)+I,(M+1) (j-1)+J}^{(\mathbf{n},I, J)}\right)_{1 \leq i \leq N_I, 1 \leq j \leq N_J}, \quad 1 \leq I, J \leq M,  \label{tau-entry of block matrix-theorem}
\end{eqnarray}
$n_1, n_2, \ldots, n_M $ are integers,  $ I_j, N_{I_j} \, (j=1,2,\cdots,K)$  are positive integers with $N_{I_1}+N_{I_2}+\cdots+N_{I_K} = N, N \geq K$ and $N_l = 0$ for $l \in \{1,2,\cdots,M\}  \backslash \{I_1,I_2,\cdots,I_K\}$, and the corresponding matrix elements of \eqref{tau-entry of block matrix-theorem} are defined by
\begin{eqnarray} \label{eqn: definition of mij}
    m_{i, j}^{(\mathbf{n},I, J)}=\sum_{v=0}^{\min (i, j)}\left[\frac{\left|p_{1}\right|^{2}}{\left(p_{0}+p_{0}^{*}\right)^{2}}\right]^{v} S_{i-v}\left(\mathbf{x}_I^{+}(\mathbf{n})+v \mathbf{s}\right) S_{j-v}\left(\mathbf{x}_J^{-}(\mathbf{n})+v \mathbf{s}^{*}\right).
\end{eqnarray}
The vectors $\mathbf{x}_{I}^{\pm}
=\left(x_{1,I}^{\pm}, x_{2,I}^{\pm}, \cdots\right)$, $I=1,2,\dots,M$,
and $\mathbf{s}=\left(s_{1}, s_{2}, \cdots\right)$
are given by
\begin{eqnarray}
&&x_{i,I}^{+}=\alpha_i x+\beta_i t+ \sum_{j=1}^{M}n_j \theta_{ij} + a_{i,I}, \label{values of x+} \\
&&x_{i,I}^{-}=\alpha_i^* x+\beta_i^* t-\sum_{j=1}^{M}n_j \theta_{ij}^* +a_{i,I}^*,\\
&&\ln \left[\frac{1}{\kappa}\left(\frac{p_{0}+p_{0}^{*}}{p_{1}}\right)\left(\frac{p(\kappa)-p_{0}}{p(\kappa)+p_{0}^{*}}\right)\right]=\sum_{r=1}^{\infty} s_{r} \kappa^{r}, \label{sr}
\end{eqnarray}
where the asterisk `$*$' represents complex conjugation, $p_0=p(0),~p_1=p^{\prime}(0)$, the $a_{i,I}$'s are arbitrary constants, and $\alpha_i,~\beta_i,~\theta_{ij},~j = 1, 2, \dots, M$, are defined by the expansions
$$
p(\kappa)-p_{0}=\sum_{r=1}^{\infty} \alpha_{r} \kappa^{r}, \quad p^{3}(\kappa)-p_{0}^{3}=\sum_{r=1}^{\infty} \beta_{r} \kappa^{r},\quad\ln \left[\frac{p(\kappa)-\mathrm{i} k_{j}}{p_{0}-\mathrm{i} k_{j}}\right]=\sum_{r=1}^{\infty} \theta_{rj} \kappa^{r}.
$$
\end{theorem}

\begin{theorem}\label{thm:rational solutions}
Let $M$ be a positive integer,  $\rho_j >0, k_j$ be real constants, and $c_j=\pm1 $, where $j=1,2,\cdots,M$. Assume  $c_j, \rho_j $ and $ k_j$ are given by \eqref{eqn:constraints of multiple zeros}. Let $\mathcal{F}_M(p), p(\kappa)$ be functions defined by \eqref{eqn:definition of F(p)} and \eqref{eqn:definition of p(kappa)} respectively, and $p_0$ be a real root of $\mathcal{F}_M^{\prime}(p)$ of multiplicity $M$. Then, the $M$-component Hirota equation \eqref{eqn: vector-Hirota} admits $\mathcal{N}$-th order rational solutions 
\begin{equation} \label{eqn:N-order rational solutions of MHirota}
 v_{j,\mathcal{N}}(x,t) = \rho_{j}\frac{g_{j,\mathcal{N}}(x,t)}{f_{\mathcal{N}}(x,t)}  e^{\mathrm{i}\left(k_j x-w_j t\right)}, \quad j=1,2,\cdots,M,
\end{equation}
where  $w_j=k_j^3+3k_j\sum_{i=1}^{M}c_i \rho_i^2+3\sum_{i=1}^{M}c_i\rho_i^2 k_i$, $\mathcal{N} = \left(N_1, N_2, \ldots, N_M\right)$ with $N_j~(j=1,2,\cdots,M)$ being nonnegative integers,
and $f_{\mathcal{N}}$ and $g_{j,\mathcal{N}}$ are defined in the same way as in Theorem~\ref{thm:rogue wave solutions}.
\end{theorem}

\begin{remark}
    It should be noted that some of the rational solutions presented in Theorem~\ref{thm:rational solutions} are related to the partial rogue waves introduced in~\cite{Yang2023Yang}. In particular, it is shown that partial rogue waves of the Sasa-Satsuma equation arise when the associated generalized Okamoto polynomials possess real but not imaginary roots, or imaginary but not real roots. We also note that explicit formulas for the numbers of real and imaginary roots of the generalized Okamoto polynomials were derived in \cite{roffelsen2025real}.
\end{remark}

\begin{remark}
When $K=1$, the $\tau$ functions reduce to single-block determinants, i.e.,
\begin{eqnarray*}
\tau_{\mathbf{n}} &=& \det_{1 \leq i, j \leq N} \left(m_{(M+1)(i-1)+I_1,(M+1) (j-1)+I_1}^{(\mathbf{n},I_1, I_1)}\right), \quad 1 \leq I_1 \leq M,
\end{eqnarray*}
where $m_{i, j}^{(\mathbf{n},I, J)}$ is given by \eqref{eqn: definition of mij}.  In this case, we define the rogue wave solutions in Theorem \ref{thm:rogue wave solutions} to be the $I_1$-th type, and simply denote $\mathbf{x}_{I_1}^{\pm}$ by $\mathbf{x}^{\pm}$, thereby omitting the explicit dependence on $I_1$.
\end{remark}

We note that the explicit expressions of the rogue wave solutions in Theorem~\ref{thm:rogue wave solutions} depend on several auxiliary parameters, such as $p(\kappa)$ and $s_n$. 
To construct rogue waves of a given order, it suffices to determine the derivatives of $p(\kappa)$ at $\kappa=0$ up to a finite order. 
However, since the function $p(\kappa)$ is defined implicitly in \eqref{eqn:definition of p(kappa)}, these derivatives are not readily accessible.
For this reason, we present the following lemma, which provides the values of $p^{(n)}(0), n\in\mathbb{N}$.

\begin{lemma} \label{lemma:expression of pn}
Let \(p_0\in\mathbb{C}\setminus\mathrm{i}\mathbb{R}\) be a root of
multiplicity \(M\) of \(\mathcal{F}_M'(p)\) and let
\(p(\kappa)\) be any local analytic branch of
\eqref{eqn:definition of p(kappa)} satisfying \(p(0)=p_0\).
Then \(p(\kappa)\) admits the following Taylor expansion near
\(\kappa=0\):
\begin{equation}\label{eq:p-series}
p(\kappa)=p_0+\sum_{n=1}^{\infty} \frac{p^{(n)}(0)}{n!} \kappa^n :=p_0+\sum_{n=1}^{\infty}p_n\kappa^n,
\end{equation}
whose coefficients satisfy
\begin{equation}\label{eq:a1-formula}
p_1^{\,M+1}=\dfrac{\mathcal{F}_M(p_0)}{\mathcal{F}_M^{(M+1)}(p_0)}.
\end{equation}
Moreover, for any $n\ge 2$, one has
\begin{equation}\label{eq:pn-recursive-simplified}
p_n=
\begin{cases}
\dfrac{M!p_1}{(M+n)!}
-\dfrac{M!p_1}{\mathcal{F}_M(p_0)}\,\Sigma_n,
& M+n \equiv 0 \pmod{M+1},\\[2ex]
-\dfrac{M!p_1}{\mathcal{F}_M(p_0)}\,\Sigma_n,
& M+n \not\equiv 0 \pmod{M+1},
\end{cases}
\end{equation}
where $\Sigma_n$ depends only on $p_1,\dots,p_{n-1}$ and is given explicitly by
\begin{equation}\label{eq:Sigma-n}
\Sigma_n=
\frac{\mathcal{F}_M^{(M+1)}(p_0)}{(M+1)!}
\!\!\!\sum_{\substack{j_1+\cdots+j_{M+1}=M+n\\ 1\le j_\ell\le n-1}}
\!\!\! p_{j_1}\cdots p_{j_{M+1}}
+\sum_{r=M+2}^{M+n}\frac{\mathcal{F}_M^{(r)}(p_0)}{r!}
\!\!\!\sum_{\substack{j_1+\cdots+j_r=M+n\\ 1\le j_\ell\le n-1}}
\!\!\! p_{j_1}\cdots p_{j_r}.
\end{equation}
\end{lemma}
Moreover, it was shown in Appendix~D of \cite{yang2023rogue} that, for the Manakov system and the three-wave resonant interaction system, when the associated algebraic equation admits a double root, the parameters $s_n$ satisfy the constraint
\begin{equation}
    s_n = 0 \quad \text{if } n \not\equiv 0 \pmod{3}.
\end{equation}
Motivated by this result, we show in what follows that a similar conclusion can be established where the associated algebraic equation $\mathcal{F}_M'(p)=0$, defined in \eqref{eqn:definition of F(p)}, has a root of multiplicity $M$. It is worth noting that the same conclusion applies to other integrable equations, such as the vector NLS equation \cite{zhang2022rogue}, where the parameters are defined in the same way as in \eqref{sr}.

\begin{lemma}\label{lemma:expansion of s_r}
Suppose that \(p_0\in\mathbb C\setminus\mathrm{i}\mathbb R\) and
\(p_0,-p_0^*\) are roots of multiplicity \(M\) of
\(\mathcal F_M'(p)\).
Denote the coefficients $s_r$ by
\begin{equation}
 \ln \left[\frac{1}{\kappa}\left(\frac{p_{0}+p_{0}^{*}}{p_{1}}\right)\left(\frac{p(\kappa)-p_{0}}{p(\kappa)+p_{0}^{*}}\right)\right]
 = \sum_{r=1}^{\infty} s_{r} \kappa^{r},
\end{equation}
and define $p(\kappa)$ implicitly through
\begin{equation}
\begin{aligned}
 \mathcal{F}_M(p(\kappa))
 &= \frac{\mathcal{F}_M(p(0))}{M+1}
\sum_{n=1}^{M+1} \exp\left(\exp\left(\dfrac{2 n \pi \mathrm{i}}{M+1}\right) \kappa\right),
\end{aligned}
\end{equation}
where $p(0)=p_0$ and $p_1=p^{\prime}(0)$.  
Then it follows that
\begin{equation}
 s_j = 0, \quad \text{whenever } j \bmod (M+1) \neq 0 .
\end{equation}

\end{lemma}

\subsection{Rogue wave patterns}
In this subsection, we focus on the rogue wave patterns associated with a single large free parameter, $K=1$ and $M\ge2$ in Theorem~\ref{thm:rogue wave solutions}, which can be characterized by the roots of the generalized Wronskian--Hermite polynomials of jump $M+1$. For \(K=1\) and \(I_1=l\), we set
\begin{equation}
    a_m:=a_{m,l}.
\end{equation} 
Here, to simplify the proof, we consider the rogue wave patterns in the \((\bar{x},t)\)-plane, where $\bar{x}=x-3\left(\sum_{j=1}^{M}c_j \rho_j^2\right)t$.

\begin{theorem} \label{thm:Rogue wave patterns}
Let $M\ge2$. Assume that $|a_m|\gg1$, $m\geq 2$ with $m \not \equiv0 \pmod{M+1}$, and all other parameters are $O(1)$ in the $l$-th type $\mathcal{N}_l$-th order rogue wave solutions ($l =1,2,\cdots,M$)
\begin{equation}\label{RW solution-(1,1) type}
  v_{1, \mathcal{N}_l}(\bar{x}, t), \quad v_{2, \mathcal{N}_l}(\bar{x}, t), \quad \cdots, \quad v_{M, \mathcal{N}_l}(\bar{x}, t),
\end{equation}
of the multi-component Hirota equation in Theorem~\ref{thm:rogue wave solutions}, where $\mathcal{N}_l=N  \mathbf{e}_{l},$
 $N$ is a positive integer and $\mathbf{e}_j$ is the standard unit vector in $\mathbb{R}^M$. We also assume that all non-zero roots of the generalized Wronskian--Hermite polynomials $W_N^{[m,M+1,l]}(z)$ are simple.
Then, we have the following results concerning the asymptotics of the rogue wave solutions.

 \begin{itemize}
   \item [(1)]  In the outer region on the $(\bar{x},t)$ plane, when $\sqrt{\bar{x}^{2}+t^{2}}=O\left(\left|a_m\right|^{1 /m}\right)$, the $\mathcal{N}_l$-th order rogue waves separate into $\Gamma-\Gamma_0$ fundamental rogue waves, where $\Gamma$ and $\Gamma_0$ are given in \eqref{degree Gamma} and \eqref{Gamma_0 formula}. 
These fundamental rogue waves are
\begin{equation}
\hat{v}_n(\bar{x}, t)=\rho_n e^{\mathrm{i}\left(k_n \bar{x}-\bar{w}_n t\right)}  \dfrac{\left[p_1 \left(\bar{x}-\hat{x}_0\right)+3p_0^2p_1 (t-\hat{t}_0)+\theta_{1n}\right]\left[p_1^* (\bar{x}-\hat{x}_0)+ 3(p_0^2)^*p_1^* (t-\hat{t}_0)-\theta_{1n}^*\right]+\left|h_0\right|^2}{\left|p_1 (\bar{x}-\hat{x}_0)+3p_0^2p_1 (t-\hat{t}_0)\right|^2+\left|h_0\right|^2},
\end{equation}
where $n=1,2,\cdots,M, \;\bar{x}=x-3\left(\sum_{j=1}^{M}c_j \rho_j^2\right)t, \; |h_0|^2=\left|p_1\right|^2/\left(p_0+p_0^*\right)^2$ and $\bar{w}_j=k_j^3+3\sum_{i=1}^{M}c_i\rho_i^2 k_i$.
The positions $(\hat{x}_0, \hat{t}_0)$ of these fundamental rogue waves are given by
\begin{equation}
    \hat{x}_0=\frac{1}{\Im(p_0^2)}\Re\!\left[\frac{\mathrm{i}\left(p_0^2\right)^*}{p_1} \left(z_0\,a_m^{\frac{1}{m}}-\Delta_l\right)\right],\quad \hat{t}_0=\frac{1}{3\Im(p_0^2)}\Im\!\left[\dfrac{1}{p_1} \left(z_0\,a_m^{\frac{1}{m}}-\Delta_l\right)\right],\label{eqn:prediction of locations of RW}
\end{equation}
where $z_0$ is any one of the non-zero simple roots of $W_{N}^{[m,M+1,l]}(z)$, $\Delta_l$ is a $z_0$-dependent $O(1)$ quantity (the expression of $\Delta_l$ is provided in \eqref{eqn:def of delta_l}), and $\left(\Re,\Im\right)$ refer to the real and imaginary parts of a complex number, respectively. The approximation error here is $ O\left(|a_{m}|^{-1 /m}\right)$.
In other words, when $|a_m|\gg1 $ and $\left(\bar{x}-\hat{x}_0\right)^2+\left(t-\hat{t}_0\right)^2=O(1)$, we have the following asymptotics
\begin{eqnarray}
&v_{n, \mathcal{N}_l}(\bar{x}, t)=\hat{v}_n\left(\bar{x}, t\right) +O\left(\left|a_m\right|^{-1 / m}\right), \quad n =1,2,\cdots,M.
\end{eqnarray}

   \item [(2)]  In the inner
	region, where $\bar{x}^2+t^2=O(1)$,	 if zero is a root of the generalized Wronskian--Hermite polynomials $W_N^{[m,M+1,l]}(z)$, then $\left[ v_{1, \mathcal{N}_l}(\bar{x}, t), v_{2, \mathcal{N}_l}(\bar{x}, t), \cdots,  v_{M, \mathcal{N}_l }(\bar{x}, t)\right]$ is approximately a lower
	$\widehat{\mathcal{N}}_l$-th order rogue wave $$v_{1, \widehat{\mathcal{N}}_l}(\bar{x},
	t), \quad v_{2, \widehat{\mathcal{N}}_l}(\bar{x}, t), \quad \cdots, \quad
	v_{M, \widehat{\mathcal{N}}_l}(\bar{x}, t)$$
  where  $\widehat{\mathcal{N}}_l=\displaystyle{\sum_{j=1}^{M}N_{j,l}\mathbf{e}_j}$, and $N_{j,l}$ refers to the value of $N_j$ against $l \in \{1,2,\cdots,M\}$ given in Theorem \ref{lemma:root structure of WH polynomials}.  Moreover, the internal parameters $\hat{a}_{j,n}, (j\geq1, n=1,2,\cdots,M)$ in this lower-order rogue wave are related to those in the original rogue wave as follows:
  \begin{equation}
     \hat{a}_{j,n} = a_j, \quad\text{ for }\quad j \not \equiv 0  \pmod{M+1}.
  \end{equation}

 \end{itemize}

\end{theorem}

\begin{remark}\label{remark:rogue wave pattern associated with generalized Hermite poly}
Assume that $N=1$, $|a_m|\gg 1$, and all other parameters are $O(1)$ in the $l$-th type $\mathbf{e}_l$-th order rogue wave solutions $(l=1,2,\cdots,M)$ in Theorem \ref{thm:rogue wave solutions}. The resulting rogue wave patterns are associated with the generalized Hermite polynomials, whose root structures are arranged in an $m$-star configuration in the complex plane. 

Denote \( l=\hat{k}m+j \), where \( \hat{k}\ge 0 \) is an integer, and \( 0\le j\le m-1 \).
Hence, in the outer region, the $\mathbf{e}_l$-th order rogue waves separate into $\hat{k} m$ fundamental rogue waves, which are arranged along a \textit{skewed $m$-star pattern}, with each ray of the star containing $\hat{k}$ rogue waves.
In the inner region, the order $\widehat{\mathcal{N}}_l$ of the lower-order rogue wave is
\begin{equation}
\widehat{\mathcal N}_l
=
\begin{cases}
\mathbf 0, & j=0,\\
\mathbf e_j, & 1\le j\le m-1.
\end{cases}
\end{equation}
In particular, for $m=2$, we discover a new \textit{rogue wave chain} pattern (see the first row of Fig. \ref{fig:rogue wave pattern-10 com Hirota}).

\end{remark}

\begin{remark}\label{remark:circular rogue wave cluster}
Assume that $|a_m|\gg1$, $l+(N-2) (M+1) < m \le l+(N-1)(M+1), \; m \not \equiv0 \pmod{M+1}, \; N\geq3$, and all other parameters are $O(1)$ in the $l$-th type $\mathcal{N}_l$-th order rogue wave solutions ($l =1,2,\cdots,M$). In the outer region on the $(\bar{x},t)$ plane, when $\sqrt{\bar{x}^{2}+t^{2}}=O\left(\left|a_m\right|^{1 /m}\right)$, the $\mathcal{N}_l$-th order rogue waves separate into $m$ fundamental rogue waves, which are arranged along a skewed circular pattern. In the inner region, we have:
\begin{itemize}
    \item When $l+(N-2)(M+1) < m < l+(N-1)(M+1)$, the order $\widehat{\mathcal{N}}_l$ of the lower-order rogue wave in the inner region is given by
    \begin{equation}
       \widehat{\mathcal{N}}_l = \mathbf{e}_{j}+(N-1)\mathbf{e}_{l},
    \end{equation}
    where $j=l+(N-1)(M+1)-m$. Note that $j \not = l$, since $m \not \equiv 0 \pmod{M+1}.$
    
    \item  When \( m = n_N =l+(N-1)(M+1) \), the order $\widehat{\mathcal{N}}_l$ of the lower-order rogue wave in the inner region is given by
\begin{equation}
    \widehat{\mathcal{N}}_l = \begin{cases}
        (N-2) \mathbf{e}_{M}, & l=1,\\
        (N-1) \mathbf{e}_{l-1}, & l \not =1.
    \end{cases}
\end{equation}
Note that the circular rogue wave cluster reported in \cite{kedziora2011circular} corresponds to a special case of the above result. In particular, for an \(N\)-th order ring-type rogue wave of the NLS equation \cite{he2013generating}, the rogue wave located at the origin is of order \(N-2\).
\end{itemize}
For $N=2$, there exists at least one value of $m$, namely $m=l+M+1$, for which the rogue wave pattern forms a circular rogue wave cluster.
\end{remark}

\begin{remark}
By applying Theorem \ref{thm:WH-mcore-factorization} to Theorem \ref{thm:Rogue wave patterns}, one finds that the $\mathcal{N}_l$-th order rogue wave solution of the $M$-component Hirota equation contains no lower-order rogue wave at the origin if the partition
\begin{equation}\label{eqn: partition of remark of RW patterns}
    \left((N-1)M+l,\; (N-2)M+l,\; \ldots,\; M+l,\; l\right)
\end{equation}
is $m$-divisible, that is, its $m$-core is empty.

\end{remark}

\begin{remark}
Although Theorem~\ref{thm:Rogue wave patterns} is formulated for \(M\ge 2\), the corresponding result for $M=1$ can be derived by a similar argument in \cite{yang2021rogue,yang2021universal}. In particular, the predicted locations of the fundamental rogue waves in the outer region are given by \eqref{eqn:prediction of locations of RW} with \(\Delta_l=0\), while the rest of the conclusions remain the same as  those in Theorem~\ref{thm:Rogue wave patterns}. Therefore, the proof for the case \(M=1\) is omitted.
\end{remark}

\subsection{Comparison between predicted and true rogue wave patterns}\label{sec:Comparison between predicted and true rogue wave patterns}
In this subsection, we present several examples of Theorem~\ref{thm:Rogue wave patterns} to compare our analytical predictions of rogue wave patterns with the corresponding rogue wave solutions. In particular, we first consider the rogue wave patterns associated with the generalized Hermite polynomials described in Remark~\ref{remark:rogue wave pattern associated with generalized Hermite poly}, which correspond to the case where the order of the determinant of the $\tau$-function shown in \eqref{tau-block matrix-theorem} is equal to $1$. Then, the rogue wave patterns described in Remark~\ref{remark:circular rogue wave cluster} are presented, which correspond to the circular rogue wave clusters. 
For all density plots of the rogue wave solutions in this paper, unless otherwise specified, the color bars are fixed to the range $[0,3]$ in order to highlight the distribution and overall pattern of the rogue waves more clearly.

\subsubsection{Rogue wave patterns associated with the generalized Hermite polynomials} \label{sec:rogue wave pattern-generalized Hermite poly}
To illustrate the rogue wave patterns associated with the generalized Hermite polynomials given in Theorem~\ref{thm:Rogue wave patterns} and Remark \ref{remark:rogue wave pattern associated with generalized Hermite poly}, we choose the $\mathbf{e}_l$-th order rogue wave of  the $10$-component Hirota equation to further illustrate the skewed $m$-star shaped rogue wave pattern, which is presented in Fig.~\ref{fig:rogue wave pattern-10 com Hirota}. In this case, we choose the parameters $c_i=-1$ $(i=1,2,\dots,10)$, together with
\begin{equation}
    \begin{aligned}
\{k_1,k_2,k_3,k_4,k_5,k_6,k_7,k_8,k_9,k_{10}\}
\approx
\{&0.5, -0.0211086, -0.21537, -0.33083, -0.418986, \\&-0.5, -0.588156, -0.703616, -0.897877, -1.41899\},
\end{aligned}
\end{equation}
and
\begin{equation}
    \begin{aligned}
\{\rho_1^2,\rho_2^2,\rho_3^2,\rho_4^2,\rho_5^2,\rho_6^2,\rho_7^2,\rho_8^2,\rho_9^2,\rho_{10}^2\}
\approx
\{
&1, 0.271554, 0.138969, 0.0959274, 0.0810141, 0.0810141,\\& 0.0959274, 
0.138969, 0.271554, 1
\}.
    \end{aligned}
\end{equation}
Furthermore, we take the root of multiplicity ten of the equation $\mathcal{F}_{10}^{\prime}(p)=0$, that is,
\begin{equation}
    p_0 = \sin\left(\frac{\pi}{11}\right)+\mathrm{i} \left(\frac{1}{2} - \cos \left(\frac{\pi}{11}\right)\right) ,
\end{equation}
and choose $p_1\approx 0.10786 - 0.00144\mathrm{i}$.

In Fig.~\ref{fig:rogue wave pattern-10 com Hirota}, the rogue waves exhibit a skewed $m$-star configuration, with peaks distributed along $m$ branches radiating from the central region. For example, the rogue waves in the first row of Fig.~\ref{fig:rogue wave pattern-10 com Hirota} are arranged on a curve, whereas those in the second and third rows are distributed along three and four branches radiating from the central region, respectively. Moreover, there is very good agreement between the predicted rogue wave locations and the corresponding true rogue wave positions.

\begin{figure}[H]
    \centering
    \includegraphics[width=0.9\linewidth]{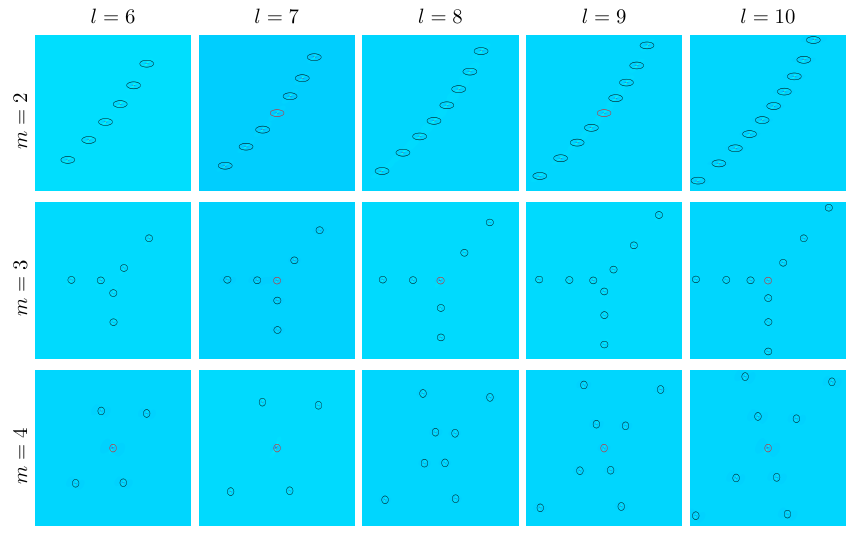}
    \caption{True $\mathbf{e}_l$-th order rogue wave patterns of the fifth component $|v_5(\bar{x},t)|$
for the $10$-component Hirota equation in the $(\bar{x},t)$-plane,
where $\bar{x}=x-\left(3\sum_{j=1}^{10} c_j\rho_j^2\right)t$.
Here, $\mathbf{e}_l$ denotes the $l$-th standard basis vector in $\mathbb{R}^{10}$,
and the large parameters $a_m$ are indicated on the left of each row, with all other free parameters set to zero. The $(\bar{x},t)$-ranges are chosen uniformly in each row as follows: first row :$-450\leq \bar{x}\leq 450$, $-910\leq t\leq 910$; second row :$-870\leq \bar{x}\leq 870$, $-930\leq t\leq 930$; third row :$-950\leq \bar{x}\leq 950$, $-810\leq t\leq 810$.
}
\label{fig:rogue wave pattern-10 com Hirota}
\end{figure}

\subsubsection{Circular rogue wave clusters in the multi-component Hirota equation}\label{sec:Circular rogue wave clusters}
In \cite{kedziora2011circular}, circular rogue wave clusters were reported for the NLS equation and later observed in other integrable systems; these solutions are associated with jump-$2$ type Schur functions. In this paper, we show that circular rogue wave clusters also arise for rogue wave solutions associated with general jump-type Schur functions, which commonly occur in multi-component integrable equations.

First, we present circular rogue wave clusters in the Hirota equation, which correspond to $M=1$ in Theorem \ref{thm:rogue wave solutions}. One can choose
\begin{equation}
    c_1=-1, \quad k_1=\frac{1}{2}, \quad \rho_1=1,
\end{equation}
and
\begin{equation}
    p_0=1 + \frac{1}{2}\mathrm{i}, \quad p_1=\frac{\sqrt{4+\mathrm{i}}}{2}.
\end{equation}
The third- to sixth-order rogue wave solutions are plotted in Fig.~\ref{fig:circular rogue waves cluster in the hirota equation}. For each order $N$, we choose $a_{2N-1}$ to be a sufficiently large number while setting all other $a_m$ to zero. It can be seen that all figures consist of many fundamental rogue waves arranged in a ring, with a possible lower-order rogue wave located at the origin. In this case, the lower-order rogue wave is expected to be of order $N-2$.

\begin{figure}[H]
    \centering
    \includegraphics[width=0.9\linewidth]{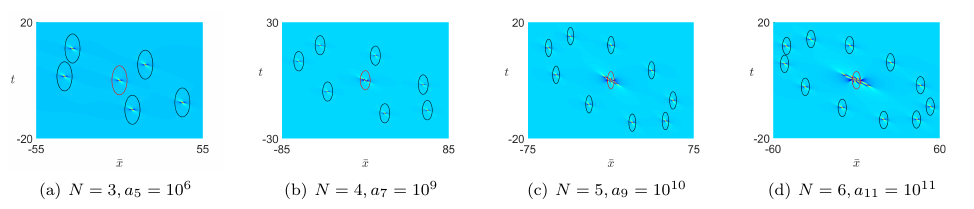}
    \caption{True $N$-th order rogue wave patterns of $|v_1(\bar{x},t)|$
for the Hirota equation in the $(\bar{x},t)$-plane,
where $\bar{x}=x - 3 c_1\rho_1^2 t$. The black circles represent the predicted positions of each fundamental rogue wave in the outer region, and the red circles represent the predicted positions of the lower-order rogue wave in the inner region.}
\label{fig:circular rogue waves cluster in the hirota equation}
\end{figure}

Moreover, it is known that the rogue wave patterns in this case are related to the root structures of the Yablonskii--Vorob’ev polynomial hierarchy $W_N^{[2N-1,2,1]}(z)$ for $N=3,4,5,6$. As discussed in Section~\ref{sec:WH with partition}, these polynomials correspond to the partitions $Y_N=(N,\cdots,3,2,1)$. By choosing $m=2N-1$, the associated Young diagrams can only remove a rim $m$-hook (see Fig. \ref{fig:partition in the hirota equation}), while the corresponding $m$-cores are
\begin{equation}
    \mathrm{core}_5(Y_3)=(1), \quad\mathrm{core}_7(Y_4)=(2,1), \quad \mathrm{core}_9(Y_5)=(3,2,1),\quad \mathrm{core}_{11}(Y_6)=(4,3,2,1),
\end{equation}
namely,
\begin{equation*}
    \mathrm{core}_5(Y_3)=\ydiagram{1}, \quad\mathrm{core}_7(Y_4)=\ydiagram{2,1}, \quad \mathrm{core}_9(Y_5)=\ydiagram{3,2,1},\quad \mathrm{core}_{11}(Y_6)=\ydiagram{4,3,2,1},
\end{equation*}
with the corresponding $\beta$-numbers
\begin{equation}
   \beta=  (1), \quad (3,1), \quad (5,3,1),\quad (7,5,3,1),
\end{equation}
respectively. By applying Lemma \ref{lem:size-of-core-by-runner-beads}, we can place the above $\beta$-numbers on a $2$-abacus and obtain the corresponding $b_1$ values:
\begin{equation}\label{eqn:1-component, n_i values for special cases}
    b_1= 1, \quad 2, \quad 3, \quad 4,
\end{equation}
where $b_0=0$ in all cases.
Here, we remark on the relationship between the above rogue wave patterns and the corresponding partitions. For each partition, only a single rim $m$-hook can be removed, which corresponds to all the rogue waves in the outer region being arranged along a single ring. Moreover, the size of the $m$-core for each case, i.e., the number of boxes, corresponds to the number of fundamental rogue waves that compose the lower-order rogue wave at the origin. In addition, the values $b_1$ computed from the $\beta$-numbers in \eqref{eqn:1-component, n_i values for special cases} give the order of the lower-order rogue wave located at the origin.

\begin{figure}[H]
    \centering
    \includegraphics[width=0.9\linewidth]{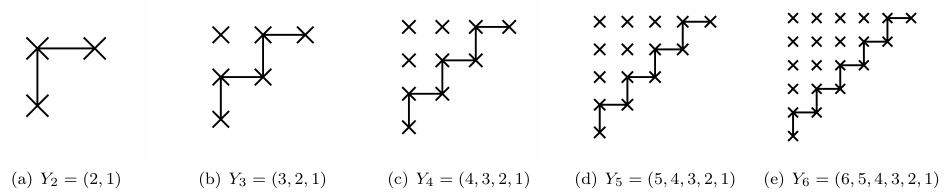}
\caption{The corresponding partitions $Y_N$ associated with Fig.~\ref{fig:circular rogue waves cluster in the hirota equation} consist of an $m$-core (not connected by the polyline) together with all removable rim $m$-hooks (connected by the polyline). Here, $m=2N-1$ and $N=2,3,4,5,6$.}
    \label{fig:partition in the hirota equation}
\end{figure}

To illustrate the circular rogue wave clusters for the $2$-component Hirota equation  given in Theorem~\ref{thm:Rogue wave patterns}, we choose the parameters as
\begin{equation}
  c_1=c_2=-1, \quad  k_1=1, \, k_2=0, \quad \rho_1=\rho_2=1.
\end{equation}
For the subsequent analysis, we will further fix the parameters
\begin{equation}
    p_0=\frac{1}{2} \left(\sqrt{3}+\mathrm{i}\right), \quad p_1=\sqrt[6]{\frac{13}{72}+\frac{\mathrm{i}}{8 \sqrt{3}}},
\end{equation}
where $p_0$ is a double root of $\mathcal{F}_2^{\prime}(p)=0$. 

Then, for a given order $\mathcal{N}=(N_1,N_2)$ with only one nonzero $N_j \geq 3$, the $\mathcal{N}$-order rogue wave solutions can appear as circular rogue wave clusters by choosing the two largest nontrivial values of $m$ and taking the corresponding parameter $a_m$ to be sufficiently large. The corresponding figures are shown in Fig.~\ref{fig:circular RW cluster of 2-com with N=3}. Each rogue wave pattern consists of a ring of fundamental rogue waves in the outer region and a lower-order rogue wave located in the inner region. The predicted location of each fundamental rogue wave in the outer region is shown as black circles in Fig.~\ref{fig:circular RW cluster of 2-com with N=3}, while the orders of inner rogue waves are given by 
\begin{eqnarray}\label{eqn:2-component, N_i values for special cases}
    (2,1),\quad (0,1), \quad (1,2), \quad (2,0).
\end{eqnarray}
It is observed that the actual rogue wave patterns are in very good agreement with the predicted results.

\begin{figure}[H]
    \centering
    \includegraphics[width=0.9\linewidth]{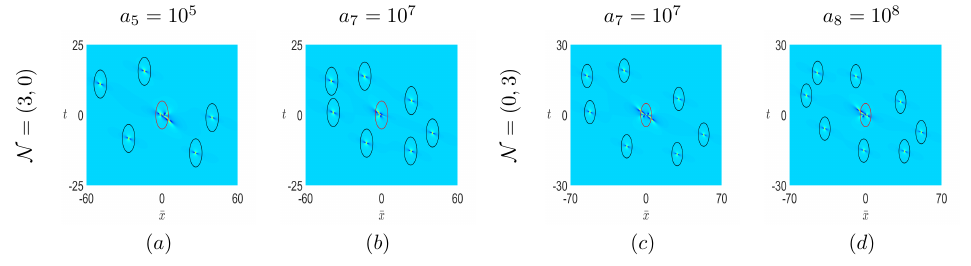}
    \caption{Circular rogue wave clusters with order
$\mathcal{N}$ of the second component $|v_2(\bar{x},t)|$ for the $2$-component Hirota equation in the $(\bar{x},t)$-plane, where $\bar{x}=x-\left(3\sum_{j=1}^2 c_j \rho_j^2\right)t$. Here, panels $(a)$ and $(b)$ correspond to \(\mathcal{N}=(3,0)\), while panels $(c)$ and $(d)$ correspond to \(\mathcal{N}=(0,3)\). The black circles represent the predicted positions of each fundamental rogue wave in the outer region, and the red circles represent the predicted positions of the lower-order rogue wave in the inner region.}
\label{fig:circular RW cluster of 2-com with N=3}
\end{figure}

Moreover, the rogue wave patterns in this case are related to the root structures of the generalized Wronskian--Hermite polynomials $W_3^{[m,3,1]}(z)$ and $W_3^{[m,3,2]}(z)$. As discussed in Sections \ref{sec:WH with partition}, these polynomials correspond to the partitions $Y_{3,1}=(5,3,1)$ and $Y_{3,2}=(6,4,2)$, respectively. By choosing $m$ to match the value used in Fig.~\ref{fig:circular RW cluster of 2-com with N=3}, the associated Young diagrams allow only the removal of a rim $m$-hook (see Fig.~\ref{fig: partitions in the 2-component hirota equation}), while the corresponding $m$-cores are
\begin{equation}
    \mathrm{core}_5(Y_{3,1})=(2,1,1), \quad \mathrm{core}_7(Y_{3,1})=(2), \quad\mathrm{core}_7(Y_{3,2})=(3,1,1), \quad \mathrm{core}_8(Y_{3,2})=(3,1),
\end{equation}
with the corresponding $\beta$-numbers
\begin{equation}
   \beta= (4,2,1), \quad (2), \quad (5,2,1), \quad (4,1),
\end{equation}
respectively. By applying Lemma \ref{lem:size-of-core-by-runner-beads}, we can place the above $\beta$-numbers on a $3$-abacus and obtain the corresponding $b_i$ values:
\begin{equation}
    (b_1,b_2)=(2,1),\quad (0,1), \quad (1,2), \quad (2,0),
\end{equation}
where $b_0=0$ in all cases. It is worth noting that these $(b_1,b_2)$ values coincide with the orders of the rogue waves in the inner region, as given in \eqref{eqn:2-component, N_i values for special cases}.

\begin{figure}[H]
    \centering
    \includegraphics[width=0.9\linewidth]{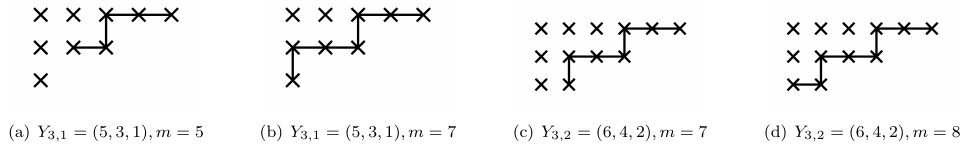}
\caption{The corresponding partitions associated with Fig.~\ref{fig:circular RW cluster of 2-com with N=3} consist of an $m$-core (not connected by the polyline) together with all removable rim $m$-hooks (connected by the polyline).}
    \label{fig: partitions in the 2-component hirota equation}
\end{figure}

In what follows, we consider circular rogue wave clusters in the $3$-component Hirota equation. To demonstrate, we choose the parameters as
\begin{equation}
  c_1=c_2=c_3=-1, \,  k_1=2, \, k_2=1, \,  k_3=0, \, \rho_1=\rho_3=\sqrt{2}, \, \rho_2=1.
\end{equation}
For the subsequent analysis, we further fix the parameters
\begin{equation}
    p_0=1+\mathrm{i} , \quad p_1=-\frac{1}{\sqrt{2}}\sqrt[4]{\frac{4}{3}+\frac{\mathrm{i}}{3}}.
\end{equation}
Then, for a given order $\mathcal{N}=(N_1,N_2,N_3)$ with only one nonzero $N_j \geq 3$, the $\mathcal{N}$-order rogue wave solutions can appear as circular rogue wave clusters by choosing the three largest nontrivial values of $m$ and taking the corresponding parameter $a_m$ to be sufficiently large. The corresponding figures are shown in Fig.~\ref{fig:circular RW cluster of 3-com with N=3}.  According to Theorem \ref{thm:Rogue wave patterns}, we can calculate that the orders of rogue waves in the inner region are
\begin{eqnarray}\label{eqn:3-component, N_i values for special cases} 
    &&\text{the first row:}  \quad (2,0,1), \quad (2,1,0),\quad (0,0,1), \nonumber\\
    &&\text{the second row:} \quad (0,2,1), \quad (1,2,0), \quad (2,0,0).
\end{eqnarray}
It is observed that the actual rogue wave patterns are in very good agreement with the predicted results.

\begin{figure}[H]
    \centering
    \includegraphics[width=0.9\linewidth]{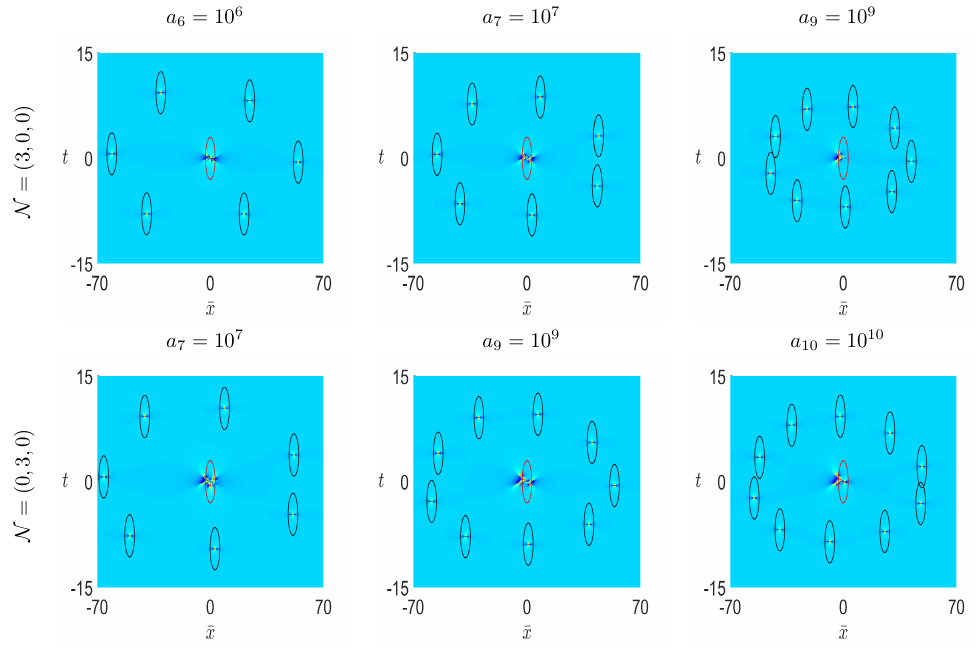}
    \caption{Circular rogue wave clusters with order
$\mathcal{N}$ of the second component $|v_2(\bar{x},t)|$ for the $3$-component Hirota equation in the $(\bar{x},t)$-plane, where $\bar{x}=x-\left(3\sum_{j=1}^3 c_j \rho_j^2\right)t$. The black circles represent the predicted positions of each fundamental rogue wave in the outer region, and the red circles represent the predicted positions of the lower-order rogue wave in the inner region.}
\label{fig:circular RW cluster of 3-com with N=3}
\end{figure}

\section*{Acknowledgments}
The authors would like to thank H.K. Yang for helpful discussions. 
C.F. Wu was supported by the National Natural Science Foundation of China (Grant No. 12471077).

\appendix

\section{Calculations of \texorpdfstring{\(N_i\)}{N\_i} in Theorem~\ref{lemma:root structure of WH polynomials}}
\label{app:Ni-procedure}
In Theorem~\ref{lemma:root structure of WH polynomials}, the multiplicity $\Gamma_0$ is expressed in terms of integers $N_i$ $(1\le i\le k-1)$.
For the reader's convenience, we outline below the procedure given in \cite{lin2024rogue} to obtain $N_i$ from $(m,k,l,N)$.

First, we define some intermediate quantities.
Let $d=\gcd(m,k)$ and denote
\begin{equation*}
    m_0\equiv m\pmod{k},\quad 1\leq m_0\leq k-1,
    \quad
    N_0\equiv N\pmod{\widehat m},\quad 0\leq N_0\leq\widehat m-1,
\end{equation*}
and
\begin{equation*}
   \widehat m=\frac{m}{d},
    \quad
    \widehat k=\frac{k}{d},
    \quad  m_1=\left\lfloor\frac{m}{k}\right\rfloor,
    \quad
    k_N=\left\lfloor\frac{N}{\widehat m}\right\rfloor.
\end{equation*}

Next, we define $\hat{N}_i$ $(0\leq i\leq k-1)$.

\begin{itemize}
    \item [Case I:]  \(0 \leq N_0 < \left\lfloor l/m_0 \right\rfloor m_1 + m_1 + 1\).
\begin{itemize}
    \item If $m_1\neq 0$,
    then we set
\[
\hat{N}_i = 
\begin{cases}
m_1, & i = l - (j-1)m_0, \quad 1 \leq j \leq \lfloor N_0/m_1 \rfloor, \\
N_0 \bmod m_1, & i = l - \lfloor N_0/m_1 \rfloor m_0, \\
0, & \text{otherwise}.
\end{cases}
\]
    \item If
    $m_1=0$,
     we have $N_0=0$, and set
\begin{equation*}
    \hat{N}_i=0,\quad 0\leq i\leq k-1.
\end{equation*}
\end{itemize}

\item[Case II:] $N_0\geq \lfloor l/m_0\rfloor m_1+m_1+1$.\\   
Let \(r\) be the integer satisfying
\[
\left\lfloor \frac{l + (r-1)k}{m_0} \right\rfloor m_1 + m_1 + r \leq N_0 < \left\lfloor \frac{l + rk}{m_0} \right\rfloor m_1 + m_1 + r + 1,\quad 1\leq r\leq m_0/d.
\]

\begin{itemize}
    \item If $m_1\neq 0$, then we set 
\[
\hat{N}_i=
\begin{cases}
m_1 + 1, & \quad i =  (l + (j-1)k) \bmod{m_0},\ 1 \leq j \leq r,\\
m_1, &  \quad 
i=
\begin{cases}
l - (j-1)m_0, & \ 1 \leq j \leq \left\lfloor \frac{l}{m_0} \right\rfloor,\\
l + k - (j-1)m_0, & \left\lfloor \frac{l}{m_0} \right\rfloor + 1 < j \leq \left\lfloor \frac{l + k}{m_0} \right\rfloor,\\
...\\
l + rk - (j-1)m_0, &\ \left\lfloor \frac{l+(r-1)k}{m_0} \right\rfloor + 1 < j \leq \left\lfloor \frac{N_0-r}{m_1} \right\rfloor,
\end{cases}\\
(N_0 - r) \bmod m_1, & \quad i = l + rk - \left\lfloor \frac{N_0 - r}{m_1} \right\rfloor m_0,\\
0, & \quad \text{otherwise}.
\end{cases}
\]

    \item If $m_1=0$, then we have $r=N_0$, $m_0=m$, and set 
\begin{equation*}
    \hat{N}_i=
\begin{cases}
    1, & \quad i= (l+(j-1)k) \bmod{ m},\quad 1\leq j\leq N_0,\\
    0, & \quad \text{otherwise}.
\end{cases}
\end{equation*}

\end{itemize}

\end{itemize}
Then we can use \(\hat{N}_i\) to determine \(N_i\) by considering two cases. 
\begin{itemize}
    \item[Case $1$:]\(d=\gcd(m,k)=1\). \\
\begin{itemize}
     \item[(i)] \quad If \(\hat{N}_0 = 0\), then\\
     \[
     N_i = \hat{N}_i \text{ for } 1 \leq i \leq k-1.
     \] \\
     \item[(ii)] \quad If \(\hat{N}_i \neq 0 \, (0 \leq i \leq j - 1)\) and \(\hat{N}_j = 0 \; (1 \leq j < k)\), then
\[
N_i = 
\begin{cases}
\hat{N}_{i+j}, & 1 \leq i \leq k - j - 1, \\
\hat{N}_{i+j-k} - 1, & k - j \leq i < k.
\end{cases}
\]\\
     \item[(iii)] \quad If all \(\hat{N}_i \neq 0 \,(0 \leq i < k)\), let \(\bar{r} = \min\{\hat{N}_0,\hat{N}_1,\ldots,\hat{N}_{k-1}\}\) and define \(\bar{N}_i = \hat{N}_i - \bar{r}\).\\
    \begin{itemize}
         \item[(a)] If \(\bar{N}_0=0\), then\\
         \[
         N_i = \bar{N}_i \text{ for } 1 \leq i \leq k-1;
         \]
         \item[(b)] If \(\bar{N}_i\) satisfies condition (ii), then
     \[
N_i = 
\begin{cases}
\bar{N}_{i+j}, & 1 \leq i \leq k - j - 1, \\
\bar{N}_{i+j-k} - 1, & k - j \leq i < k.
\end{cases}
\]
\end{itemize}
\end{itemize}

\item[{Case $2$:}] 
{
\(d=\gcd(m,k) \neq 1\). \\
Define
\[
l_1=l\bmod d,\qquad 0\leq l_1\leq d-1,
\]
and
\[
\tilde{N}_{l_j}^{(1)} = 
\begin{cases}
\lfloor k_N \widehat{m} / \widehat{k} \rfloor + 1, & 1 \leq j \leq k_N \widehat{m} \bmod{\widehat{k}}, \\
\lfloor k_N \widehat{m} / \widehat{k} \rfloor, & k_N \widehat{m} \bmod \widehat{k} < j \leq \widehat{k},
\end{cases}
\]
where \(l_j = l_1 + (j-1)d\).
Next we set
\[
\tilde{N}_{l_j}^{(2)} = 
\begin{cases}
\tilde{N}_{l_j}^{(1)} + \hat{N}_{l_j + k - (k_N m) \bmod k}, & 0 \leq l_j < (k_N m) \bmod k,\\
\tilde{N}_{l_j}^{(1)} + \hat{N}_{l_j - (k_N m) \bmod k}, & (k_N m) \bmod k \leq l_j < k.
\end{cases}
\]\\
Then, the \(N_i\) are determined as follows.
\begin{itemize}
     \item If \(d \nmid l\) (i.e., \(l_1 \neq 0\)):
\[
N_i = 
\begin{cases}
\tilde{N}_{l_j}^{(2)}, & i = l_j \ (1 \leq j \leq \widehat{k}), \\
0, & \text{otherwise}.
\end{cases}
\]

\item If \(d\mid l\) (i.e., \(l_1=0\)), then
\[
N_i=
\begin{cases}
\tilde{N}_{l_j}^{(2)},
& \tilde{N}_0^{(2)}=0,\quad
  i=l_j,\quad 2\leq j\leq\widehat{k},\\
\tilde{N}_0^{(2)}-1,
& \tilde{N}_0^{(2)}>0,\quad i=k-1,\\
\tilde{N}_{l_j}^{(2)},
& \tilde{N}_0^{(2)}>0,\quad
  i=l_j-1,\quad 2\leq j\leq\widehat{k},\\
0,
& \text{otherwise}.
\end{cases}
\]
\end{itemize}
}
\end{itemize}

\section{Root plots of the generalized Wronskian--Hermite polynomials}\label{appendix:Root plots of the generalized WH polynomials}
In this appendix, we illustrate the root structures of the Yablonskii--Vorob'ev polynomial hierarchy and the Okamoto polynomial hierarchy \cite{yang2021rogue,yang2021universal,yang2023rogue}, namely the Wronskian--Hermite polynomials $W_N^{[m,k,l]}(z)$ with $k=2,3$, in Figs.~\ref{fig:WHzeros_k2_l1_staircase}-\ref{fig:WHzeros_k3_l1_staircase}. 
Note that, for a given $N$, the roots of $W_N^{[m,k,l]}(z)$ degenerate to the trivial zero configuration whenever $m > l+(N-1)k$ or $m \equiv 0 \pmod{k}$. 
Consequently, we restrict our presentation to non-degenerate values of $m$. 
For $N \geq 3$, the zeros of the Wronskian--Hermite polynomials $W_N^{[m,k,l]}(z)$ corresponding to the largest $k-1$ non-degenerate values of $m$ exhibit a circular root structure. For example, when \(N \ge 3\), the polynomials corresponding to the largest value of \(m\) in each row of Fig.~\ref{fig:WHzeros_k2_l1_staircase}, namely, the last panel, exhibit a circular root structure. In contrast, in Figs.~\ref{fig:WHzeros_k3_l1_staircase}-\ref{fig:WHzeros_k3_l2_staircase}, the polynomials corresponding to the two largest values of \(m\) in each row, namely, the last two panels, both exhibit circular root structures.

\begin{figure}[H]
    \centering
    \includegraphics[width=0.63\linewidth]{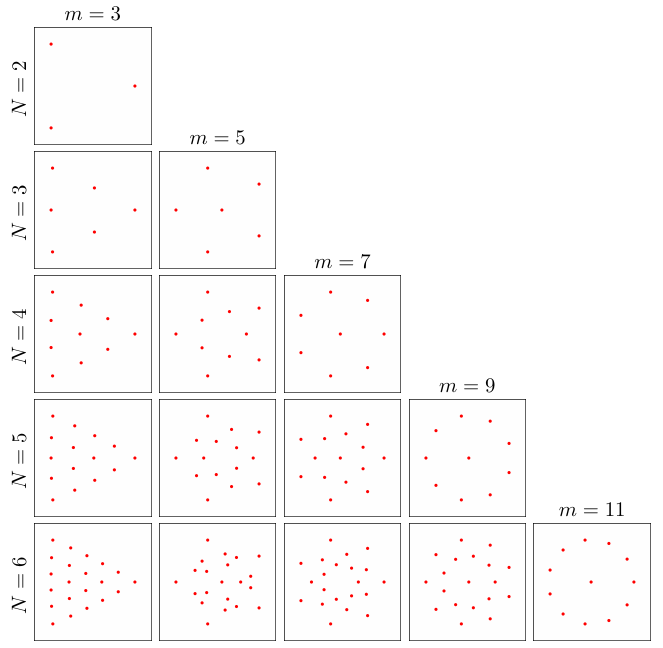}
    \caption{Roots of the generalized Wronskian--Hermite polynomials $W^{[m,2,1]}_{N}(z)$, namely the Yablonskii--Vorob’ev polynomial hierarchy.
The rows correspond to $N=2,\dots,6$, and the columns correspond to $m=3,5,7,9,11$.
From $N=2$ to $6$, the number of displayed cases increases by one in each row. The plotting domain in the complex plane is chosen adaptively to ensure that all zeros are included.
}
\label{fig:WHzeros_k2_l1_staircase}
\end{figure}

\begin{figure}[H]
    \centering
    \includegraphics[width=0.85\linewidth]{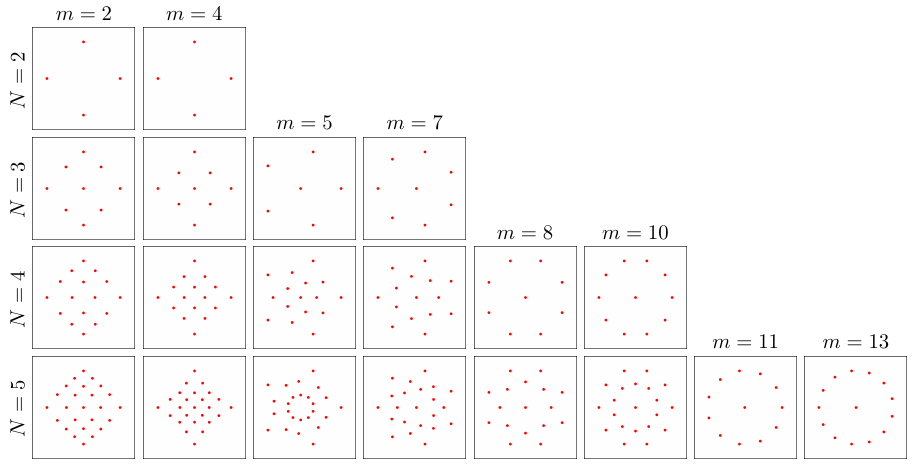}
    \caption{Roots of the generalized Wronskian--Hermite polynomials $W^{[m,3,1]}_{N}(z)$, namely the Okamoto polynomial hierarchy.
From $N=2$ to $5$, the number of displayed cases increases by two in each row.  The plotting domain in the complex plane is chosen adaptively to ensure that all zeros are included.}
\label{fig:WHzeros_k3_l1_staircase}
\end{figure}

\begin{figure}[H]
    \centering
    \includegraphics[width=0.9\linewidth]{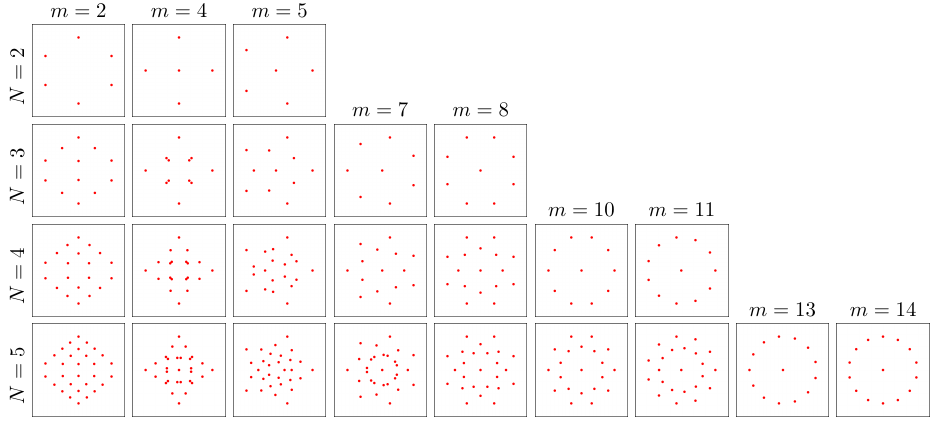}
    \caption{Roots of the generalized Wronskian--Hermite polynomials $W^{[m,3,2]}_{N}(z)$, namely the Okamoto polynomial hierarchy. From $N=2$ to $5$, the number of displayed cases increases by two in each row. The plotting domain in the complex plane is chosen adaptively to ensure that all zeros are included.}
\label{fig:WHzeros_k3_l2_staircase}
\end{figure}

\section{Proof of Theorem \ref{thm:rogue wave solutions}}
In this Appendix, we employ Hirota’s bilinear method and the KP reduction method to derive rogue wave solutions of the multi-component Hirota equation in Theorem \ref{thm:rogue wave solutions}. 

First, we can transform the multi-component Hirota equation \eqref{eqn: vector-Hirota} into a set of bilinear equations \cite{shi2026soliton2} 
\begin{flalign}
&\left[
 D_x^3 - D_t
+ 3 \mathrm{i} k_j D_x^2
- 3 \left( k_j^2 + 2 \sum_{l=1}^M c_l \rho_l^2 \right) D_x
\!-\! 3 \mathrm{i} k_j \sum_{l=1}^M c_l\rho_l^2\!+\! 3 \mathrm{i} \sum_{l=1}^M c_l \rho_l^2 k_l
\right] g_j \cdot f +3 \mathrm{i} \sum_{l=1}^M (k_j - k_l)c_l\rho_l^2 s_{jl} g_l^{*}\!=\!0,& \label{eqn: bilinear 1}
\\
&\left[ D_x + \mathrm{i} (k_j - k_l) \right] g_j \cdot g_l
= \mathrm{i} (k_j - k_l) s_{jl} f,&\label{eqn: bilinear 2}
\\
&\left(
D_x^2 - 2 \sum_{l=1}^M c_l \rho_l^2
\right) f \cdot f
+ 2 \sum_{l=1}^M c_l \rho_l^2 |g_l|^2
=0, \quad j=1,2,\cdots,M,&\label{eqn: bilinear 3}
\end{flalign}
under the nonzero boundary condition at $\pm\infty$ by the variable transformation
\begin{equation}
    v_{j}(x,t) = \rho_{j}\frac{g_{j}(x,t)}{f(x,t)}  e^{\mathrm{i}\left(k_j x-\left(k_j^3+3k_j\sum_{i=1}^{M}c_i \rho_i^2+3\sum_{i=1}^{M}c_i\rho_i^2 k_i\right) t\right)} , \quad j=1,2,\cdots,M,
\end{equation}
where $f(x,t)$ is a real-valued function, $g_j(x,t)$ is a complex-valued function, and $D$ is Hirota’s bilinear operator \cite{hirota2004direct}.

Then, we define 
\begin{equation*}
  \begin{aligned}
m^{\mathbf{n}} &=\frac{1}{p+q}\left[\prod_{j=1}^{M}\left(-\frac{p-\mathrm{i} k_j}{q+\mathrm{i} k_j}\right)^{n_j}  \right] e^{\xi+\eta}, \\
\xi &=p x+p^2 y+p^3 t+\sum_{j=1}^M\frac{1}{p-\mathrm{i} k_j} r_j +\xi_0(p), \\
\eta &=q x-q^2 y+q^3 t+\sum_{j=1}^M\frac{1}{q+\mathrm{i} k_j} r_j +\eta_0(q),
\end{aligned}
\end{equation*}
where $\mathbf{n}=\left(n_1, n_2, \ldots, n_M\right)$ with $n_j$ being integers, $p, q$   are arbitrary complex constants,  $j=1,2,\dots, M$,  and $ \xi_0(p), \eta_0(q)$ are arbitrary functions of $p$ and $q$ respectively.
Let $\mathcal{A}_i$ and $\mathcal{B}_j$ be differential operators of order $i$ and $j$, respectively, defined by
\begin{equation}
\mathcal{A}_i(p)=\frac{1}{i !}\left[f_1(p) \partial_p\right]^i, \quad \mathcal{B}_j(q)=\frac{1}{j !}\left[f_2(q) \partial_q\right]^j,
\end{equation}
where $ f_1(p), f_2(q) $ are arbitrary functions of $p$ and $q$ respectively.
Then it can be calculated that \cite{ohta2012general} the determinant
\begin{eqnarray}
\tau_{\mathbf{n}} = \det_{1 \leq  \nu, \mu  \leq N}\left( m_{i_\nu, j_\mu}^{\mathbf{n}}  \right)
\end{eqnarray}
where $\left(i_1, i_2, \cdots, i_N\right)$ and $\left(j_1, j_2, \cdots, j_N\right)$ are arbitrary sequences of indices, and the matrix element $m_{i j}^{\mathbf{n}}$ is defined as
\begin{equation} \label{phi equation}
m_{i j}^{\mathbf{n}}=\mathcal{A}_i \mathcal{B}_j m^{\mathbf{n}},
\end{equation}
would satisfy the bilinear equations \cite{shi2026soliton2}
\begin{eqnarray}
&&\left( D_{r_j} D_x - 2 \right) \tau_{\mathbf{n}} \cdot \tau_{\mathbf{n}}
= -2 \tau_{\mathbf{n}+\mathbf{e}_j} \tau_{\mathbf{n}-\mathbf{e}_j},\label{eqn: KP1}
\\
&&\left( D_x^2 - D_y + 2 a_j D_x \right) \tau_{\mathbf{n}+\mathbf{e}_j} \cdot \tau_{\mathbf{n}}
= 0,\label{eqn: KP2}
\\
&&\left( D_x^3 + 3 D_x D_y - 4 D_t
+ 3 a_j \left( D_x^2 + D_y \right)
+ 6 a_j^2 D_x \right) \tau_{\mathbf{n}+\mathbf{e}_j} \cdot \tau_{\mathbf{n}}
= 0,\label{eqn: KP3}
\\
&&\Big(
D_{r_l} \left( D_x^2 - D_y + 2 a_j D_x \right)
- 4 ( D_x + a_j - a_l )
\Big) \tau_{\mathbf{n}+\mathbf{e}_j} \cdot \tau_{\mathbf{n}}
+ 4 ( a_j - a_l ) \tau_{\mathbf{n}+\mathbf{e}_j+\mathbf{e}_l} \cdot \tau_{\mathbf{n}-\mathbf{e}_l}
= 0,\label{eqn: KP4}
\\
&&\left( D_x + a_j - a_l \right) \tau_{\mathbf{n}+\mathbf{e}_j} \cdot \tau_{\mathbf{n}+\mathbf{e}_l}
= ( a_j - a_l ) \tau_{\mathbf{n}+\mathbf{e}_j+\mathbf{e}_l} \tau_{\mathbf{n}},\label{eqn: KP5}
\end{eqnarray}
where $j=1,2,\cdots,M,$ $\mathbf{n}=(n_1,n_2,\cdots,n_M),$ $a_j = \mathrm{i} k_j$ and $\mathbf{e}_j$ is the standard unit vector in $\mathbb{R}^M$.

In what follows, we establish the reductions from the bilinear equations \eqref{eqn: KP1}–\eqref{eqn: KP5} in the KP hierarchy to the bilinear equations \eqref{eqn: bilinear 1}–\eqref{eqn: bilinear 3}, which lead to rogue wave solutions of the multi-component Hirota equation \eqref{eqn: vector-Hirota}. Following the approaches developed in \cite{ohta2012general}, the rogue wave solutions of the multi-component Hirota equation \eqref{eqn: vector-Hirota} can be derived through lengthy but straightforward calculations. For brevity, we only present the conditions required to achieve the reduction, which are as follows:

\begin{itemize}
  \item [i)] {\it Dimension reduction}
  \begin{equation}
\left(-\sum_{l=1}^M c_l \rho_l^2 \partial_{r_l}+\partial_x\right) \tau_{\mathbf{n}}= C \tau_{\mathbf{n}},
  \end{equation}
  where $C$ is some constant.

  \item [ii)] {\it Complex conjugate reduction}

  \begin{equation}\label{complex conjugacy condition}
    \tau_{\mathbf{n}} = \tau_{-\mathbf{n}}^*.
  \end{equation}

  \item [iii)]{\it  Gauge transformation}
\begin{equation}
    x \to x - 3\sum_{j=1}^M c_j\rho_j^2\, t, \qquad t \to t.
\end{equation}

\end{itemize}
After the above reductions and transformations, the bilinear equations \eqref{eqn: KP1}–\eqref{eqn: KP5} in the KP hierarchy are reduced to the bilinear equations \eqref{eqn: bilinear 1}–\eqref{eqn: bilinear 3}. Finally, by applying the simplifications developed in \cite{ohta2012general}, the proof of Theorem~\ref{thm:rogue wave solutions} is completed.

\begin{proof}[\bf Proof of Lemma \ref{lemma:expression of pn}]
Denote
\begin{equation}\label{eq:implicit-pkappa}
    \mathcal{F}_M(p(\kappa))=\mathcal{F}_M(p(0))\,E(\kappa),
\qquad 
E(\kappa)=\frac{1}{M+1}\sum_{j=1}^{M+1}\exp\!\left(\exp\!\left(\frac{2j\pi \mathrm{i}}{M+1}\right)\kappa\right).
\end{equation}
Let
\begin{equation}
    \omega=\exp\!\left(\frac{2\pi \mathrm{i}}{M+1}\right),
\end{equation}
then $\exp\!\left(2j\pi \mathrm{i}/(M+1)\right)=\omega^{j}$ and hence
\begin{equation}
    E(\kappa)=\frac{1}{M+1}\sum_{j=1}^{M+1}e^{\omega^{j}\kappa}.
\end{equation}
Expanding each exponential into its power series, we obtain
\begin{equation} \label{eq:E-series}
E(\kappa)
=\frac{1}{M+1}\sum_{j=1}^{M+1}\sum_{n=0}^{\infty}\frac{(\omega^{j}\kappa)^n}{n!}
=\frac{1}{M+1}\sum_{n=0}^{\infty}\left(\sum_{j=0}^{M}\omega^{jn}\right)\frac{\kappa^n}{n!}.
\end{equation}

Now we evaluate the finite sum $\sum_{j=0}^{M}\omega^{nj}.$ If $(M+1)\mid n$, then $\omega^{n}=1$ and hence
\begin{equation}
    \sum_{j=0}^{M}\omega^{nj}=\sum_{j=0}^{M}1=M+1.
\end{equation}
If $(M+1)\nmid n$, then $\omega^{n}\neq 1$ and
\begin{equation}
    \sum_{j=0}^{M}\omega^{nj}=\sum_{j=0}^{M}(\omega^{n})^{j}
=\frac{1-(\omega^{n})^{M+1}}{1-\omega^{n}}
=\frac{1-\omega^{n(M+1)}}{1-\omega^{n}}
=\frac{1-1}{1-\omega^{n}}=0.
\end{equation}
Consequently,
\begin{equation}
    \frac{1}{M+1}\sum_{j=0}^{M}\omega^{nj}=
\begin{cases}
1, & n\equiv 0 \pmod{M+1},\\
0, & n\not\equiv 0 \pmod{M+1}.
\end{cases}
\end{equation}
Substituting this into \eqref{eq:E-series} yields
\begin{equation}
    E(\kappa)=\sum_{n=0}^{\infty}\frac{\kappa^{(M+1)n}}{((M+1)n)!}.
\end{equation}
In particular, the Taylor expansion of $E(\kappa)$ contains only powers $\kappa^{(M+1)n}$ and no other powers.

Denote $u(\kappa)=p(\kappa)-p_0$. We have the following expansion
\begin{equation}
    u(\kappa)=\sum_{n=1}^{\infty}p_n\kappa^n.
\end{equation}
Since \(p_0\) is a zero of multiplicity \(M\) of
\(\mathcal F_M'(p)\), we have
\begin{equation}
\mathcal F_M^{(r)}(p_0)=0,
\quad r=1,2,\ldots,M,
\qquad
\mathcal F_M^{(M+1)}(p_0)\neq0.
\end{equation}
Therefore, the Taylor expansion of \(\mathcal F_M\) at \(p=p_0\)
gives
\begin{equation}\label{eq:F-Taylor}
\mathcal{F}_M(p(\kappa))
=\mathcal{F}_M(p_0)
+\sum_{r=M+1}^{\infty}\frac{\mathcal{F}_M^{(r)}(p_0)}{r!}\,u^r.
\end{equation}
Substituting \eqref{eq:F-Taylor} into \eqref{eq:implicit-pkappa} and cancelling the constant term
$\mathcal{F}_M(p_0)$ on both sides yield
\begin{equation}\label{eq:coef-identity}
\sum_{r=M+1}^{\infty}\frac{\mathcal{F}_M^{(r)}(p_0)}{r!}\,u(\kappa)^r
=
\mathcal{F}_M(p_0)\sum_{s=1}^{\infty}\frac{\kappa^{s(M+1)}}{(s(M+1))!}.
\end{equation}

Comparing the coefficients of $\kappa^{M+1}$ in \eqref{eq:coef-identity}, the right-hand side
contributes $\mathcal{F}_M(p_0)/(M+1)!$. On the left-hand side, the smallest possible power of
$\kappa$ arises from
\[
u(\kappa)^{M+1}=\left(p_1\kappa+p_2\kappa^2+p_3\kappa^3+\cdots\right)^{M+1},
\]
whose $\kappa^{M+1}$-coefficient equals $p_1^{M+1}$. Hence,
\[
\frac{\mathcal{F}_M^{(M+1)}(p_0)}{(M+1)!}\,p_1^{M+1}
=
\frac{\mathcal{F}_M(p_0)}{(M+1)!},
\]
which proves \eqref{eq:a1-formula}. This yields $M+1$ possible values for $p_1$, and we may choose any one of them.

Fix $n\ge 2$ and compare the coefficients of $\kappa^{M+n}$ in \eqref{eq:coef-identity}. From the right-hand side of \eqref{eq:coef-identity}, the coefficient of $\kappa^{M+n}$ equals
\begin{equation}
    \begin{cases}
\dfrac{\mathcal{F}_M(p_0)}{(M+n)!}, & M+n \equiv 0 \pmod{M+1}, \\[1ex]
0, & M+n \not \equiv 0 \pmod{M+1}.
\end{cases}
\end{equation}
On the left-hand side, for each $r\ge M+1$, the coefficient of $\kappa^{M+n}$ in $u(\kappa)^r$ is
\begin{equation}\label{eqn:coefficients in the left hand side}
    \sum_{\substack{j_1+\cdots+j_r=M+n\\ j_\ell\ge 1}}
p_{j_1}p_{j_2}\cdots p_{j_r}.
\end{equation}
The only way this sum can involve $p_n$ is through the case $r=M+1$. Indeed, if $r\ge M+2$, then
from $j_1+\cdots+j_r=M+n$ with $j_\ell\ge 1$, we obtain
$\max j_\ell\le M+n-(r-1)\le n-1$, so only $p_1,\dots,p_{n-1}$ appear.

For $r=M+1$, the coefficient of $\kappa^{M+n}$ in $u(\kappa)^{M+1}$ is
\begin{equation}
\sum_{\substack{j_1+\cdots+j_{M+1}=M+n\\ j_\ell\ge 1}}
p_{j_1}\cdots p_{j_{M+1}}.
\end{equation}
The terms containing \(p_n\) arise by taking one index equal to
\(n\) and the remaining \(M\) indices equal to \(1\). Their total
contribution is therefore
\((M+1)p_1^M p_n\).
Separating this term from the remaining contributions (which depend only on $p_1,\dots,p_{n-1}$), the coefficient of $\kappa^{M+n}$ in $u(\kappa)^{M+1}$ can be rewritten as
\begin{equation}\label{eq:key-splitting}
(M+1)p_1^{\,M}p_n
+\!\!\!\sum_{\substack{j_1+\cdots+j_{M+1}=M+n\\ 1\le j_\ell\le n-1}}
\!\!\! p_{j_1}\cdots p_{j_{M+1}}.
\end{equation}

Collecting the coefficient of $\kappa^{M+n}$ in \eqref{eq:coef-identity} and using
\eqref{eq:key-splitting}, we arrive at
\begin{align}
&\frac{\mathcal{F}_M^{(M+1)}(p_0)}{(M+1)!}(M+1)p_1^{\,M}p_n
+\frac{\mathcal{F}_M^{(M+1)}(p_0)}{(M+1)!}
\!\!\!\sum_{\substack{j_1+\cdots+j_{M+1}=M+n\\ 1\le j_\ell\le n-1}}
\!\!\! p_{j_1}\cdots p_{j_{M+1}} +\sum_{r=M+2}^{M+n}\frac{\mathcal{F}_M^{(r)}(p_0)}{r!}
\!\!\!\sum_{\substack{j_1+\cdots+j_r=M+n\\ 1\le j_\ell\le n-1}}
\!\!\! p_{j_1}\cdots p_{j_r}\nonumber\\
& 
=
\begin{cases}
\dfrac{\mathcal{F}_M(p_0)}{(M+n)!}, & M+n \equiv 0 \pmod{M+1},\\[1ex]
0, & M+n \not \equiv 0 \pmod{M+1},
\end{cases}
\label{eq:coef-Mn-balance}
\end{align}
which is equivalent to \eqref{eq:pn-recursive-simplified}.
This completes the proof.
\end{proof}

\begin{proof}[\bf Proof of Lemma \ref{lemma:expansion of s_r}]
Recall that $p(\kappa)$ is defined implicitly by
\begin{equation}
\mathcal{F}_M\bigl(p(\kappa)\bigr)
=\mathcal{F}_M(p_0)\,E(\kappa),\qquad 
E(\kappa)=\sum_{r=0}^{\infty}\frac{\kappa^{(M+1)r}}{((M+1)r)!}.
\end{equation}
Since $p_0$ is a root of $\mathcal{F}_M'(p)=0$ with a multiplicity of $M$, we have
\begin{equation}\label{eq:multi-root-assumption}
\mathcal{F}_M^{(r)}(p_0)=0\quad (r=1,2,\dots,M), 
\qquad \mathcal{F}_M^{(M+1)}(p_0)\neq 0.
\end{equation}

In what follows, we express the generating function of
\(\{s_r\}_{r\ge1}\) in terms of \(E(\kappa)\).
Note that $(M+1)s_r$ is defined by the following expansion
\begin{equation}\label{eq:s-def}
\ln\!\left[\frac{1}{\kappa^{M+1}}
\left(\frac{p(0)+p(0)^*}{p_1}\right)^{M+1}
\left(\frac{p(\kappa)-p(0)}{p(\kappa)+p(0)^*}\right)^{M+1}\right]
=\sum_{r=1}^{\infty}(M+1)s_r\,\kappa^r.
\end{equation}

We now show that the ratio
\begin{equation}\label{eq:ratio}
\left(\frac{p(\kappa)-p(0)}{p(\kappa)+p(0)^*}\right)^{M+1},
\end{equation}
depends on $\kappa$ only through $E(\kappa)$.
Recall that $\mathcal{F}_M(p)$ is a rational function of the form
\begin{equation}
    \mathcal{F}_M(p)=p+\sum_{m=1}^{M}\frac{c_m}{p-a_m},
\end{equation}
where $c_m \in \mathbb{R}$ and $a_m \in \mathrm{i} \mathbb{R}$ with $ a_m \not = a_n$ if $m \not = n$. 

Define
\begin{equation}
\Pi(p):=\prod_{m=1}^{M}(p-a_m).
\end{equation}
Since \(\Pi(p)\) is monic and
\[
\mathcal F_M(p)=p+\sum_{m=1}^{M}\frac{c_m}{p-a_m},
\]
the polynomial
\begin{equation}
P_{p_0}(p)
:=
\Pi(p)\bigl[\mathcal F_M(p)-\mathcal F_M(p_0)\bigr]
\end{equation}
is monic and has degree \(M+1\).
Since \(p_0\) is a zero of multiplicity \(M\) of
\(\mathcal F_M'(p)\), we have
\begin{equation}
\mathcal F_M^{(r)}(p_0)=0,
\quad r=1,\ldots,M,
\qquad
\mathcal F_M^{(M+1)}(p_0)\neq0.
\end{equation}
Hence \(P_{p_0}(p)\) has a zero of multiplicity \(M+1\) at
\(p=p_0\). Since it is monic and has degree \(M+1\), it follows that
\begin{equation}\label{eq:P1-factor}
\Pi(p)\bigl[\mathcal F_M(p)-\mathcal F_M(p_0)\bigr]
=
(p-p_0)^{M+1}.
\end{equation}
Substituting \(p=p(\kappa)\) and using
\(\mathcal F_M(p(\kappa))=\mathcal F_M(p_0)E(\kappa)\), we obtain
\begin{equation}\label{eq:key-factor1-Q}
\bigl(p(\kappa)-p_0\bigr)^{M+1}
=
\mathcal F_M(p_0)\bigl(E(\kappa)-1\bigr)
\Pi\bigl(p(\kappa)\bigr).
\end{equation}
Since \(c_m\in\mathbb R\) and \(a_m\in\mathrm{i}\mathbb R\), we have
\begin{equation}
\mathcal F_M(-p^*)=-\mathcal F_M(p)^*.
\end{equation}
Consequently, \(q_0:=-p_0^*\) is also a zero of multiplicity \(M\)
of \(\mathcal F_M'(p)\). By the same argument, the monic polynomial
\begin{equation}
P_{q_0}(p)
:=
\Pi(p)\bigl[\mathcal F_M(p)-\mathcal F_M(q_0)\bigr]
\end{equation}
satisfies
\begin{equation}
\Pi(p)\bigl[\mathcal F_M(p)-\mathcal F_M(q_0)\bigr]
=
(p-q_0)^{M+1}.
\end{equation}
Therefore,
\begin{equation}\label{eq:key-factor2-Q}
\bigl(p(\kappa)+p_0^*\bigr)^{M+1}
=
\bigl[
\mathcal F_M(p_0)E(\kappa)-\mathcal F_M(-p_0^*)
\bigr]
\Pi\bigl(p(\kappa)\bigr).
\end{equation}

Finally, dividing \eqref{eq:key-factor1-Q} by
\eqref{eq:key-factor2-Q}, the common factor
\(\Pi(p(\kappa))\) cancels, and we obtain
\begin{equation}\label{eq:final-ratio-s}
\left(
\frac{p(\kappa)-p_0}{p(\kappa)+p_0^*}
\right)^{M+1}
=
\frac{
\mathcal F_M(p_0)\bigl(E(\kappa)-1\bigr)
}{
\mathcal F_M(p_0)E(\kappa)-\mathcal F_M(-p_0^*)
},
\end{equation}
which depends on $\kappa$ only through $E(\kappa)$. Then, one can obtain \begin{equation}
\ln\!\left[\frac{1}{\kappa^{M+1}}
\left(\frac{p(0)+p(0)^*}{p_1}\right)^{M+1}
\frac{ \mathcal{F}_M(p(0)) (E(\kappa)-1)}{\mathcal{F}_M(p_0)E(\kappa)-\mathcal{F}_M(-p_0^*)}\right]
=\sum_{r=1}^{\infty}(M+1)s_r\,\kappa^r.
\end{equation}
Therefore $(M+1)s_j=0$ whenever $j\not\equiv0\pmod{M+1}$, i.e.,
\begin{equation}
    s_j=0,\quad j\not\equiv0\pmod{M+1}.
\end{equation}
This completes the proof.

\end{proof}

\section{Proof of Theorem \ref{thm:Rogue wave patterns}}
In this Appendix, we provide the rogue wave patterns given in Theorem \ref{thm:Rogue wave patterns}. Assume $\left|a_m\right|$ is large and all remaining parameters are $O(1)$ in the $l$-th type $\mathcal{N}_l$-th order rogue wave solutions
\begin{equation} 
  v_{1,\mathcal{N}_l}(x,t), \quad
  v_{2,\mathcal{N}_l}(x,t), \quad
  \ldots, \quad
  v_{M,\mathcal{N}_l}(x,t),
\end{equation}
which are expressed in terms of the $\tau$-function $\sigma_{\mathbf{n}}(\bar{x},t)=\det\left(\tau_{\mathbf{n}}^{[l,l]}\left(x-3\sum_{i=1}^{M}c_i \rho_i^2 t,t\right)\right)$ defined in Theorem~\ref{thm:rogue wave solutions}. Therefore,  for simplicity, we consider the rogue wave patterns in the \((\bar{x},t)\)-plane, where $\bar{x}=x-3\sum_{i=1}^{M}c_i \rho_i^2\, t$. 
The proof proceeds by considering two distinct regions, namely, the outer region and the inner region.

{\bf Proof of the outer region}. We first consider the case when $\left(\bar{x},t\right)$ is far away from the origin and $(\bar{x}^2+t^2)^{1/2}=O(|a_m|^{1/m})$. In this circumstance, we have
\begin{equation} \label{Simplified Schur polynomial-asmptotic}
\begin{aligned}
S_j\left(\mathbf{x}^{+}(\mathbf{n})+\nu \mathbf{s}\right)=S_j\left(x_1^{+}, x_2^{+}, \cdots,x_{M}^{+},\nu s_{M+1}, x_{M+2}^{+}, \cdots, x_{2M+1}^{+}, \nu s_{2(M+1)}, \cdots, x_m^{+}+\nu s_m, \cdots\right) 
\sim   S_j(\mathbf{v}),
\end{aligned}\end{equation}
where
$$
\mathbf{v}=\left(p_1 \bar{x}+ 3p_0^2p_1  t, 0, \cdots, 0, a_m, 0, \cdots\right).
$$
As shown in Lemma~\ref{lemma:expansion of s_r}, the parameters satisfy $s_1=s_2=\cdots=s_{M}
=s_{M+2}=\cdots=s_{2M+1}
=s_{2M+3}=\cdots=0.$
Recalling the definition of Schur polynomials, one can rewrite them in the form
\begin{equation}\label{Schur polynomial VS p1-3com}
S_j(\mathbf{v})=a_m^{\,j/m}\, p_j^{[m]}(z),
\end{equation}
where the scaled variable $z$ is given by
\begin{equation}
z=a_m^{-1/m}\bigl(\alpha_1 \bar{x}+\beta_1  t\bigr)
= a_m^{-1/m}\bigl(p_1 \bar{x} +  3p_0^2p_1   t\bigr).
\end{equation}
As a consequence, the determinant built from Schur polynomials admits the following asymptotic representation:
\begin{equation}\label{Schur polynomial VS WH x+}
 \det_{1 \leq i, j \leq N}
 \!\left[
 S_{(M+1)(i-1)+l-\nu_j}\!\left(\mathbf{x}^{+}(\mathbf{n})+\nu_j \mathbf{s}\right)
 \right]
 \sim
 \left(c_N^{[m,M+1,l]}\right)^{-1}
 a_m^{\frac{\left(N(2l+(M+1)(N-1)) - 2\sum_{j=1}^{N}\nu_j\right)}{2m}}
 W_{N}^{[m,M+1,l]}(z),
\end{equation}
and, analogously,
\begin{equation}
 \det_{1 \leq i, j \leq N}
 \!\left[
 S_{(M+1)(i-1)+l-\nu_j}\!\left(\mathbf{x}^{-}(\mathbf{n})+\nu_j \mathbf{s}^*\right)
 \right]
 \sim
 \left(c_N^{[m,M+1,l]}\right)^{-1}
 \left(a_m^*\right)^{\frac{\left(N(2l+(M+1)(N-1)) - 2\sum_{j=1}^{N}\nu_j\right)}{2m}}
 W_{N}^{[m,M+1,l]}(z^*).
\end{equation}
Here, we adopt the convention that $S_j\equiv 0$ for all $j<0$.

Next, by applying the Laplace expansion, we rewrite the $\tau$-function $\sigma_{\mathbf{n}}$ in the following form:
\begin{equation}\label{Laplace expansion of tau function}
  \begin{gathered}
\sigma_{\mathbf{n}}
=
\sum_{0 \leq \nu_1<\nu_2<\cdots<\nu_N \leq (M+1) N-1}
 \det_{1 \leq i, j \leq N}
 \!\left[
 \left(h_0\right)^{\nu_j}
 S_{(M+1)(i-1)+l-\nu_j}\!\left(\mathbf{x}^{+}(\mathbf{n})+\nu_j \mathbf{s}\right)
 \right] \\
\times
 \det_{1 \leq i, j \leq N}
 \!\left[
 \left(h_0^*\right)^{\nu_j}
 S_{(M+1)(i-1)+l-\nu_j}\!\left(\mathbf{x}^{-}(\mathbf{n})+\nu_j \mathbf{s}^*\right)
 \right],
\end{gathered}
\end{equation}
where $h_0=p_1/(p_0+p_0^*)$.

It is readily seen that, among all possible index selections $\{\nu_j\}$, the highest-order contribution in the parameter $a_m$ arises from the choice $\nu_j=j-1$. Consequently, we obtain the asymptotic behavior
\begin{equation}
\sigma_{\mathbf{n}}
\sim
|\alpha|^2\,|a_m|^{\frac{N(2l+M(N-1))}{m}}
\left|W_{N}^{[m,M+1,l]}(z)\right|^2,
\end{equation}
where
\[
\alpha = h_0^{N(N-1)/2}\left(c_N^{[m,M+1,l]}\right)^{-1}.
\]
The above asymptotic analysis shows that the leading-order term of $\sigma_{\mathbf{n}}$ is independent of the multi-index $\mathbf{n}$. Therefore, when $(\bar{x},t)$ is away from the point $(\hat{x}_0,\hat{t}_0)$, which is associated with the zeros of the polynomial $W_{N}^{[m,M+1,l]}(z)$ through
\begin{equation}
    z_0 = a_m^{-1/m}\left(p_1 \hat{x}_0 + 3p_0^2p_1 \hat{t}_0\right),
\end{equation}
we have
\begin{equation}
   \frac{\sigma_{\mathbf{n}_i}}{\sigma_{\mathbf{n}_0}} \sim 1,
   \qquad i=1,2,\ldots,M,
   \qquad \text{for } |a_m|\gg 1.
\end{equation}

However, when $(\bar{x},t)$ is close to $(\hat{x}_0,\hat{t}_0)$, the coefficient of the highest-order term in $a_m$ vanishes. To handle this situation, it is necessary to take into account lower-order contributions in $a_m$, which requires a more refined asymptotic analysis. In this regime, namely when $(\bar{x},t)$ is near $(\hat{x}_0,\hat{t}_0)$, we obtain
\begin{eqnarray*}
S_j\!\left(\mathbf{x}^{+}(\mathbf{n})+\nu \mathbf{s}\right)
&=&S_j\!\left(x_1^{+}, x_2^{+}, \ldots, x_{M}^{+}, \nu s_{M+1},
x_{M+2}^{+}, \ldots, x_{2M+1}^{+}, \nu s_{2(M+1)}, \ldots, x_m^{+}+\nu s_m, \ldots \right)
\\
&=&\Bigl[S_j(\hat{\mathbf{v}})
+\hat{x}_2^{+}\!\left(\hat{x}_0,\hat{t}_0\right) S_{j-2}(\hat{\mathbf{v}})\Bigr]
\Bigl[1+O\!\left(|a_m|^{-2/m}\right)\Bigr],
\qquad |a_m|\gg 1,
\end{eqnarray*}
where
\begin{eqnarray*}
\hat{\mathbf{v}}
&=&\Bigl(p_1 \bar{x}+ 3p_0^2p_1 t+\sum_{j=1}^{M} n_j \theta_{1j},
\,0,\,\ldots,\,0,\,a_m,\,0,\,\ldots \Bigr),\\
\hat{x}_2^{+}(x,t)
&=&p_2 \bar{x}+\left(3p_0\left( p_1^2+p_0 p_2\right)\right)  t,
\end{eqnarray*}
with the parameter $a_1$ in $x_1^{+}$ set to zero. 
In analogy with \eqref{Schur polynomial VS p1-3com}, we further have
\begin{equation}
S_j(\hat{\mathbf{v}})=a_m^{j/m}\, p_j^{[m]}(\hat{z}),
\end{equation}
where
\begin{equation}
\hat{z}
=a_m^{-1/m}\!\left(p_1 \bar{x} + 3p_0^2p_1 t
+\sum_{j=1}^{M} n_j \theta_{1j}\right).
\end{equation}

Under these circumstances, there are two possible index choices $\nu_j$ that yield leading-order contributions in $a_m$ for $\sigma_{\mathbf{n}}$. One corresponds to $\nu=(0,1,\ldots,N-1),$
while the other is given by $\nu=(0,1,\ldots,N-2,N).$

\paragraph{Case (1): the index choice $\nu_j=j-1$}
For this choice of indices, the leading-order contribution arises from two distinct parts. 
The first part originates from the term $S_j(\hat{\mathbf{v}})$. 
By extracting the dominant contribution involving $\mathbf{x}^{+}(\mathbf{n})$ from the asymptotic expansion of the determinants, we obtain
\begin{equation}
  \alpha\, a_m^{\frac{N(2l+M(N-1))}{2m}}
  W_{N}^{[m,M+1,l]}(\hat{z})
  \left[1+O\!\left(|a_m|^{-2/m}\right)\right].
\end{equation}
Next, we expand $W_{N}^{[m,M+1,l]}(\hat{z})$ in a neighborhood of $z_0$. 
Since $W_{N}^{[m,M+1,l]}(z_0)=0$, a first-order Taylor expansion yields
\begin{equation}
  W_{N}^{[m,M+1,l]}(\hat{z})
  = a_m^{-1/m}
  \Bigl[
      p_1\bigl(\bar{x}-\hat{x}_0\bigr)
      + 3p_0^2p_1 \bigl(t-\hat{t}_0\bigr)
      +\sum_{j=1}^{M} n_j \theta_{1j}
  \Bigr]
  \left[W_{N}^{[m,M+1,l]}\right]'(z_0)
  \left[1+O \left(|a_m|^{-1/m}\right)\right].
\end{equation}
Consequently, the corresponding leading-order term in $a_m$ is given by
\begin{equation}\label{1 index-part 1}
  \alpha\, a_m^{\frac{N(2l+M(N-1))-2}{2m}}
  \Bigl[
      p_1\bigl(\bar{x}-\hat{x}_0\bigr)
      + 3p_0^2p_1 \bigl(t-\hat{t}_0\bigr)
      +\sum_{j=1}^{M} n_j \theta_{1j}
  \Bigr]
  \left[W_{N}^{[m,M+1,l]}\right]'(z_0)
  \left[1+O\!\left(|a_m|^{-1/m}\right)\right].
\end{equation}

The second leading-order contribution comes from those determinants containing the factor
$\hat{x}_2^{+}(\hat{x}_0,\hat{t}_0)\,S_{j-2}(\hat{\mathbf{v}})$. 
More precisely, this contribution takes the form
\begin{eqnarray}
  \sum_{j=1}^{N}
  \det_{1\le i\le N}
  &&\left[
      S_{(M+1)(i-1)+l}(\hat{\mathbf{v}}),\ldots,
      S_{(M+1)(i-1)+l-j+2}(\hat{\mathbf{v}}),
      S_{(M+1)(i-1)+l-j-1}(\hat{\mathbf{v}}),\right. \nonumber\\
     && \left.
      S_{(M+1)(i-1)+l-j}(\hat{\mathbf{v}}),\ldots,
      S_{(M+1)(i-1)+l-N+1}(\hat{\mathbf{v}})
  \right] \times
  \hat{x}_2^{+}(\hat{x}_0,\hat{t}_0)\,
  h_0^{N(N-1)/2}
  \left[1+O\!\left(|a_m|^{-1/m}\right)\right].
  \label{1 index-part 2}
\end{eqnarray}

Combining \eqref{1 index-part 1} and \eqref{1 index-part 2}, we obtain the complete leading-order contribution in $a_m$
(from the first determinant in \eqref{Laplace expansion of tau function} involving $\mathbf{x}^{+}(\mathbf{n})$) as \cite{yang2023rogue}
\begin{equation*}
  \alpha\, a_m^{\frac{N(2l+M(N-1))-2}{2m}}
  \Bigl[
      p_1\bigl(\bar{x}-\hat{x}_0\bigr)
      + 3p_0^2p_1 \bigl(t-\hat{t}_0\bigr)
      +\sum_{j=1}^{M} n_j \theta_{1j}
      +\Delta_l
  \Bigr]
  \left[W_{N}^{[m,M+1,l]}\right]'(z_0)
  \left[1+O\!\left(|a_m|^{-1/m}\right)\right],
\end{equation*}
where
\begin{equation}\label{eqn:def of delta_l}
  \Delta_l
  =
  \frac{\hat{x}_2^{+}(\hat{x}_0,\hat{t}_0)}{a_m^{1/m}}
  \frac{
       \widetilde{W}_{N}^{[m,M+1,l]} (z_0)
  }{
      \left[W_{N}^{[m,M+1,l]}\right]'(z_0)
  },
\end{equation}
with
\begin{eqnarray}\label{eqn: two term delta}
    \widetilde{W}_{N}^{[m,M+1,l]} (z_0)=c_N^{[m,M+1,l]}\sum_{j=1}^{N}
  \det_{1\le i\le N}
  &&\left[
      p_{(M+1)(i-1)+l}^{[m]}(z_0),\ldots,
      p_{(M+1)(i-1)+l-j+2}^{[m]}(z_0),
      p_{(M+1)(i-1)+l-j-1}^{[m]}(z_0),\right.\nonumber\\
      &&\left.
      p_{(M+1)(i-1)+l-j}^{[m]}(z_0),\ldots,
      p_{(M+1)(i-1)+l-N+1}^{[m]}(z_0)
  \right].
\end{eqnarray}
Since $\hat{x}_2^{+}(\hat{x}_0,\hat{t}_0)
  = O \left(|a_m^{1/m}|\right),$
it follows that $\Delta_l=O(1)$. 
As in \cite{yang2023rogue}, this term can be absorbed into a redefinition of the parameters
$(\hat{x}_0,\hat{t}_0)$, leading to
\begin{equation}
  \alpha\, a_m^{\frac{N(2l+M(N-1))-2}{2m}}
  \Bigl[
      p_1\bigl(\bar{x}-\hat{x}_0\bigr)
      + 3p_0^2p_1 \bigl(t-\hat{t}_0\bigr)
      +\sum_{j=1}^{M} n_j \theta_{1j}
  \Bigr]
  \left[W_{N}^{[m,M+1,l]}\right]'(z_0)
  \left[1+O\!\left(|a_m|^{-1/m}\right)\right],
\end{equation}
where $(\hat{x}_0,\hat{t}_0)$ is defined as
\begin{eqnarray}
\hat{x}_0&=&\Re\!\left[\frac{1}{p_1}\left(z_0\,a_m^{\frac{1}{m}}-\Delta_l\right)\right]- 
\frac{\Re(p_0^2)}{\Im(p_0^2)}\Im\!\left[\frac{1}{p_1}\left(z_0\,a_m^{\frac{1}{m}}-\Delta_l\right)\right] = \frac{1}{\Im(p_0^2)}\Re\!\left[\frac{\mathrm{i}\left(p_0^2\right)^*}{p_1} \left(z_0\,a_m^{\frac{1}{m}}-\Delta_l\right)\right], \\
\hat{t}_0&=&\frac{1}{3\Im(p_0^2)}\Im\!\left[\dfrac{1}{p_1} \left(z_0\,a_m^{\frac{1}{m}}-\Delta_l\right)\right].
\end{eqnarray}
By the same argument, the second determinant in
\eqref{Laplace expansion of tau function} involving $\mathbf{x}^{-}(\mathbf{n})$
and corresponding to the index choice $\nu_j=j-1$ contributes
\begin{equation}
  \alpha^{*} (a_m^{*})^{\frac{N(2l+M(N-1))-2}{2m}}
  \Bigl[
      p_1^{*}\bigl(\bar{x}-\hat{x}_0\bigr)
      +3(p_0^2)^*p_1^*\bigl(t-\hat{t}_0\bigr)
      -\sum_{j=1}^{M} n_j \theta_{1j}^{*}
  \Bigr]
  \left[W_{N}^{[m,M+1,l]}\right]'(z_0^{*})
  \left[1+O\!\left(|a_m|^{-1/m}\right)\right].
\end{equation}

\paragraph{Case (2): the index choice $\nu=(0,1,\ldots,N-2,N)$}
For this choice of indices, the dominant contribution in $a_m$ can be evaluated in the same manner as in
\eqref{Schur polynomial VS WH x+}. 
More precisely, the leading-order term is given by
\begin{equation}
  |h_0 |^2 |\alpha|^2
  |a_m|^{\frac{N(2l+M(N-1))-2}{m}}
  \left|\left[W_{N}^{[m,M+1,l]}\right]'(z_0)\right|^2
  \left[1+O\!\left(|a_m|^{-1/m}\right)\right].
\end{equation}

Collecting the two leading contributions corresponding to the index choices
$\nu_j=j-1$ and $\nu=(0,1,\ldots,N-2,N)$, we arrive at the following asymptotic expression for the
$\tau$-function $\sigma_{\mathbf{n}}$:
\begin{eqnarray}
  \sigma_{\mathbf{n}}(\bar x,t)
  &=&
  |\alpha|^2
  \left|
      \left[W_{N}^{[m,M+1,l]}\right]'(z_0)
  \right|^2
  |a_m|^{\frac{N(2l+M(N-1))-2}{m}}
  \nonumber\\
  && \times
  \Bigl(
      \bigl[
          p_1(\bar{x}-\hat{x}_0)
          +3p_0^2p_1(t-\hat{t}_0)
          +\sum_{j=1}^{M} n_j\theta_{1j}
      \bigr] 
      \bigl[
          p_1^{*}(\bar{x}-\hat{x}_0)+
          3(p_0^2)^*p_1^*(t-\hat{t}_0)
          -\sum_{j=1}^{M} n_j\theta_{1j}^{*}
      \bigr]
      +|h_0|^2
  \Bigr)
  \nonumber\\
  && \times
  \left[1+O\!\left(|a_m|^{-1/m}\right)\right].
  \label{asymptotics of tau fucntion}
\end{eqnarray}

Finally, under the assumption that all nonzero roots of the generalized Wronskian--Hermite polynomial
$W_{N}^{[m,M+1,l]}(z)$ are simple, the leading-order term in \eqref{asymptotics of tau fucntion}
does not vanish. 
Consequently, in a neighborhood of $(\hat{x}_0,\hat{t}_0)$, the $\mathcal{N}_l$-th order rogue wave is
asymptotically approximated by a fundamental rogue wave of the multi-component Hirota equation
described in Theorem~\ref{thm:Rogue wave patterns}, with an error of order
$O(|a_m|^{-1/m})$.

{\bf Proof of the inner region}.
To investigate the patterns of the $l$-th type $\mathcal{N}_l$-th order rogue waves of the multi-component Hirota equation in the regime $|a_m|\gg 1$ within the inner region characterized by $\bar{x}^2+t^2=O(1)$, we first adopt a procedure similar to that in \cite{ohta2012general}. Specifically, we rewrite the determinant $\tau_{\mathbf{n}}$ as a block determinant of size $(M+2)N\times(M+2)N$,
\begin{equation}\label{three expansion form}
\tau_{\mathbf{n}}
=
\left|
\begin{array}{cc}
\mathbf{O}_{N\times N} & \Phi_{N\times (M+1)N} \\
-\Psi_{(M+1)N\times N} & \mathbf{I}_{(M+1)N\times (M+1)N}
\end{array}
\right|,
\end{equation}
where the block entries are given by
\[
\Phi_{i,j}
=
\left(h_0\right)^{j-1}
S_{(M+1)(i-1)+l+1-j}\!\left[\mathbf{x}^{+}(\mathbf{n})+(j-1)\mathbf{s}\right],
\]
\[
\Psi_{i,j}
=
\left(h_0^*\right)^{i-1}
S_{(M+1)(j-1)+l+1-i}\!\left[\mathbf{x}^{-}(\mathbf{n})+(i-1)\mathbf{s}^*\right].
\]

It is evident that every entry in the determinant \eqref{three expansion form} is a polynomial in the large parameter $a_m$. To make this polynomial dependence explicit, we note that the Schur polynomials $S_j(\mathbf{x}^{+}+\nu\mathbf{s})$ and $S_j(\mathbf{x}^{-}+\nu\mathbf{s}^*)$ admit the expansions
\begin{equation}\label{S expand}
S_j\!\left(\mathbf{x}^{+}+\nu \mathbf{s}\right)
=
\sum_{q=0}^{\lfloor j/m \rfloor}
\frac{a_m^{\,q}}{q!}
S_{j-qm}\!\left(\mathbf{y}^{+}+\nu \mathbf{s}\right),
\quad
S_j\!\left(\mathbf{x}^{-}+\nu \mathbf{s}^*\right)
=
\sum_{q=0}^{\lfloor j/m \rfloor}
\frac{(a_m^*)^{\,q}}{q!}
S_{j-qm}\!\left(\mathbf{y}^{-}+\nu \mathbf{s}^*\right),
\end{equation}
where $\lfloor a \rfloor$ denotes the greatest integer less than or equal to $a$, and
\begin{equation}\label{transformation between x and y}
\mathbf{x}^{+}
=
\mathbf{y}^{+}+(0,\ldots,0,a_m,0,\ldots),
\qquad
\mathbf{x}^{-}
=
\mathbf{y}^{-}+(0,\ldots,0,a_m^*,0,\ldots).
\end{equation}

In principle, the leading-order behavior of $\tau_{\mathbf{n}}$ with respect to $a_m$ could be obtained by retaining only the highest-degree terms in $a_m$ from each matrix entry. However, a direct implementation of this idea leads to a vanishing determinant. To overcome this difficulty, we follow the strategy developed in \cite{yang2023rogue,zhang2022rogue,lin2024rogue}, performing suitable row and column operations to extract the correct leading-order contribution.
At the end of these row and column operations, $\tau_{\mathbf{n}}$ can be reduced to
the form
	\begin{equation}\label{rewrite form}
		\tau_{\mathbf{n}}=\beta \left|a_m\right|^{\gamma}\left|\begin{array}{cc}
			\mathbf{O}_{\bar{N}_l\times \bar{N}_l} & \widehat{\Phi}_{\bar{N}_l \times \widehat{N}} \\
			-\widehat{\Psi}_{\widehat{N} \times \bar{N}_l} & \mathbf{I}_{\widehat{N} \times \widehat{N}}
		\end{array}\right|\left[1+O\left(|a_m|^{-1}\right)\right],
	\end{equation}
where $\beta \not = 0, \gamma>0$  are constants, $\bar{N}_l=\displaystyle{\sum_{n=1}^M}N_{n,l}$,  $\displaystyle{\widehat{N}=\max_{1\leq n \leq M}\left((M+1)N_{n,l}-M + n\right)}$,
	\begin{equation}
	\begin{aligned}
		&\widehat{\Phi}=\left(\begin{array}{llll}
			\widehat{\Phi}_{N_{1,l} \times \widehat{N}}^{(1)} \\
			\widehat{\Phi}_{N_{2,l} \times \widehat{N}}^{(2)}\\
			\vdots \\
			\widehat{\Phi}_{N_{M,l} \times \widehat{N}}^{(M)}
		\end{array}\right), \quad \widehat{\Psi}=\left(\begin{array}{llll}
			\widehat{\Psi}_{\widehat{N} \times N_{1,l}}^{(1)} & \widehat{\Psi}_{\widehat{N} \times N_{2,l}}^{(2)} &
			\cdots&
			\widehat{\Psi}_{\widehat{N} \times N_{M,l}}^{(M)}
		\end{array}\right),
\\
		&\widehat{\Phi}_{i, j}^{(I)}=\left(h_0\right)^{j-1} S_{(M+1)(i-1)+I-j+1}\left[\mathbf{y}^{+}(\mathbf{n})+\left(j-1+\nu_0\right) \mathbf{s}\right],
 \\
		&\widehat{\Psi}_{i, j}^{(J)}=\left(h_0^*\right)^{i-1} S_{(M+1)(j-1)+J-i+1}\left[\mathbf{y}^{-}(\mathbf{n})+\left(i-1+\nu_0\right) \mathbf{s}^*\right],
\end{aligned}
\end{equation}
with $\nu_0=N-\bar{N}_l$ and the values of $N_{n,l}\in \Z_{\ge0}$ referring to the value of $N_n$ against $l\in\{1,2,\ldots,M\}$ given in Theorem \ref{lemma:root structure of WH polynomials}.
Since the rogue wave solutions are independent of the constants $\beta$ and $\gamma$, we can rewrite \eqref{rewrite form} into an $M \times M$ block determinant
\begin{equation} \label{pattern tau-block matrix}
\tau_{\mathbf{n}}=\det\left(\begin{array}{llll}
\tau_{\mathbf{n}}^{[1,1]} & \tau_{\mathbf{n}}^{[1,2]}&\cdots&\tau_{\mathbf{n}}^{[1,M]} \\
\tau_{\mathbf{n}}^{[2,1]} & \tau_{\mathbf{n}}^{[2,2]}&\cdots&\tau_{\mathbf{n}}^{[2,M]} \\ \vdots & \vdots& \ddots &\vdots \\\tau_{\mathbf{n}}^{[M,1]} & \tau_{\mathbf{n}}^{[M,2]}&\cdots&\tau_{\mathbf{n}}^{[M,M]}
\end{array}\right)\left[1+O\left(|a_m|^{-1}\right)\right],
\end{equation}
where 
\begin{eqnarray}
\tau_{\mathbf{n}}^{[I, J]}=\left(m_{(M+1)(i-1)+I,(M+1)(j-1)+J}^{(\mathbf{n},I, J)}\right)_{1 \leq i \leq N_{I,l}, 1 \leq j \leq N_{J,l}}, \quad 1 \leq I, J \leq M,
\end{eqnarray}
and the corresponding matrix elements are defined by
\begin{eqnarray}
m_{i, j}^{(\mathbf{n},I, J)}&=&\sum_{v=0}^{\min (i, j)}\left[\frac{\left|p_{1}\right|^{2}}{\left(p_{0}+p_{0}^{*}\right)^{2}}\right]^{v} S_{i-v}\left(\mathbf{y}_I^{+}(\mathbf{n})+v \mathbf{s} + \nu_0 \mathbf{s}\right) S_{j-v}\left(\mathbf{y}_J^{-}(\mathbf{n})+v \mathbf{s}^{*} + \nu_0 \mathbf{s}^{*}\right).
\end{eqnarray}

Next, with similar arguments to \cite{yang2023rogue,zhang2022rogue}, $\mathbf{y}_I^{\pm}(\mathbf{n})$ can be replaced by $\mathbf{x}_I^{\pm}(\mathbf{n})$, while $\nu_0 \mathbf{s}$ and $\nu_0 \mathbf{s}^{*}$ can be absorbed into $\mathbf{x}_I^{+}(\mathbf{n})$ and $\mathbf{x}_I^{-}(\mathbf{n})$, respectively. Then, the determinant in \eqref{pattern tau-block matrix} becomes a lower $\widehat{\mathcal{N}}_l$-th order rogue wave 
$$
v_{n, \widehat{\mathcal{N}}_l}(\bar{x}, t), \quad n=1,2,\dots,M,
$$
of the multi-component Hirota equation as given in Theorem \ref{thm:rogue wave solutions}, where $\widehat{\mathcal{N}}_l=(N_{1,l},N_{2,l},\dots,N_{M,l})$. The corresponding internal parameters $\{\hat{a}_{n,I}\}$ with 
$$n=(1,2,\dots,M,M+2,M+3,\dots,2M+1,2M+3,2M+4,\dots), $$ are related to the original internal parameters as
\begin{equation}
     \hat{a}_{j,I} = a_j, \quad\text{ for }\quad j \not \equiv 0  \pmod{M+1}, \quad I=1,2,\dots,M.
\end{equation}
Here, we use the fact $s_j=0$ when $j \mod (M+1) \not =0$ in Lemma \ref{lemma:expansion of s_r}. Moreover, when $j \mod (M+1) =0$, the term $\hat{a}_{j,I}$ can be eliminated by performing appropriate row and column operations in the same manner as in \cite{yang2023rogue}.

This completes the proof of Theorem \ref{thm:Rogue wave patterns}.

\bibliographystyle{abbrvnat}
\bibliography{references.bib}

@article{zhang2022rogue,
  title={{Rogue waves and their patterns in the vector nonlinear {S}chr{\"o}dinger equation}},
  author={Zhang, Guangxiong and Huang, Peng and Feng, Bao-Feng and Wu, Chengfa},
  journal={J. Nonlinear Sci.},
  volume={33},
  number={6},
  pages={116},
  year={2023},
  publisher={Springer}
}

@article{lin2024rogue,
  title={{Rogue wave pattern of multi-component derivative nonlinear {S}chr{\"o}dinger equations}},
  author={Lin, Huian and Ling, Liming},
  journal={Chaos},
  volume={34},
  number={4},
  pages={043126},
  year={2024},
  publisher={AIP Publishing}
}

@book{hirota2004direct,
  title={{The direct method in soliton theory}},
  author={Hirota, Ryogo},
  number={155},
  year={2004},
  publisher={Cambridge University Press}
}

@article{ohta2012general,
  title={{General high-order rogue waves and their dynamics in the nonlinear {S}chr{\"o}dinger equation}},
  author={Ohta, Yasuhiro and Yang, Jianke},
  journal={Proc. R. Soc. A},
  volume={468},
  number={2142},
  pages={1716-1740},
  year={2012},
  publisher={The Royal Society Publishing}
}

@article{Yang2023Yang,
  title={{Partial-rogue waves that come from nowhere but leave with a trace in the Sasa-Satsuma equation}},
  author={Yang, Bo and Yang, Jianke},
  journal={Phys. Lett. A},
  volume={458},
  pages={128573},
  year={2023},
  publisher={Elsevier}
}

@article{yang2023rogue,
  title={{Rogue wave patterns associated with Okamoto polynomial hierarchies}},
  author={Yang, Bo and Yang, Jianke},
  journal={Stud. Appl. Math.},
  volume={151},
  number={1},
  pages={60--115},
  year={2023},
  publisher={Wiley Online Library}
}

@article{romdhane2008zeros,
  title={{On the zeros of d-symmetric d-orthogonal polynomials}},
  author={Romdhane, N Ben},
  journal={J. Math. Anal. Appl.},
  volume={344},
  number={2},
  pages={888-897},
  year={2008},
  publisher={Elsevier}
}

@article{garcia2015oscillation,
  title={{Oscillation theorems for the Wronskian of an arbitrary sequence of eigenfunctions of Schr{\"o}dinger’s equation}},
  author={Garc{\'\i}a-Ferrero, M{\textordfeminine}{\'A}ngeles and G{\'o}mez-Ullate, David},
  journal={Lett. Math. Phys.},
  volume={105},
  number={4},
  pages={551-573},
  year={2015},
  publisher={Springer}
}

@article{kedziora2011circular,
  title={{Circular rogue wave clusters}},
  author={Kedziora, David J and Ankiewicz, Adrian and Akhmediev, Nail},
  journal={Phys. Rev. E},
  volume={84},
  number={5},
  pages={056611},
  year={2011},
  publisher={APS}
}

@article{he2013generating,
  title={{Generating mechanism for higher-order rogue waves}},
  author={He, JS and Zhang, HR and Wang, LH and Porsezian, K and Fokas, AS},
  journal={Phys. Rev. E},
  volume={87},
  number={5},
  pages={052914},
  year={2013},
  publisher={APS}
}

@article{fukutani2000special,
  title={{Special polynomials and the Hirota bilinear relations of the second and the fourth Painlev{\'e} equations}},
  author={Fukutani, Satoshi and Okamoto, Kazuo and Umemura, Hiroshi},
  journal={Nagoya Math. J.},
  volume={159},
  pages={179--200},
  year={2000},
  publisher={Cambridge University Press}
}

@article{kametaka1983poles,
  title={{On poles of the rational solution of the Toda equation of Painlev{\'e}-IV type}},
  author={Kametaka, Yoshinori},
  journal={Proc. Jpn. Acad. A},
  volume  = {59},
  pages   = {453--455},
  year={1983}
}

@article{clarkson2003second,
  title={{The second Painlev{\'e} equation, its hierarchy and associated special polynomials}},
  author={Clarkson, Peter A and Mansfield, Elizabeth L},
  journal={Nonlinearity},
  volume={16},
  number={3},
  pages={R1-R26},
  year={2003}
}

@book{james2006representation,
  title={{The representation theory of the symmetric groups}},
  author={James, Gordon Douglas},
  year={2006},
  publisher={Springer}
}

@article{felder2012zeros,
  title={{Zeros of Wronskians of Hermite polynomials and Young diagrams}},
  author={Felder, Giovanni and Hemery, AD and Veselov, AP},
  journal={Physica D},
  volume={241},
  number={23-24},
  pages={2131-2137},
  year={2012},
  publisher={Elsevier}
}

@article{bonneux2020coefficients,
  title={{Coefficients of Wronskian Hermite polynomials}},
  author={Bonneux, Niels and Dunning, Clare and Stevens, Marco},
  journal={Stud. Appl. Math.},
  volume={144},
  number={3},
  pages={245-288},
  year={2020},
  publisher={Wiley Online Library}
}

@article{gramain2012core,
  title={On core and bar-core partitions},
  author={Gramain, Jean-Baptiste and Nath, Rishi},
  journal={Ramanujan J.},
  volume={27},
  number={2},
  pages={229-233},
  year={2012},
  publisher={Springer}
}

@article{olsson2009core,
  title={{Core partitions and block coverings}},
  author={Olsson, J{\o}rn B},
  journal={Proc. Amer. Math. Soc.},
  pages={2943-2951},
  volume={137},
  number={9},
  year={2009},
  publisher={JSTOR}
}

@article{oblomkov1999monodromy,
  title={{Monodromy-free Schr{\"o}dinger operators with quadratically increasing potentials}},
  author={Oblomkov, Aleksei Anatolyevich},
  journal={Theor. Math. Phys.},
  volume={121},
  number={3},
  pages={1574-1584},
  year={1999},
  publisher={Springer}
}

@article{bonneux2018recurrence,
  title={{Recurrence relations for Wronskian Hermite polynomials}},
  author={Bonneux, Niels and Stevens, Marco},
  journal={SIGMA},
  volume={14},
  pages={048},
  year={2018}
}

@article{grosu2021irreducibility,
  title={{The irreducibility of some Wronskian Hermite polynomials}},
  author={Grosu, Codru{\c{t}} and Grosu, Corina},
  journal={Indag. Math.},
  volume={32},
  number={2},
  pages={456-497},
  year={2021},
  publisher={Elsevier}
}

@article{yang2021rogue,
  title={{Rogue wave patterns in the nonlinear Schr{\"o}dinger equation}},
  author={Yang, Bo and Yang, Jianke},
  journal={Physica D},
  volume={419},
  pages={132850},
  year={2021},
  publisher={Elsevier}
}

@article{roffelsen2025real,
  title={{On real and imaginary roots of generalised Okamoto polynomials}},
  author={Roffelsen, Pieter and Stokes, Alexander},
  journal={J. Lond. Math. Soc.},
  volume={112},
  number={4},
  pages={e70329},
  year={2025},
  publisher={Wiley Online Library}
}

@article{shi2026soliton2,
  title={{Soliton solutions to the coupled Sasa-Satsuma-mKdV equation}},
  author={Shi, Changyan and Feng, Bao-Feng},
  journal={arXiv:2603.18207},
  year={2026}
}

@article{GouldHopper1962,
  author  = {Gould, H. W. and Hopper, A. T.},
  title   = {{Operational formulas connected with two generalizations of Hermite polynomials}},
  journal = {Duke Math. J.},
  volume  = {29},
  number  = {1},
  pages   = {51--63},
  year    = {1962}
}

@article{yang2021universal,
  title={{Universal rogue wave patterns associated with the Yablonskii-Vorob’ev polynomial hierarchy}},
  author={Yang, Bo and Yang, Jianke},
  journal={Physica D},
  volume={425},
  pages={132958},
  year={2021},
  publisher={Elsevier}
}

@article{solli2007optical,
  title={{Optical rogue waves}},
  author={Solli, Daniel R and Ropers, Claus and Koonath, Prakash and Jalali, Bahram},
  journal={Nature},
  volume={450},
  number={7172},
  pages={1054-1057},
  year={2007},
  publisher={Nature Publishing Group UK London}
}

@article{clarkson2003fourth,
  title={{The fourth Painlev{\'e} equation and associated special polynomials}},
  author={Clarkson, Peter A},
  journal={J. Math. Phys.},
  volume={44},
  number={11},
  pages={5350--5374},
  year={2003},
  publisher={American Institute of Physics}
}

@article{noumi1999symmetries,
  title={{Symmetries in the fourth Painlev{\'e} equation and Okamoto polynomials}},
  author={Noumi, Masatoshi and Yamada, Yasuhiko},
  journal={Nagoya Math. J.},
  volume={153},
  pages={53--86},
  year={1999},
  publisher={Cambridge University Press}
}

@article{zhang2025dark,
  title={{Dark soliton and breather solutions to the coupled Sasa-Satsuma equation}},
  author={Zhang, Guangxiong and Shi, Changyan and Wu, Chengfa and Feng, Bao-Feng},
  journal={J. Nonlinear Sci.},
  volume={35},
  number={1},
  pages={7},
  year={2025},
  publisher={Springer}
}

@article{liu2023painleve,
  author  = {N. Liu and Z.-Z. Lan and J.-D. Yu},
  title   = {{Painlev{\'e} asymptotics for the coupled Sasa--Satsuma equation}},
  journal = {Proc. Amer. Math. Soc.},
  volume  = {151},
  number  = {9},
  pages   = {3763--3778},
  year    = {2023}
}

@book{macdonald1998symmetric,
  title={{Symmetric functions and Hall polynomials}},
  author={Macdonald, Ian Grant},
  year={1998},
  publisher={Oxford University Press}
}

@article{bertola2015zeros,
  title={{Zeros of large degree Vorob’ev--Yablonski polynomials via a Hankel determinant identity}},
  author={Bertola, Marco and Bothner, Thomas},
  journal={Int. Math. Res. Not. IMRN},
  volume={2015},
  number={19},
  pages={9330--9399},
  year={2015},
  publisher={Oxford University Press}
}

@article{buckingham2020large,
  title={{Large-degree asymptotics of rational Painlev{\'e}-IV functions associated to generalized Hermite polynomials}},
  author={Buckingham, Robert},
  journal={Int. Math. Res. Not. IMRN},
  volume={2020},
  number={18},
  pages={5534--5577},
  year={2020},
  publisher={Oxford University Press}
}

@article{balogh2016hankel,
  title={{Hankel determinant approach to generalized Vorob’ev-Yablonski polynomials and their roots}},
  author={Balogh, Ferenc and Bertola, Marco and Bothner, Thomas},
  journal={Constr. Approx.},
  volume={44},
  number={3},
  pages={417-453},
  year={2016},
  publisher={Springer}
}

@article{borcea2009polya,
  title={P{\'o}lya-{S}chur master theorems for circular domains and their boundaries},
  author={Borcea, Julius and Br{\"a}nd{\'e}n, Petter},
  journal={Ann. of Math.},
  pages={465--492},
  volume  = {170},
  number  = {1},
  year={2009},
  publisher={JSTOR}
}

@book{obreschkoff1963verteilung,
  title={Verteilung und berechnung der Nullstellen reeller Polynome},
  author={Obreschkoff, Nikola},
  year={1963},
  publisher={VEB Deutscher Verlag der Wissenschaften},
  address={Berlin}
}

@book{GantmacherKrein2002,
  author    = {Gantmacher, F. R. and Krein, M. G.},
  title     = {Oscillation Matrices and Kernels and Small Vibrations of Mechanical Systems},
  publisher = {AMS Chelsea Publishing},
  address   = {Providence, RI},
  year      = {2002},
  series    = {AMS Chelsea Publishing},
  volume    = {345}
}

@article{AlseidiMargaliotGarloff2019,
  author  = {Alseidi, Rola and Margaliot, Michael and Garloff, J{\"u}rgen},
  title   = {On the Spectral Properties of Nonsingular Matrices that Are Strictly Sign-Regular for Some Order with Applications to Totally Positive Discrete-Time Systems},
  journal = {J. Math. Anal. Appl.},
  volume  = {474},
  number  = {1},
  pages   = {524--543},
  year    = {2019}
}

@article{tsuda2005universal,
  author  = {Tsuda, Teruhisa},
  title   = {{Universal characters, integrable chains and the Painlev{\'e} equations}},
  journal = {Adv. Math.},
  volume  = {197},
  number  = {2},
  pages   = {587--606},
  year    = {2005}
}

@article{gomez2021complete,
  author  = {G{\'o}mez-Ullate, David and Grandati, Yves and Milson, Robert},
  title   = {Complete Classification of Rational Solutions of {$A_{2n}$}-{P}ainlev{\'e} Systems},
  journal = {Adv. Math.},
  volume  = {385},
  pages   = {107770},
  year    = {2021}
}

@article{VanIseghem1987, 
author = {Van Iseghem, J.}, 
title = {Vector orthogonal relations. {V}ector {QD}-algorithm},
journal = {J. Comput. Appl. Math.}, 
volume = {19}, 
number = {1}, 
pages = {141--150}, 
year = {1987} }

@article{Maroni1989, 
author = {Maroni, Pascal}, 
title = {L'orthogonalit{\'e} et les r{\'e}currences de polyn{\^o}mes d'ordre sup{\'e}rieur {\`a} deux}, 
journal = {Annales de la Facult{\'e} des Sciences de Toulouse. Math{\'e}matiques}, 
series = {5}, 
volume = {10}, 
number = {1}, 
pages = {105--139}, 
year = {1989}
}

@article{DouakMaroni1992,
  author  = {Douak, K. and Maroni, P.},
  title   = {Les polyn{\^o}mes orthogonaux ``classiques'' de dimension deux},
  journal = {Analysis},
  volume  = {12},
  pages   = {71--107},
  year    = {1992}
}

@article{douak1995relation,
  author    = {Douak, Khalfa},
  title     = {The relation of the d-orthogonal polynomials to the {Appell} polynomials},
  journal   = {J. Comput. Appl. Math.},
  volume    = {70},
  number    = {2},
  pages     = {279--295},
  year      = {1995}
}

@book{noumi2004painleve,
  title={Painlev{\'e} equations through symmetry},
  author={Noumi, Masatoshi},
  volume={223},
  year={2004},
  publisher={American Mathematical Society}
}

@article{clarkson2024rational,
  author  = {Clarkson, Peter A. and Dunning, Clare},
  title   = {Rational solutions of the fifth {Painlev{\'e}} equation. {G}eneralized {Laguerre} polynomials},
  journal = {Stud. Appl. Math.},
  year    = {2024},
  volume  = {152},
  number  = {1},
  pages   = {453--507}
}

@article{masuda2002determinant,
  author  = {Masuda, Tetsu and Ohta, Yasuhiro and Kajiwara, Kenji},
  title   = {A determinant formula for a class of rational solutions of {Painlev{\'e} V} equation},
  journal = {Nagoya Math. J.},
  volume  = {168},
  pages   = {1--25},
  year    = {2002}
}

@article{Yablonskii1959RationalPII,
  author  = {Yablonskii, A. I.},
  title   = {On rational solutions of the second {Painlev{\'e}} equation},
  journal = {Vesti Akad. Navuk BSSR Ser. Fiz. Tkh. Nauk},
  volume  = {3},
  pages   = {30--35},
  year    = {1959}
}

@article{Vorobev1965RationalPII,
  author  = {Vorob'ev, A. P.},
  title   = {On the rational solutions of the second {Painlev{\'e}} equation},
  journal = {Differ. Equ.},
  volume  = {1},
  pages   = {58--59},
  year    = {1965}
}

@article{Okamoto1986StudiesPainleveIII,
  author  = {Okamoto, Kazuo},
  title   = {Studies on the {Painlev{\'e}} equations. {III}. {Second} and fourth {Painlev{\'e}} equations, {$P_{\mathrm{II}}$} and {$P_{\mathrm{IV}}$}},
  journal = {Math. Ann.},
  volume  = {275},
  pages   = {221--255},
  year    = {1986}
}

@article{KajiwaraOhta1996PII,
  author  = {Kajiwara, K. and Ohta, Y.},
  title   = {Determinantal structure of the rational solutions for the {Painlev{\'e}} {II} equation},
  journal = {J. Math. Phys.},
  volume  = {37},
  pages   = {4693--4704},
  year    = {1996}
}

@article{KajiwaraOhta1998PIV,
  author  = {Kajiwara, K. and Ohta, Y.},
  title   = {Determinant structure of the rational solutions for the {Painlev{\'e}} {IV} equation},
  journal = {J. Phys. A},
  volume  = {31},
  pages   = {2431--2446},
  year    = {1998}
}

@article{BuckinghamMiller2014PII,
  author  = {Buckingham, Robert J. and Miller, Peter D.},
  title   = {Large-degree asymptotics of rational {Painlev{\'e}}-{II} functions: noncritical behaviour},
  journal = {Nonlinearity},
  volume  = {27},
  number  = {10},
  pages   = {2489--2578},
  year    = {2014}
}

@article{BalogounBertola2025PVHankel,
  author  = {Balogoun, Malik and Bertola, Marco},
  title   = {Rational solutions of {Painlev{\'e}} {V} from {Hankel} determinants and the asymptotics of their pole locations},
  journal = {SIGMA},
  volume  = {21},
  pages   = {097},
  year    = {2025}
}

@article{BuckinghamMiller2015PIICritical,
  author  = {Buckingham, Robert J. and Miller, Peter D.},
  title   = {Large-degree asymptotics of rational {Painlev{\'e}}-{II} functions: critical behaviour},
  journal = {Nonlinearity},
  volume  = {28},
  number  = {6},
  pages   = {1539--1596},
  year    = {2015}
}

@article{BuckinghamMiller2022PIVIsomonodromy,
  author  = {Buckingham, Robert J. and Miller, Peter D.},
  title   = {Large-degree asymptotics of rational {Painlev{\'e}}-{IV} solutions by the isomonodromy method},
  journal = {Constr. Approx.},
  volume  = {56},
  number  = {2},
  pages   = {233--443},
  year    = {2022}
}

@article{BothnerMiller2020PIIIAsymptotics,
  author  = {Bothner, Thomas and Miller, Peter D.},
  title   = {Rational solutions of the {Painlev{\'e}}-{III} equation: large parameter asymptotics},
  journal = {Constr. Approx.},
  volume  = {51},
  pages   = {123-224},
  year    = {2020}
}

@article{noumi1998affine,
  author  = {Noumi, Masatoshi and Yamada, Yasuhiko},
  title   = {Affine {Weyl} groups, discrete dynamical systems and
             {Painlev{\'e}} equations},
  journal = {Commun. Math. Phys.},
  volume  = {199},
  number  = {2},
  pages   = {281--295},
  year    = {1998}
}

@article{feigin2021quasi,
  author  = {Feigin, Misha V and Halln{\"a}s, Martin A and
             Veselov, Alexander P},
  title   = {Quasi-invariant {Hermite} polynomials and the
             {Lassalle--Nekrasov} correspondence},
  journal = {Commun. Math. Phys.},
  volume  = {386},
  number  = {1},
  pages   = {107--141},
  year    = {2021}
}

@article{vignat2013proof,
  author  = {Vignat, Christophe and L{\'e}v{\^e}que, Olivier},
  title   = {Proof of a conjecture by {G}azeau et al. using the {Gould--Hopper} polynomials},
  journal = {J. Math. Phys.},
  volume  = {54},
  number  = {7},
  pages   = {073513},
  year    = {2013}
}

@article{chang2011gould,
  author  = {Chang, Jen-Hsu},
  title   = {The {Gould--Hopper} polynomials in the
             {Novikov--Veselov} equation},
  journal = {J. Math. Phys.},
  volume  = {52},
  number  = {9},
  pages   = {092703},
  year    = {2011}
}

@article{felder2026harmonic,
  author  = {Felder, Giovanni and Veselov, Alexander P and
             Nekrasov, Nikita},
  title   = {Harmonic locus and {Calogero--Moser} spaces},
  journal = {Commun. Math. Phys.},
  volume  = {407},
  number  = {2},
  pages   = {22},
  year    = {2026}
}

@article{bonneux2020asymptotic,
  author  = {Bonneux, Niels},
  title   = {Asymptotic behavior of {Wronskian} polynomials that are
             factorized via {$p$}-cores and {$p$}-quotients},
  journal = {Math. Phys. Anal. Geom.},
  volume  = {23},
  number  = {4},
  pages   = {36},
  year    = {2020}
}

@article{grosu2021expansion,
  author  = {Grosu, Codru{\c{t}} and Grosu, Corina},
  title   = {The expansion of {Wronskian Hermite} polynomials in the
             {Hermite} basis},
  journal = {SIGMA},
  volume  = {17},
  pages   = {003},
  year    = {2021}
}

@article{bertola2024exactly,
  author  = {Bertola, Marco and Chavez-Heredia, Eduardo and Grava, Tamara},
  title   = {Exactly solvable anharmonic oscillator, degenerate
             orthogonal polynomials and {Painlev{\'e} II}},
  journal = {Commun. Math. Phys.},
  volume  = {405},
  number  = {2},
  pages   = {52},
  year    = {2024}
}

@article{borcea2009lee,
  author  = {Borcea, Julius and Br{\"a}nd{\'e}n, Petter},
  title   = {The {Lee--Yang} and {P{\'o}lya--Schur} programs. {I}.   {Linear} operators preserving stability},
  journal = {Invent. Math.},
  volume  = {177},
  number  = {3},
  pages   = {541--569},
  year    = {2009}
}

@article{aptekarev1998multiple,
  author  = {Aptekarev, Alexander I.},
  title   = {Multiple orthogonal polynomials},
  journal = {J. Comput. Appl. Math.},
  volume  = {99},
  number  = {1--2},
  pages   = {423--447},
  year    = {1998}
}

@article{VanAsscheCoussement2001,
  author  = {Van Assche, Walter and Coussement, Els},
  title   = {Some classical multiple orthogonal polynomials},
  journal = {J. Comput. Appl. Math.},
  volume  = {127},
  number  = {1--2},
  pages   = {317--347},
  year    = {2001}
}

@inproceedings{kuijlaars2010multiple,
  author    = {Kuijlaars, Arno B. J.},
  title     = {Multiple orthogonal polynomials in random matrix theory},
  booktitle = {Proceedings of the International Congress of
               Mathematicians, Hyderabad, India, 2010},
  editor    = {Bhatia, Rajendra},
  volume    = {III},
  pages     = {1417--1432},
  publisher = {Hindustan Book Agency},
  address   = {New Delhi},
  year      = {2010}
}

@article{filipuk2015computing,
  author  = {Filipuk, Galina and Haneczok, Maciej and
             Van Assche, Walter},
  title   = {Computing recurrence coefficients of multiple orthogonal
             polynomials},
  journal = {Numer. Algorithms},
  volume  = {70},
  number  = {3},
  pages   = {519--543},
  year    = {2015}
}

@article{bleher2011random,
  author  = {Bleher, Pavel M. and Delvaux, Steven and
             Kuijlaars, Arno B. J.},
  title   = {Random Matrix Model with External Source and a Constrained
             Vector Equilibrium Problem},
  journal = {Comm. Pure Appl. Math.},
  volume  = {64},
  number  = {1},
  pages   = {116--160},
  year    = {2011}
}

@article{sokal2024multiple,
  author  = {Sokal, Alan D.},
  title   = {Multiple Orthogonal Polynomials, \(d\)-Orthogonal Polynomials,
             Production Matrices, and Branched Continued Fractions},
  journal = {Trans. Amer. Math. Soc. Ser. B},
  volume  = {11},
  pages   = {762--797},
  year    = {2024}
}
\end{document}